\documentclass{article}
\usepackage[scaled]{helvet}
\usepackage[utf8]{inputenc}
\usepackage{amsmath,amssymb,amsthm,mathrsfs,color,times,textcomp,sectsty,verbatim}
\usepackage{xcolor}
\usepackage[colorlinks=true]{hyperref}
\hypersetup{urlcolor=blue, citecolor=red, linkcolor=blue}
\usepackage[colorinlistoftodos]{todonotes}
\usepackage[toc,page]{appendix}

\usepackage{authblk}
\usepackage{titlesec}
\usepackage{lipsum}

 \usepackage{hyperref}
\hypersetup{
    colorlinks,
    citecolor=red,
    filecolor=orange,
    linkcolor=blue,
    urlcolor=purple
}

\usepackage[a4paper, total={6in, 8in}]{geometry}

\newcommand{\isom}{\cong} 
\newcommand{\union}{\cup}
\newcommand{\ints}{\cap}

\DeclareMathOperator{\End}{End}

\newcommand{\SO}{\mathrm{SO}}

\newcommand{\Spin}{\mathrm{Spin}}

\newcommand{\CC}{\mathbb{C}}

\newcommand{\HH}{\mathbb{H}}

\newcommand{\KK}{\mathbb{K}}

\newcommand{\OO}{\mathbb{O}}

\newcommand{\RR}{\mathbb{R}}

\newcommand{\ZZ}{\mathbb{Z}}

\newcommand{\cA}{\mathcal{A}}
\newcommand{\cB}{\mathcal{B}}

\newcommand{\cE}{\mathcal{E}}
\newcommand{\cF}{\mathcal{F}}
\newcommand{\cG}{\mathcal{G}}

\newcommand{\cJ}{\mathcal{J}}
\newcommand{\cK}{\mathcal{K}}
\newcommand{\cL}{\mathcal{L}}

\newcommand{\cN}{\mathcal{N}}

\newcommand{\cQ}{\mathcal{Q}}

\newcommand{\cU}{\mathcal{U}}
\newcommand{\cV}{\mathcal{V}}
\newcommand{\cW}{\mathcal{W}}
\newcommand{\cX}{\mathcal{X}}

\newcommand{\del}{\partial}

\newcommand{\vol}{\mathrm{Vol}}

\newcommand{\id}{\mathrm{Id}}
\newcommand{\grad}{\mathrm{grad}}
\newcommand{\hess}{\mathrm{Hess}}
\newcommand{\tr}{\mathrm{tr}}
\newcommand{\shadow}{\mathrm{Shadow}}
\newcommand{\Cay}{\mathcal{C}ay}
\newcommand{\Jac}{\mathrm{Jac}}

\newcommand{\FFJ}{\mathfrak{J}}

\newcommand{\FFX}{\mathfrak{X}}
\newcommand{\FFY}{\mathfrak{Y}}

\newcommand{\matii}[2]{\left(\begin{array}{cc} #1\\#2\end{array}\right)}
\newcommand{\matiii}[3]{\left(\begin{array}{ccc} #1\\#2\\#3\end{array}\right)}

\newcommand{\Addresses}{{
  \bigskip

\textsc{Department of Mathematics, University of Michigan, Ann Arbor, MI, 48109.} \par\nopagebreak 

\textit{E-mail address}: \texttt{\email{ruanyp@umich.edu}}.

}}

\newcommand*{\email}[1]{%
    \normalsize\href{mailto:#1}{#1}\par
    }

\newtheorem{thm}{Theorem}[section]

\newtheorem{prop}[thm]{Proposition}
\newtheorem{lem}[thm]{Lemma}
\newtheorem{cor}[thm]{Corollary}
\newtheorem{conj}[thm]{Conjecture}

\theoremstyle{definition}
\newtheorem{dfn}[thm]{Definition}

\newtheorem{notation}[thm]{Notation}
\newtheorem*{rmk}{Remark}
\title{Filling volume minimality and boundary rigidity of metrics close to a negatively curved symmetric metric}
\author{Yuping Ruan}
\date{}
\begin{document}
\maketitle
\begin{abstract}
This paper generalizes D. Burago and S. Ivanov's work \cite{Burago2} on filling volume minimality and boundary rigidity of almost real hyperbolic metrics. We show that regions with metrics close to a negatively curved symmetric metric are strict minimal fillings and hence boundary rigid. This includes perturbations of complex, quaternionic and Cayley hyperbolic metrics.
\end{abstract}
\tableofcontents

\section{Introduction}\label{s1}
\subsection{Boundary rigidity and filling minimality}\label{s1s1}
Let $M=(M^n,g)$ be a compact Riemannian manifold with boundary $\del M$. Its \emph{boundary distance function}, denoted by $\mathrm{bd}_M$, is the restriction of the Riemannian distance $d_M$ to $\del M\times \del M$. We say a Riemannian manifold is \emph{boundary rigid} if its metric is uniquely determined by its boundary distance function. More precisely, we have the following definition.
\begin{dfn}[Boundary rigidity]\label{bdry_rigid}
A compact Riemannian manifold $M$ (with boundary) is \emph{boundary rigid} if every Riemannian manifold $M'$ with $\del M'=\del M$ and $\mathrm{bd}_{M'}=\mathrm{bd}_M$ is isometric to $M$ via a boundary-preserving isometry. 
\end{dfn}
It is easy to construct manifolds which are \emph{not} boundary rigid. For example, if there exists some proper open subset which does not intersect any shortest path connecting boundary points, then the metric on this open subset does not affect the boundary distance function and hence such manifolds are not boundary rigid. In particular, ``large'' spherical caps, i.e. proper open subsets of a sphere $S^n\subset\RR^{n+1}$ bounded by a  hyperplane in $\RR^{n+1}$ which properly contain a hemisphere are not boundary rigid. Such manifolds must be avoided if one seeks boundary rigidity. Therefore it is reasonable to first consider simple manifolds, i.e. manifolds with strictly convex boundary such that every two points are connected by a unique geodesic segment and geodesics do not have conjugate points. In particular, we have the following conjecture by Michel.

\begin{conj}[Michel, \cite{Michel}]\label{Michel_conj}
All simple manifolds are boundary rigid.
\end{conj}

Here and below, by \emph{region} we mean a connected open set with a smooth boundary. 

A lot of progress has been made toward boundary rigidity. Pestov and Uhlmann \cite{Pestov_Uhlmann} proved the above conjecture in dimension 2. In higher dimensions, regions in $\RR^n$ (Besikovitch \cite{Besikovitch}; Gromov \cite{Gromov}), in the open hemisphere $S^n_+$ (Michel \cite{Michel}) and in rank-$1$ symmetric spaces of non-compact type (following the volume entropy rigidity theorem by Besson, Courtois and Gallot \cite{BCG1}) are known to be boundary rigid. Burago and Ivanov proved boundary rigidity for almost Euclidean (\cite{Burago1}) and almost real hyperbolic (\cite{Burago2}) regions. Recently a very general result by Stefanov, Uhlmann and Vasy in \cite[Corollary 1.2]{SUV} showed that a simple manifold $(M,g)$ is boundary rigid if it satisfies any of the following conditions: 
\begin{enumerate}
\item[(1).] $(M,g)$ has non-positive sectional curvature;
\item[(2).] $(M,g)$ has non-negative sectional curvature;
\item[(3).] $(M,g)$ has no focal points.
\end{enumerate}
We refer the readers to \cite{Croke2}, \cite{Ivanov3} and \cite{SU} for a survey on boundary rigidity.

\begin{dfn}[Filling minimality]\label{minimal_filling}
A compact Riemannian manifold $M$ with boundary $\del M$ is a \emph{minimal filling} if, for every compact Riemannian manifold $M'$ with $\del M'=\del M$, the inequality
$$d_{M'}(x,y)\geq d_M(x,y)\quad\forall x,y\in\del M$$
implies that 
$$\vol(M')\geq\vol(M).$$
We say that $M$ is a \emph{strict minimal filling} if, in addition, the equality
$$\vol(M')=\vol(M)$$
holds only when $M$ and $M'$ are isometric via an isometry which fixes all boundary points.
\end{dfn}
\begin{rmk}
The idea of filling Riemannian manifolds was introduced by Gromov in \cite{Gromov}. If $M=(M,g)$ has a connected boundary with dimension $\geq 2$ and is a minimal filling, then by \cite[2.2A Proposition]{Gromov} $\vol(M)$ is called the \emph{filling volume} of the boundary $(\del M, \mathrm{bd}_M)$ equipped with boundary distance function denoted by $\mathrm{FillVol}(\del M, \mathrm{bd}_{(M,g)})$. 
See \cite{Gromov} for a detailed discussion.
\end{rmk} 

Similar to boundary rigidity, not all manifolds are minimal fillings. For example, ``large spherical caps'' fail to be minimal fillings because they have larger volume than the hemisphere with the same boundary and boundary distance function. In \cite{Burago1} and \cite{Burago2}, Burago and Ivanov made the following conjecture.

\begin{conj}[Burago-Ivanov, \cite{Burago1, Burago2}]\label{filling_conj}
Every simple manifold is a strict minimal filling.
\end{conj}

When $M$ is simple, its volume is uniquely determined by its boundary distance function due to Santal\'o's formula \cite{Santalo}. Moreover, if $M'$ shares the same boundary distance function with $M$, $M'$ also has to be simple due to strict triangular inequality and smoothness of $\mathrm{bd}_M$.Therefore boundary rigidity is a direct corollary of filling minimality when the manifold is simple. A similar argument also works for \emph{strong geodesically minimizing} (SGM) manifolds (See \cite[1. Preliminaries]{Croke1} for the definition of the SGM condition and \cite[Lemma 5.1]{Croke1} for details of this argument.) which allows non-convex boundaries. In particular, compact regions with a smooth boundary inside a simply connected negatively curved manifold satisfies the SGM condition.
 
Unlike boundary rigidity, little is understood about filling minimality. For example, Gromov's sphere filling conjecture (\cite[pp. 13]{Gromov}) asks whether a hemisphere of dimension $n+1$ is a minimal filling for its boundary. This is still open with partial results proved by Gromov \cite[pp. 59]{Gromov} and Bangert-Croke-Ivanov-Katz \cite[Corollary 1.8]{BCIK}. Croke, Dairbekov and Sharafutdinov proved ``local filling minimality'' in \cite[Proposition 1.2]{CDS} for simple manifolds with ``limited positive curvature along geodesics''. Local filling minimality here refers to the case when $M'$ and $M$ (See Definition \ref{minimal_filling}) have the same underlying manifold and very close metrics. Burago and Ivanov in \cite{Burago1} and \cite{Burago2} proved strict filling minimality for almost Euclidean and almost hyperbolic regions (when $M$ is almost Euclidean and almost hyperbolic and $M'$ is arbitary). 

The filling minimality problem can be more difficult than the boundary rigidity problem. This is because having the same boundary distance function provides more information than having a larger boundary distance function due to \cite[Lemma 5.1]{Croke1}. In the case when $M$ is a simple manifold. we denote by $\del_-T^1M$ the collection of unit vectors on $\del M$ pointing inside $M$ and $\del_+T^1M$ the collection of unit vectors on $\del M$ pointing outside $M$. Since any maximally extended geodesic in a simple manifold $M$ intersects the boundary transversely and has finite length, there is a one-one correspondence between maximally extended geodesics in $(M,g)$ and the corresponding triples $(v,w,l)\in \del_-T^1M\times \del_+T^1M\times \RR_+$ recording their initial vectors, exit vectors and lengths. Recall that geodesics are length minimizing and do not admit conjugate points when $M$ is simple, we have the following one-one correspondence.
$$\{(p,q)|p\neq q\in \del M\}\leftrightarrow\left\{\begin{array}{l}~~(\pi_-(\dot\gamma(0)),\pi_+(\dot\gamma(l)),l) \\
\in(T\del M)^2\times \RR_+\end{array}\left|\begin{array}{l}\gamma:[0,l]\to M \mathrm{~a~maximally} \\
\mathrm{extended~unit~speed~geodesic;}\\
\pi_\pm:\del_\pm T^1M\to T\del M \\
\mathrm{orthogonal~projection.}\end{array}\right.\right\}:=\mathrm{Lens}(M).$$
We call $\mathrm{Lens}(M)$ the \emph{lens data} of $M$. Similar to the boundary rigidity problem, we have the \emph{lens rigidity} problem which asks whether lens data can determine the manifold up to an isometry. If $M'$ and a simple manifold $M$ have the same boundary and the same boundary distance function, \cite[Lemma 5.1]{Croke1} implies that $\mathrm{Lens(M)}$ and $\mathrm{Lens}(M')$ are canonically identified. Therefore proving boundary rigidity in this case is the same as proving lens rigidity. This observation has been used in many results on boundary rigidity (for example \cite{SUV}.). Unfortunately, having a larger boundary distance function can mess up the lens data, especially the part of data recording initial and exit vectors. (In fact, if we neglect the length part from the lens data, we can still study the corresponding \emph{scattering rigidity} problem. See for example \cite{BGuJ}.) Hence methods using lens data cannot directly apply in a similar way when working on the filling minimality problem.

\subsection{Statements of main results}\label{s1s2}
In this paper, we will generalize Burago and Ivanov's work in \cite{Burago1} and \cite{Burago2} to metrics close to a negatively curved symmetric metric. Notice that negatively curved symmetric metrics come from rank-1 symmetric spaces of non-compact type, we only need to consider metric perturbations of regions in real, complex, quaternionic and Cayley hyperbolic spaces. Let $\KK$ be one of the following 
\begin{enumerate}
\item[(i).]$\CC$, the field of complex numbers;
\item[(ii).]$\HH$, the algebra of all quaternions;
\item[(iii).]$\OO$, the algebra of all octonions.
\end{enumerate}
Since the real hyperbolic case has been proved in \cite[Theorem 1.6]{Burago2}, it remains for us to prove the following theorem.

\begin{thm}\label{main_thm_alt}
Let $M=\KK\mathbf{H}^n\neq\RR\mathbf{H}^n$ be a negatively curved symmetric space, i.e. $M=\CC\mathbf{H}^n$, $\HH\mathbf{H}^n$ or $\OO\mathbf{H}^2$. For any compact region $D\subset M$ (not necessarily simple) with a smooth boundary, there is a $C^r$-neighborhood (for a suitable r) of the symmetric metric on $D$ such that, for every metric $g$ from this neighborhood, the Riemannian manifold $(D,g)$ is a strict minimal filling.
\end{thm}
\begin{rmk}
The condition on $g$ being sufficiently close to the symmetric metric appears in several statements throughout this paper. This is given by several complicated constraints on $g$ related to the diameter of the region (with respect to the symmetric metric) and the dimension of the manifold. In particular we assume that $g$ is negatively curved so that $(M,g)$ satisfies the SGM condition. We do not track the number of derivatives required for our arguments to work but we will summarize all constraints on $g$ at the end of this paper. 

\end{rmk}

Since the SGM condition and strict filling minimality imply boundary rigidity as explained right after Conjecture \ref{filling_conj}, we have the following direct corollary of Theorem \ref{main_thm_alt}.

\begin{cor}\label{main_thm}
Under the same assumption as in Theorem \ref{main_thm_alt}, $(D,g)$ is boundary rigid. 
\end{cor}

\subsection{Plan of the proof}
This paper is a generalization of \cite{Burago2}. While we employ the same general constructions, various proofs become much more complicated for other rank-$1$ symmetric spaces. Therefore we need to improve some technical facts to prove filling minimality and boundary rigidity of almost rank-$1$ symmetric metrics. In Sections \ref{s2}-\ref{s4}, we briefly recall  \cite[Sections 2-4]{Burago2}: Let $(D,g)$ be as in Theorem \ref{main_thm_alt}. We can then extend $g$ smoothly to a metric on $\KK\mathbf{H}^n$ which coincides with the symmetric metric on $\KK\mathbf{H}^n$ (also denoted by $g$ for simplicity) outside a compact neighborhood of $D$.
\begin{enumerate}
\item[(1).] In Section \ref{s2}, we give the proof for Theorem \ref{main_thm_alt}. This uses the following tools discussed in the rest of the paper, namely two maps and a notion of Riemannian structure on an open subset of some $L^\infty$ space.  The first map is a distance-preserving map $\Phi:M\to\cL:=L^\infty(S)$, where $M=(\KK\mathbf{H}^n,g)$ and $S$ is a suitable measure space. (Following the traditions of geometric measure theory, we refer to Lipschitz maps from manifolds of any dimension to a normed space as \emph{surfaces}). Later in Section \ref{s4}, we choose $S$ to be the visual boundary of $M$ with a visual measure and $\Phi$ to be the Busemann function with respect to a fixed point in the interior as in \cite{Burago2}. The other map is a ``projection'' map $P_\sigma$ from a suitable neighborhood of $\Phi(M)\subset \cL$ to $M$ in the sense that $P_\sigma\circ\Phi=\id_M$. We also assume that $P_\sigma$ precomposing any $1$-Lipschitz map $f$ from a $dn$-dimensional manifold does not increase volumes. Moerover, we require that $P_\sigma\circ f$ preserves volumes if and only if the image of $P_\sigma\circ f$ is contained in $\Phi(M)$. If $(D',g')$ shares the same boundary with $(D,g)$ with larger boundary distance function, then we can extend $\Phi|_{\del D}$ to a 1-Lipschitz map $\Phi':(D',g')\to \cL$. Therefore strict filling minimality of $(D,g)$ follows from the above properties of $P_\sigma$.
\item[(2).] In Section \ref{s3}, we introduce the aforementioned ``Riemannian structure'' on a suitable open neighborhood $\cU$ of $\Phi(M)\subset\cL$ so that we can define $dn$-dimensional ($d=2,4,8$ when $\KK=\CC,\HH,\OO$ respectively) Riemannian volumes on $\cU$ and Jacobians for maps involving $\cU$. Under this ``Riemannian structure'', $\Phi$ is volume preserving and any $1$-Lipschitz map from a $dn$-dimensional manifold to $\cU$ is volume non-increasing. This construction is given by introducing a Riemannian metric on $\cU$ satisfying some natural conditions. With the help of this ``Riemannian structure'', it remains for us to find a $P_\sigma$ whose $dn$-dimensional Jacobian is smaller than $1$ on $\cU\setminus\Phi(M)$.

\item[(3).] Section \ref{s4} introduces the construction of $\Phi$ and a ``locally orthogonal projection'' $P:\cU\to M$ defined via a barycenter construction. The aformentioned ``projection'' map $P_\sigma$ will eventually be a small perturbation of this $P$ and therefore we expect $P$ to be almost volume non-increasing. For any $\phi$ in the domain of $P$, a direct computation decomposes the derivative map $d_\phi P$ into two different linear operators $A^{-1}_\phi:T_{P(\phi)}M\to T_{P(\phi)}M$ and $E_\phi:T_\phi\cL\to T_{P(\phi)}M$. Then it remains for us to study the Jacobians of the above two maps. 
\end{enumerate}

In the cases of real hyperbolic spaces discussed in \cite{Burago2}, the operator $A_\phi=\id$. This follows from the formula for the Hessian of the Busemann functions and sectional curvatures being constantly $-1$ (up to scaling). The main technical difficulty in our paper comes from the fact that other non-compact rank-1 symmetric spaces have sectional curvatures ranging from $-4$ to $-1$ (see \cite{BH99}). This leaves $A_\phi$ much more complicated even in the model case when $g$ is symmetric. We resolved this difficulty from the observation that $A_\phi$ and $E_\phi$ have closely related matrix expressions under a suitable choice of basis. The main difference of our proof compared to \cite{Burago2} comes from the remaining technical part of the paper.

\begin{enumerate}
\item[(4).] In Section \ref{s5}, we introduce some notion of ``almost rank-1 structure'' (depending smoothly on $g$) to construct suitable bases convenient for further computations. Then we construct an operator $\widehat{A}_\phi$ to approximate $A_\phi$ by using the data from the Hessian of the Busemann functions in symmetric spaces under these bases. (In the cases of real hyperbolic spaces, $\widehat{A}_\phi=\id$) Moreover, computation shows that $\widehat{A}_\phi$ and $E_\phi$ have matrix representations closely related to a positive definite matrix $Q_\phi$ (see \eqref{big_fat_matrix}). Hence the study of eigenvalues of $Q_\phi$ leads to the desired Jacobian and norm estimates. In particular, we proved that the $dn$-dimensional Jacobian of $\widehat{A}^{-1}_\phi\circ E_\phi$ is bounded above by $1$. This implies that the $dn$-dimensional Jacobian of $P$ is bounded above by $1$ plus some error term. In other words, $P$ is almost volume non-increasing. The Cayley hyperbolic case turns out to be technically more difficult due to the non-associativity of octonionic mulitiplication. See Appendix \ref{appB} for a brief introduction to the Cayley hyperbolic space.
\item[(5).] Section \ref{s6} is the most technical part of this paper. The main goal of this section is to provide a detailed estimate on the error terms $A_\phi-\widehat{A}_\phi$ and $1-\Jac(\widehat{A}^{-1}_\phi\circ E_\phi)$ so that we can construct $P_\sigma$ as a perturbation of $P$ which decreases $dn$-dimensional volumes. A similar version of the statements and proofs introduced in this section also applies to the real hyperbolic cases.
\item[(6).] Section \ref{s7} is similar to \cite[Section 7, a compression trick]{Burago2} which constructs the aforementioned $P_\sigma$ and verifies the required properties introduced in Section \ref{s2} (see Proposition \ref{key_prop}).
\end{enumerate}
\textbf{Acknowledgements}: I would like to heartily thank my advisor Ralf Spatzier for his support during the entire work. I am also very grateful to Chris Connell for helpful and thorough discussions on this subject.

\section{Proof of the theorems}\label{s2}
This section is a review of \cite[Section 2. Proof of the theorems]{Burago2}. The purpose of this section is to prove Theorem \ref{main_thm_alt} and Corollary \ref{main_thm} from Proposition \ref{key_prop}, which is the $\KK\mathbf{H}^n$ version of \cite[Proposition 2.1]{Burago2}. This proposition asserts that we can ``embed'' $\KK\mathbf{H}^n$ into $L^\infty(\del_\infty M)=L^\infty(S^{dn-1})$ ($d=2,4,8$ when $\KK=\CC,\HH,\OO$, respectively) and also ``project'' an open subset of $L^\infty(S^{dn-1})$ to $\KK\mathbf{H}^n$ with certain properties. The rest of the paper provides the technical details of these two maps along with verification of properties.

Let $g_0$ denote the standard metric on $\KK\mathbf{H}^n$ such that sectional curvature on $\KK\mathbf{H}^n$ lies in the interval $[-4,-1]$. Let $D\subset \KK\mathbf{H}^n$ be a region with smooth boundary. Let $g$ be a Riemannian metric on $D$ which is $C^r $-sufficiently close to $g_0|_D$ for a suitable $r$. (See the remark after Proposition \ref{final_prop}.)

Fix a point $x_0\in\KK \mathbf{H}^n$. Let $B_{x_0}(R)$ be the ball of radius $R$ in $\KK \mathbf{H}^n$ centered at $x_0$ with respect to the symmetric Riemannian metric $g_{0}$. Fix an $R>0$ such that $D\subset B_{x_0}(R/5)$. By \cite[Theorem 2.3.6]{Hormander98} the metric $g$ can be smoothly extended from $D$ to $\KK\mathbf{H}^n$ which coincides with $g_0$ outside $B_{x_0}(R/2)$. Moreover, the extension can be constructed in such a way that it converges to $g_0$ as $g$ converges to $g_0|_D$.

We denote the extension by the same letter $g$ and let $M=(\KK\mathbf{H}^n,g)$. We also assume in addition that $g$ is sufficiently close to $g_0$ such that $B_{x_0}(R/5)$ is contained in the ball of radius $R/4$ centered at $x_0$ with respect to the metric $g$. Our goal is to prove that for any region $D\subset B_{x_0}(R/5)$, the space $(D,g)\subset M$ is a minimal filling and is boundary rigid.

Let $S=S^{dn-1}$ and $\cL=L^\infty(S)$. For any $r>0$, let $\cB(r)$ be the ball of radius $r$ in $\cL$ centered at the origin. 

The technical results established in the rest of the paper can be summarized by the following proposition.

\begin{prop}\label{key_prop}
If $g$ is sufficiently close to $g_0$, then there exists a distance preserving map $\Phi: M\to \cL$ such that $\Phi(x_0)=0\in\cL$ and a Lipschitz map
$$P_\sigma:\cB(R)\union\Phi(M)\to M$$
satisfying the following properties
\begin{enumerate}
\item[(1).] $P_\sigma\circ\Phi=\id_M.$
\item[(2).] For every $dn$-dimensional Riemannian manifold $N$ and every 1-Lipschitz map $f:N\to \cB(R)$, the composition $P_\sigma\circ f$  does not increase $dn$-dimensional volumes.
\item[(3).] For any $N$ and $f$ as above, $f(N)\subset \Phi(M)$ provided the composition $P_\sigma \circ  f:N \to M$ preserves volumes of all measurable sets.
\end{enumerate}
\end{prop}
\begin{proof}[Proof of Theorem \ref{main_thm_alt} from Proposition \ref{key_prop}]
The proof is the same as the proof of \cite[Theorem 1.6]{Burago2} from \cite[Propsition 2.1]{Burago2}. We will present the proof for reader's convenience.

Let $g$ be sufficiently close to $g_0$ so that the maps $\Phi$ and $P_\sigma$ from Proposition \ref{key_prop} exist. Let $D'$ be a smooth compact manifold with boundary $\del D'=\del D$ and $g'$ be a metric on $D'$ such that
$$d_{(D',g')}(x,y)\geq d_{(D,g)}(x,y),\quad \forall x,y\in \del D.$$
For simplicity we write $M'=(D',g')$. Notice that $(D,g)\subset M$, we have
$$d_{(D,g)}(x,y)\geq d_M(x.y),\quad \forall x,y\in \del D.$$
Therefore 
$$d_{M'}(x,y)\geq d_M(x,y),\quad \forall x,y\in\del D.$$
Since $\Phi$ is distance preserving with respect to $M$, $\Phi|_{\del D}$ is 1-Lipschitz with respect to the metric on $M'$. Therefore we can apply the method in \cite[Proposition 1.6]{Ivanov2} (or \cite[Proposition 4.9]{Burago1}) to construct $\Phi':M'\to \cL$ as an extension of $\Phi|_{\del D}$. Here is an explicit formula for $\Phi'$: 
$$\Phi'(x)(s)=\chi_R(\inf\{\Phi(y)(s)+d_{M'}(x,y):y\in\del D\}),\quad \forall x\in M', s\in S,$$
where $\chi_R:\cL\to\cL$ is a cutoff function given by 
$$\chi_R(\phi)(s)=\min\{R/2,\max\{-R/2,\phi(s)\}\}.$$
(Since we assumed that $D\subset B_{x_0}(R/5)$ and that for any $x\in B_{x_0}(R/5)$, $d_g(x,x_0)<R/4$ at the beginning of this section, the cutoff function does not change anything when $x\in\del D$. Therefore $\Phi|_{\del D}=\Phi'|_{\del D}$) 

Consider a map $\pi=P_\sigma\circ \Phi':M'\to M$. The first assertion in Proposition \ref{key_prop} implies that $\pi|_{\del M'}=\id_{\del D}$, therefore $D\subset \pi(M')$. The second assertion of Proposition \ref{key_prop} implies that $\vol(M')\geq\vol(D,g)$. Therefore $(D,g)$ is a minimal filling.

To prove that $(D,g)$ is a strict minimal filling, suppose that $\vol(M')=\vol(D,g)$. Then $\pi$ is volume-preserving. By the third assertion of Propositon \ref{key_prop} we have $\Phi'(M')\subset \Phi(M)$. Therefore $\pi=\Phi^{-1}\circ\Phi'$. Since $\Phi'$ is 1-Lipschitz and $\Phi$ is distance preserving, $\pi$ is therefore a 1-Lipschitz volume-preserving map. It follows from \cite[Lemma 9.1]{Burago1} that $\pi$ is an isometry. Hence $(D,g)$ is a strict minimal filling.
\end{proof}

\begin{proof}[Proof of Corollary \ref{main_thm}]
The prove is also a $\KK\mathbf{H}^n$ version of \cite[Proof of Theorem 1.3]{Burago2}. For reader's convienience we will present the proof here.

Let $D$ and $g$ be as above, and let $g'$ be a Riemannian metric on $D'$ such that $D'$ and $D$ share the same boundary and $g'$ induces the same boundary distance function as $(D,g)$. By Theorem \ref{main_thm_alt} $(D,g)$ is a strict minimal filling. Hence it suffices to show that  $\vol(D,g)=\vol(D',g')$. Since $D$ is a region in contained in a large ball in $M$, it satisfies the SGM condition introduced by C. B. Croke \cite{Croke1} if $g$ is sufficiently close to $g_0$ (for example, $g$ having negative sectional curvature). Hence $\vol(D,g)=\vol(D',g')$ by \cite[Lemma 5.1]{Croke1}. 
\end{proof}

\section{General setup and computations}\label{s3}

Recall that in Proposition \ref{key_prop} we need existence of an ``embedding'' $\Phi:M\to\cL$ and a ``projection'' $P_\sigma:\cB(R)\union\Phi(M)\to M$ satisfying certain properties. We will adopt the general setup introduced in \cite[Section 3, General computations]{Burago2} in order to help us understand these maps. For reader's convenience, we will list their major concepts and results without proof.

\begin{notation}
In this section, we assume that $(M,g)$ is a $n$-dimensional manifold where any two points are connected by a geodesic realizing the distance. This is always the case when $(M,g)$ is complete. Let $S=S^{n-1}$ and $\cL:=L^\infty(S)$. We equip $S$ with the standard (Haar) probability measure $ds$. In the rest of the paper, we let $L^2(S)=L^2(S,ds)$.

We denote by $T^1M$ the unit tangent bundle of $M$ and by $T^1_xM$ its fiber over $x\in M$. Let $ds_x$ be the standard probability measure on $T^1_xM$ with respect to the Riemannian metric $g$. We will use $ds_x=ds_{x,g}$ for simplicity when there is no ambiguity of metric.

\end{notation}

\begin{dfn}[Special embedding]\label{special_embedding}
A map $\Phi:M\to \cL$ is a \textit{special embedding} if there is a family $\{\Phi_s\}_{s\in S}$ of real-valued functions on $M$ such that the following holds:
\begin{enumerate}
\item[(1).] For every $x\in M$, the image $\Phi(x)$ is a function $s\to\Phi_s(x)$ which belongs to $\cL$.
\item[(2).] The function $(x,s)\to \Phi_s(x)$ is smooth on $M\times S$.
\item[(3).] Every function $\Phi_s:M\to \RR$ is distance-like; that is, $|\grad \Phi_s|\equiv 1$. 
\item[(4).] For every $x\in M$, the map $s\to\grad\Phi_s(x)$ is a diffeomorphism between $S$ and $T^1_xM$.
\end{enumerate}
\end{dfn}
\begin{rmk}
Every special embedding $\Phi$ is a distance preserving map. The third assertion in Definition \ref{special_embedding} implies that $\Phi$ is 1-Lipschitz. To prove that it is distance preserving for all $x,y\in M$, consider a unit speed geodesic $\gamma$ connecting $x$ and $y$. By the fourth assertion in Definition \ref{special_embedding}, there exists some $s\in S$ such that $\grad \Phi_s(x)$ is the initial velocity vector of $\gamma$. Since $\Phi_s$ is distance-like, its gradient curves are geodesics, which implies that $\gamma$ is a gradient curve of $\Phi_s$. Therefore 
$$\Phi_s(y)-\Phi_s(x)=d_M(x,y)$$
and hence $\| \Phi(x)-\Phi(y)\|_{L^\infty}\geq d_{M}(x,y)$. Thus $\Phi$ is distance-preserving.
\end{rmk}
\begin{notation}
Let $\alpha_x:T^1_xM\to S$ be the inverse of $s\to\grad\Phi_s(x)$ and $\alpha:T^1M\to S$ be a map such that $\alpha|_{T^1_xM}=\alpha_x$ for every $x\in M$.

We define a probability measure $d\mu_x$ by the pushforward of the standard probability measure $ds_x$ on $T^1_xM$ to $S$. In other words,
$$d\mu_x=(\alpha_x)_*ds_x.$$
We denote by $\lambda(x,s)$ the density of $d\mu_x$ at $s\in S$ with respect to $ds$. The second and the fourth assertion in Definition \ref{special_embedding} imply that $\lambda:M\times S\to \RR$ is smooth and positive.
\end{notation}
\begin{dfn}[Scalar product, Riemannian metric and special Riemannian metric]\label{fcnal_RM}
A symmetric bilinear form $G$ on $\cL$ is called a \emph{scalar product} on $\cL$ if it is $L^2$-compatible (with respect to $ds$). In other words, there exists some positive constants $c,C$ such that 
$$c\|u\|^2_{L^2(S)}\leq G(u,u)\leq C\|u\|^2_{L^2(S)},\quad \forall u\in\cL.$$
A \emph{Riemannian metric} in an open subset $\cU\subset\cL$ is a smooth family $G=\{G_\phi\}_{\phi\in\cU}$ of scalar products on $\cL$. In other words, for any point $\phi\in\cU$, there is a scalar product $G_\phi$ defined on $T_\phi\cL=\cL$ which depends smoothly on the base point $\phi$.

Let $\Phi:M\to\cL$ be a special embedding and $G$ be a Riemannian metric in an open subset $\cU\subset\cL$ comtaining $\Phi(M)$. We say that $G$ is \emph{special} with respect to $\Phi$ if the following hold:
\begin{enumerate}
\item[(1).] For every $\phi\in\cU$, the scalar product $G_\phi$ has the form
$$G_\phi(X,Y)=n\int_SX(s)Y(s)d\nu_\phi(s),\quad \forall X,Y\in\cL,$$
where $\nu_\phi$ is a probability measure on $S$.
\item[(2).] Every measure $\nu_\phi$ has positive density bounded away from zero with respect to $ds$; these densities depend smoothly on $\phi$.
\item[(3).] If $\phi=\Phi(x)$ for an $x\in M$, then $\nu_\phi=\mu_x$.
\end{enumerate} 
\end{dfn}
\begin{notation}
Let $G$ be a scalar product in $\cL$ and $V$ be a $n$-dimensional Euclidean space. Let $T:\cL\to V$ be a linear map bounded with respect to $G$. Denote by $\Jac_{G,\cW}T$ the Jacobian of $T|_\cW$, where $\cW\subset \cL$ is an arbitary $n$-dimensional subspace. We define Jacobian of $T$ as
$$\Jac_GT:=\sup_{\cW\subset \cL,~\dim(\cW)=n}\Jac_{G,\cW}T.$$

Let $M$ be an arbitary $n$-dimensional Riemannian manifold. For any smooth map $F:\cL\to M$, we denote by $d_\phi F:T_\phi \cL\to T_{F(\phi)}M$ the tangent map of $F$ at $\phi$.

For simplicity, we will use $T_x\Phi=\Phi_*(T_{x}M)$ for any $x\in M$ and $\Phi:M\to\cL$.
\end{notation}
\begin{dfn}[$L^2$-smooth]\label{L2_smooth}
Let $\cU\subset\cL$ be an open subset of $\cL$. We say that a map $P:\cU\to M$ is \emph{$L^2$-smooth} if it is differentiable with respect to the $L^\infty$ structure and its derivative at every point $\phi\in\cU$ can be extended to a bounded linear map from $L^2$ to a fiber of $TM$ which depends smoothly on $\phi$.
\end{dfn}
\begin{dfn}[Projection]\label{Projection}
Let $\Phi:M\to\cU\subset\cL$ be a smooth isometric immersion with respect to a Riemannian metric $G$ on $\cU$. We say that a map $P:\cU\to M$ is a \emph{projection} if it is $L^2$-smooth and satisfies the following two properties.
\begin{enumerate}
\item[(1).] $P\circ\Phi=\id_M$;
\item[(2).] for every $x\in M$, $d_{\Phi(x)}P(V)=0$ for every vector $V\in\cL$ orthogonal (with respect to $G$) to $T_x\Phi$.
\end{enumerate}
\end{dfn}
\begin{prop}\label{crit_jacobian_cond}\cite[Proposition 3.13]{Burago2}
Let $\Phi:M\to\cU\subset\cL$ be a special embedding and $G$ be a Riemannian metric with respect to $\Phi$. Let $P:\cU\to M$ be a projection in the sense of Definition \ref{Projection}. Then for every $x\in M$ and every $V\in\cL$ orthogonal to $T_x\Phi$, we have
$$d_{\Phi(x)}\Jac_GP(V)=0.$$
\end{prop}
\begin{prop}\label{vol_ineq}\cite[Lemma 3.14, Lemma 3.15]{Burago2}
Let $N$ be a $dn$-dimensional Riemannian manifold (with volume form $d\mathrm{vol}_N$) and $f:N\to\cL$ be a 1-Lipschitz map. Suppose that $G$ is a special Riemannian metric in an open subset $\cU\subset\cL$ with respect to a special embedding $\Phi:M\to\cL$ and $f(N)\subset\cU$.  Assume $P:\cU\to M$ is an $L^2$-smooth map. Then we have the following inequalities
\begin{enumerate}
\item[(1).] $f$ does not increase $n$-dimensional volume. In other words,
$$\vol_G(f):=\int_Nd\mathrm{vol}_{f^*G}\leq\vol(N),$$
where $d\mathrm{vol}_{f^*G}$ denotes the volume form on $N$ with repect to the Riemannian metric $f^*G$.
\item[(2).]  
$$\vol(P\circ f):=\int_N (P\circ f)^*d\mathrm{vol}_{M}\leq \int_N \Jac_GP(f(x))d\mathrm{vol}_{N}(x).$$
\end{enumerate}
\end{prop}

\section{The construction in $\KK\mathbf{H}^n$}\label{s4}

Recall in Section \ref{s2} we proved Theorem \ref{main_thm_alt} and Corollary \ref{main_thm} assuming Proposition \ref{key_prop}, which asserts the existence of an ``embedding'' map $\Phi$ and a ``projection'' map $P_\sigma$. In this section, we will present the main construction behind these aforementioned maps. We will adopt the same construction as in \cite[Section 4, The construction]{Burago2} for $\KK\mathbf{H}^n$ with a metric $g$ close to the symmetric metric.

To simplify exposition, we do not track the dependence on $g$ and its derivatives in our proof. We say that a dependence on $g$ is \emph{smooth} if for every integer $k>0$ there exists an $r>0$ such that this dependence is $k$-times differentiable with respect to the $C^r$-norm on a neighborhood of $g_0$ in the space of metrics. 

\begin{notation}
Let $B_{x_0}(r)$ the ball of radius $r$ in $\KK\mathbf{H}^n$ centered at $x_0$ with respect to the symmetric metric $g_0$, where $x_0\in\KK\mathbf{H}^n$ is a fixed point. Recall that in Section \ref{s2} we assumed that $D\subset B_{x_0}(R/5)$ with metric $g$ smoothly exteneded to the whole $\KK\mathbf{H}^n$ (also denoted as $g$). Moreover, $g$ is $C^r$ close to $g_0$ and $g\equiv g_0$ outside $B_{x_0}(R/2)$. Let $M=(\KK\mathbf{H}^n,g)$ and $M_0=(\KK\mathbf{H}^n,g_0)$. Notice that $M_0$ and $M$ share the same underlying manifold, we denote by $d_{g_0}$, $d_g$ the distance functions on $M_0$ and $M$ respectively.
 
Since $g$ and $g_0$ coincide outside a compact set, boundaries at infinity for both $M$ and $M_0$ are canonically identified. Let $S=S^{dn-1}=\del_\infty M$ be the boundary at infinity and $\cL:=L^\infty(S)$ as in Section \ref{s2}. For every $s\in S$, we denote by $\Phi_s:M\to\RR$ the Busemann function of a geodesic ray starting at $x_0$ towards $s\in S$. To be more precise, if we let $\gamma(t)$ be the unit speed geodesic such that $\gamma(0)=x_0$ and $\gamma(\infty)=s$, we have
$$\Phi_s(x)=\lim_{t\to\infty}d_g(x,\gamma(t))-t,\quad x\in M,s\in S.$$
We define the ``embedding map'' $\Phi:M\to\cL$ such that
$$\Phi(x)(s)=\Phi_s(x),\quad x\in M, s\in S.$$
\end{notation}

\begin{lem}\label{verification_embedding}
The map $\Phi$ defined above depends smoothly on $x,s$ and $g$. If $g$ is sufficiently close to $g_0$, then $\Phi$ is a special embedding in the sense of Definition \ref{special_embedding}.
\end{lem}

We recall some notations introduced in the Section \ref{s3} before we give a proof for the above lemma. 
\begin{notation}
Let $ds_{x,g}$ be the standard probability measure on $T^1_xM$ with respect to the Riemannian metric $g$. 
Let $\alpha_{x,g}:T^1_xM\to S$ be the inverse of $s\to\grad\Phi_s(x)$ and $\alpha_g:T^1M\to S$ satisfying $\alpha_g|_{T^1_xM}=\alpha_{x,g}$ for every $x\in M$. For any $v\in T^1M$, denoted by $v(t)\in T^1M$ the image of $v$ after applying the geodesic flow for time $t$.  Then for any $v\in T^1M$ we have $\alpha_g(v)=v(-\infty)$, the negative infinity endpoint of the geodesic with initial vector $v$.

We define a probability measure $\mu_{x,g}$ by the pushforward of the standard probability measure on $T^1_xM$ to $S$. In other words,
$$\mu_{x,g}=(\alpha_{x,g})_*ds_{x,g}.$$

We denote by $\lambda_g(x,s)$ the density of $\mu_{x,g}$ at $s\in S$ with respect to the Haar measure $ds=(\alpha_{x_0,g_0})_*ds_{x_0,g_0}$. Lemma \ref{verification_embedding} and Definition \ref{special_embedding} imply that $\lambda_g:M\times S\to \RR$ is smooth and positive. For simplicity we will use $ds_x$, $\alpha_x$, $\alpha$, $\mu_x$ and $\lambda(x,s)$ instead of $ds_{x,g}$, $\alpha_{x,g}$, $\alpha_g$, $\mu_{x,g}$ and $\lambda_g(x,s)$ when there is no ambiguity on the choice of the metric $g$.
\end{notation} 

\begin{proof}[Proof of Lemma \ref{verification_embedding}]
Let $H_{s,c}$ be the horosphere in $M_0$ at $s\in S$ such that
\begin{enumerate}
\item[(1).] $H_{s,c}$ is tangent to $B_{x_0}(c)$ with $c>R>0$;
\item[(2).] Any geodesic ray starting at $H_{s,c}$ towards $s$ does not intersect the interior of $B_{x_0}(c)$. 
\end{enumerate}
Since $g$ and $g_0$ coincide on $\KK\mathbf{H}^n\setminus B_{x_0}(R/2)$, horospheres of $M$ contained in $\KK\mathbf{H}^n\setminus B_{x_0}(R/2)$ coincide with those of $M_0$. Therefore, 
\begin{align}\label{buse_formula}
\Phi_s(x)=d_g(x,H_{s,c_1})-d_g(x_0,H_{s,c_2})-c_1+c_2,\quad x\in M, s\in S,c_1,c_2\gg 1.
\end{align}
The first and the third conditions in Definition \ref{special_embedding} follow immediately from the definition of Busemann functions. To verify the second and the fourth conditions we first recall that in the proof of Corollary \ref{main_thm} we assumed that $g$ has negative sectional curvature. Notice that $\alpha_{g_0}$ is smooth, for any $v\in T^1M$, smoothness of the map $\alpha_g(v)=\alpha_g(v(-T))$ follows from choosing arbitarily large $T>0$ and the smoothness of $\alpha_{g_0}$. (This is because $\alpha_g(v(-T))=\alpha_{g_0}(v(-T))=$ when $T\gg 1$.) This verifies the fourth condition in Definition \ref{special_embedding} and also gives a smooth diffeomorphism from $T^1M$ to $M\times S$ by identifying $v\in T_xM$ with $(x,\alpha_g(v))$. Denoted by $\beta:T^1M_0\to  \RR$ a smooth map such that $\beta(v)= c$ for any $v$ orthogonal to $ H_{v(-\infty),c}$. The same map is also defined on those points in $T^1M$ where $g=g_0$. For sufficiently large $T>0$, by \eqref{buse_formula} we have
$$\Phi_{s}(x)=-\beta(\alpha^{-1}_{x,g}(s)(-T))+\beta(\alpha^{-1}_{x_0,g}(s)(-T)),\quad x\in M, s\in S.$$
Since all maps involved in the above formula are smooth, we have $\Phi_s(x):M\times S\to \RR$ is smooth. This verifies the second assertion in Definition \ref{special_embedding}.
\end{proof}

\begin{lem}\label{delta_X_density}
If $g=g_0$, then
$$\lambda(x,s)=e^{-\delta(M_0)\Phi_s(x)}=e^{-(dn+d-2)\Phi_s(x)},\quad\forall x\in M,s\in S,$$
where $\delta(M_0)$ is the volume growth entropy of $M_0$ defined as
$$\delta(M_0)=\overline{\lim}_{r\to \infty}\frac{\ln\left(\vol(\{x\in M_0| d_{M_0}(x,x_0)<r\})\right)}{r}.$$
\end{lem}
\begin{proof}
 The proof can be found in \cite{BCG1}.
\end{proof}
\begin{rmk}
 For more general $g$, we can assume that $g$ and $g_0$ are sufficiently close such that
\begin{align}\label{density_est}
\frac{1}{2}e^{-(dn+d-2)\Phi_s(x)}\leq\lambda(x,s)\leq 2e^{-(dn+d-2)\Phi_s(x)},\quad\forall x\in B_{x_0}(R_0),s\in S,
\end{align}
for some choice of positive real number $R_0>0$ to be determined. This will be useful in later computations when we choose a specific $R_0$ depending only on $R$ and $n$ to help verifying some properties in our construction.
\end{rmk}

Let $\cB(R)$ be the ball of radius $R$ centered at $0$ in $\cL$ (with respect to $L^\infty$-norm). We define a projection as the following.

\begin{dfn}\label{proj}
Let $\cU$ be a neighborhood of $\Phi(M)\union\cB(R)$ in $\cL$. For any $\phi\in\cU$, we define a map $\Omega_{\phi,g}:M\to T^*M$ as
\begin{align}\label{W_phi}
\Omega_{\phi,g}(x)=\int_Se^{(dn+d)[\Phi_s(x)-\phi(s)]}d\Phi_s(x)d\mu_{x,g}(s).
\end{align}
Let $P:\cU\to M$ be such that $\Omega_{\phi,g}(P(\phi))=0.$
\end{dfn}

We first prove that it is well-defined, which is the $\KK\mathbf{H}^n$ version of \cite[Lemma 4.4]{Burago2}. For reader's convenience, we provide a slightly different proof.

\begin{lem}\label{well_defined_P}
If $g$ is sufficiently close to $g_0$, then there exists a smooth map $P$ satisfying Definition \ref{proj} such that
\begin{enumerate}
\item[(1).]$P(\Phi(x))=x$ for all $x\in M$;
\item[(2).]There exists some constant $R_1=R_1(n,R)$ depending only on $n$ and $R$ such that for any $\phi\in\cB(R)$, $P(\phi)\in B_{x_0}(R_1)$ and $d_g(P(\phi),x_0)\leq R_1$. 
\end{enumerate}
Hence as a direct corollary of the second assertion, $\Phi(P(\phi))\in\cB(R_1)$ for any $\phi\in \cB(R)$.
\end{lem}
\begin{proof}
If $\phi=\Phi(x)$, then we define $P(\phi)=x$ and it satisfies the requirements in Definition \ref{proj}. In the rest of the proof we extend $P$ to a neighborhood of $\Phi(M)$ containing $\cB(R)$.

Consider a map $E:\cL\to L^2(S)$ given by
$$E(\phi)=e^{-(dn+d)\phi(s)}.$$
Let $\phi\in\cL$ and $\psi=E(\phi)$. Then the equation $\Omega_{\phi,g}(x)=0$ takes the form
$$\int_S\psi(s)e^{(dn+d)\Phi_s(x)}d\Phi_s(x)d\mu_x(s)=0.$$
Notice that when $\phi=\Phi(x)$, spherical symmetry implies that
$$\int_SE\circ\Phi_s(x)e^{(dn+d)\Phi_s(x)}d\Phi_s(x)d\mu_x(s)=0.$$
Hence the equation $\Omega_{\phi,g}(x)=0$ is equivalent to the following
$$\int_S\left(\psi(s)-E\circ\Phi_s(x)\right)e^{(dn+d)\Phi_s(x)}d\Phi_s(x)d\mu_x(s)=0,$$
which is equivalent to 
$$\int_S\left(\psi(s)-E\circ\Phi_s(x)\right)e^{2(dn+d)\Phi_s(x)}d_x(E\circ\Phi)(s)d\mu_x(s)=0.$$
Define $\cE:M\times L^2(S)\to T^*M$ as 
$$\cE(x,\psi)=\int_S\left(\psi(s)-E\circ\Phi_s(x)\right)e^{2(dn+d)\Phi_s(x)}d_x(E\circ\Phi)(s)d\mu_x(s).$$
By the Implicit Function Theorem applied to $\cE(x,\psi)=0$, there exists a smooth map $\widetilde P:\widetilde\cU\to M$ defined on a neighborhood $\widetilde\cU$ of $E(\Phi(M))$ such that $\cE(\widetilde P(\psi),\psi)=0$. Therefore we can extend $P$ to $E^{-1}(\widetilde\cU)$ by setting $P=\widetilde P\circ E$.

It remains for us to extend $P$ to $\cB(R)$. Let 
$$\omega_g(x)=e^{(dn+d)\Phi_s(x)}\lambda_g(x,s)d\Phi_s(x)\in T^*_xM$$
and hence 
$$\Omega_{\phi,g}(x)=\int_SE(\phi(s))\omega_g(x)ds.$$

When $g=g_0$, Lemma \ref{delta_X_density} implies that $\omega_g=d\left(e^{2\Phi_s}/2\right)$. A classic result from \cite{BCG1} (to be more precise, the first assertion of Lemma \ref{hess}) implies that
$$\nabla\omega_{g_0}>e^{2\Phi_s}g_0.$$
Hence
$$\nabla \Omega_{\phi,g_0}\geq\int_SE(\phi(s))e^{2\Phi_s}g_0ds>0.$$
For general $g$ close to $g_0$, we denote the induced quadratic form of $\nabla\omega_g$ and $\nabla\Omega_{\phi,g}$ by the same notations. Similar to the remark for Lemma \ref{delta_X_density}, we can assume that $g$ is sufficiently close to $g_0$ in the sense that 
\begin{align}\label{cond_on_g_2}
\nabla \omega_g\geq\frac{1}{2}e^{2\Phi_s} g, \quad \forall x\in B_{x_0}(\widetilde R_0)
\end{align}
for some choice of $\widetilde R_0=\widetilde R_0(n,R)>2R>0$ which will be determined in later part of this proof. Notice that $B_{x_0}(r)$ is convex in $M$ when $r>2R$. Then
$$
\nabla \Omega_{\phi,g}|_x\geq\frac{1}{2}\int_SE(\phi(s))e^{2\Phi_s(x)}g|_xds>0,\quad \forall x\in B_{x_0}(\widetilde R_0).
$$
Therefore for any unit speed geodesic segment $\gamma$ in $ B_{x_0}(\widetilde R_0)$, the function $\Omega_{\phi,g}(\dot\gamma(t))$ has derivative equal to $\nabla\Omega_{\phi,g}(\dot\gamma,\dot\gamma)$ and hence is strictly increasing, which implies that $\Omega_{\phi,g}=0$ has at most one solution in $B_{x_0}(\widetilde R_0)\subset M$. Moreover, since $\nabla \Omega_{\phi,g}$ is non-degenerate, we can apply the Implicit Function Theorem to $\Omega_{\phi,g}=0$, which proves that we can extend $x=P(\phi)$ smoothly if the equation $\nabla \Omega_{\phi,g}=0$ has a solution for any $\phi\in \cB(R)$.

To prove the existence of such a solution, we first claim that 
$\inf_{x\not\in B_{x_0}(r)}\int_Se^{2\Phi_s(x)}ds\to\infty$ as $r$ tends to infinity (independent of the choice of $g$). Define
$$\shadow_g(N,p)=\{s\in S|~\mathrm{Im}\gamma_{s,p}\ints N\neq\emptyset\},\quad N\subset M, p\in M\sqcup S,$$
where $\gamma_{s,p}$ denotes the geodesic ray (with respect to $g$) starting at $p$ towards $s$.
Since in Lemma \ref{verification_embedding} we assumed that $g$ is negatively curved, for any $r>R/2$ and any $x\in \KK\mathbf{H}^n\setminus B_{x_0}(R/2)$, we have $\shadow_g(B_{x_0}(r),x)=\shadow_{g_0}(B_{x_0}(r),x)$. Define
$$D_x=\shadow_{g_0}(B_{x_0}(2R),x)\setminus\shadow_{g_0}(B_{x_0}(R),x)\subset S,\quad\forall x\in M\sqcup S.$$
Notice that for any $r>R>0$ and any $s\in S$, the set 
$$\shadow_{g_0}(B_{x_0}(r),s):=\lim_{x\to s}\shadow_{g_0}(B_{x_0}(r),x)$$
has a fixed positive area only depending on $r$ and $n$ with respect to the Haar measure $ds$ on $S$. Moreover, the area of $\shadow_{g_0}(B_{x_0}(r),x)$ only depends on $d_{g_0}(x,x_0)$ for any $x\in M\sqcup S$. Hence in particular $\mathrm{Area}(D_s,ds)$ equal to some positive constant $I(n,R)>0$ for any $s\in S$ and there exists a continuous non-negative function $\cA:\RR_{\geq 0}\to \RR_{\geq 0}$ such that $\mathrm{Area}(D_x,ds)=\cA(d_{g_0}(x,x_0))\to I(n,R)$ as $d_{g_0}(x,x_0)\to\infty$. Notice that $\Phi_s(x)\geq d_g(x,x_0)-2\mathrm{diam}_g(B_{x_0}(2R))$ for any $s\in \shadow_{g_0}(B_{x_0}(2R),x)$ and any $x\in M$. Without loss of generality we can assume that $g$ is sufficiently close to $g_0$ such that $\mathrm{diam}_g(B_{x_0}(2R))<5R$. Then we have
$$\int_Se^{2\Phi_s(x)}ds\geq \cA(d_{g_0}(x,x_0))e^{2(d_g(x,x_0)-10R)}\geq \cA(d_{g_0}(x,x_0))e^{2(d_{g_0}(x,x_0)-12R)},\quad\forall x\in M,$$
which proves that $\inf_{x\not\in B_{x_0}(r)}\int_Se^{2\Phi_s(x)}ds\to\infty$ as $r$ tends to infinity. Therefore 
$$\inf_{x\not\in B_{x_0}(r)}\int_SE(\phi(s))e^{2\Phi_s(x)}ds\geq e^{-(dn+d)R}\inf_{x\not\in B_{x_0}(r)}\int_Se^{2\Phi_s(x)}ds\geq \widehat{\cA}(n,R,r)\to\infty$$
as $r$ tends to infinity, where $\widehat{\cA}(n,R,\cdot)$ is a positive continuous function defined on $\RR_{\geq 0}$ depending on $n$ and $R$. Hence there exists some constant $r_0=r_0(n,R)>0$ such that 
$$\nabla \Omega_{\phi,g}\geq r_0g,\quad\forall x\in B_{x_0}(\widetilde R_0)$$
for $\widetilde R_0>0$ as in \eqref{cond_on_g_2}. Let $\cX_{\phi,g}$ be the dual of $\Omega_{\phi,g}$ with respect to the metric $g$ and define a function $\cF=\left|\cX_{\phi,g}\right|_g^2=g(\cX_{\phi,g},\cX_{\phi,g})$ on $M$. Notice that $\phi\in\cB(R)$ implies $\cF(x_0)\leq e^{2(dn+d)R}$ and that 
$$\sqrt{\cF(\gamma(t))}\geq\left|\frac{g(\cX_{\phi,g},\dot\gamma(t))}{\|\dot\gamma(t)\|_g}\right|=|\Omega_{\phi,g}(\dot\gamma(t))|\geq\int_{0}^t\nabla\Omega_{\phi,g}(\dot\gamma(\tau),\dot\gamma(\tau))d\tau-|\Omega_{\phi,g}(\dot\gamma(0))|,$$ 
where $\gamma$ is any unit speed geodesic segment in $M$. Choose $R_1(n,R)=2e^{2(dn+d)R}/r_0>0$ and $g$ sufficiently close to $g_0$ such that $\widetilde R_0=R_1+1$, we have

$$\cF(x)>\cF(x_0),\quad\forall x\in B_{x_0}(\widetilde R_0)\setminus B_{x_0}(R_1).$$
Hence $\cF|_{B_{x_0}(\widetilde R_0)}$ achieves minimum at some point $x_{\min}\in B_{x_0}(R_1)$. In particular
\begin{align*}
0=&\cX_{\phi,g}\cF(x_{\min})=2g\left(\nabla_{\cX_{\phi,g}}\cX_{\phi,g},\cX_{\phi,g}\right)
=2\left(\nabla \Omega_{\phi,g}\right)\left(\cX_{\phi,g},\cX_{\phi,g}\right) 
\geq2r_0g(\cX_{\phi,g},\cX_{\phi,g}),
\end{align*}
which implies that $\cX_{\phi,g}(x_{\min})=0$. Hence $P$ is well-defined and smooth. 
\end{proof}

Before we state Lemma \ref{AE_decomp}, we first introduce some notations.

\begin{notation}
 For any $x\in B_{x_0}(\widetilde R_0)$ and $s\in S$, we define a linear operator $A_{x,s}:T_xM\to T_xM$ by 
\begin{align}\label{Axs}
A_{x,s}(\xi)=e^{-(dn+d)\Phi_s(x)}\lambda(x,s)^{-1}\nabla_\xi[e^{(dn+d)\Phi_s(x)}\lambda(x,s)\grad\Phi_s(x)],
\end{align}
where $\nabla_\xi$ denotes the Levi-Civita derivative along $\xi$. 

Recall that in the previous proof, we set
$$\omega_g(x)=e^{(dn+d)\Phi_s(x)}\lambda(x,s)d\Phi_s(x)\in T^*_xM.$$
Hence for any $g$ sufficiently close to $g_0$ in the sense of \eqref{cond_on_g_2} and $0\neq\xi\in T_xM$, we have
\begin{align}\label{invertibility_Axs}
\langle A_{x,s}(\xi),\xi\rangle=e^{-(dn+d)\Phi_s(x)}\lambda(x,s)^{-1}(\nabla\omega_g)(\xi,\xi)>0.
\end{align}

Let $\phi\in\cL$ and $x=P(\phi)$. Denoted by $\rho_\phi$ a function  on $S$ such that
\begin{align}\label{rho_phi}
\rho_\phi(s)=e^{(dn+d)[\Phi_s(x)-\phi(s)]}
\end{align}
Let $\overline\rho_\phi$ be the same function normalized with respect to the measure $\mu_x$. In other words,
\begin{align}\label{rho_normalized}
\overline\rho_\phi=\frac{\rho_\phi}{\int_S\rho_\phi d\mu_x}.
\end{align}

Assuming $x=P(\phi)\in B_{x_0}(\widetilde R_0)$, we define a linear operator $A_\phi:T_xM\to T_xM$ by
\begin{align}\label{A}
A_\phi=\int_S\overline\rho_\phi(s)A_{x,s}d\mu_x(s),
\end{align}
which is invertible due to \eqref{invertibility_Axs}.
\end{notation}
\begin{lem}\label{AE_decomp}
Let $\phi\in\cU$ and $x=P(\phi)$. Then differentiating $P$ yields
$$d_\phi P=A^{-1}_\phi\circ E_\phi,$$
for any $\phi$ such that $x\in B_{x_0}(\widetilde R_0)$. The linear map $E_\phi:\cL\to T_xM$ is given by
\begin{align}\label{E}
E_\phi(X(s))=(dn+d)\int_SX(s)\overline\rho_\phi(s)\grad \Phi_s(x)d\mu_x(s)
\end{align}
for any $X(s)\in T_\phi\cL=\cL$.
\end{lem}
\begin{proof}
By Definition \ref{proj}, We have
$$\Omega_{\phi,g}(P(\phi))=0,\quad\forall \phi\in\cU.$$
Let $x=P(\phi)$. For any $X\in T_\phi\cL=\cL$, $\xi=d_\phi P(X)\in T_xM$ and any vector field $\widetilde\eta$ on $M$, we differentiate the above equation evaluated at $\widetilde\eta$ and obtain

$$\int_S D_\phi(\rho_\phi(s))(X)\langle\grad\Phi_s(x),\widetilde\eta \rangle d\mu_x(s)+\xi\left[\Omega_{\phi,g}(P(\phi))(\widetilde\eta)\right]=0.$$
Notice that 
$$\xi\left[\Omega_{\phi,g}(P(\phi))(\widetilde\eta)\right]=\xi\langle\cX_{\phi,g},\widetilde\eta \rangle=\langle\nabla_\xi\cX_{\phi,g},\widetilde\eta\rangle+\langle\cX_{\phi,g},\nabla_\xi\widetilde\eta\rangle,$$
where $\cX_{\phi,g}$ is the dual of $\Omega_{\phi,g}$ with respect to $g$ as in the proof of Lemma \ref{well_defined_P} and hence vanish at $P(\phi)$.
By arbitariness of $\widetilde\eta$, we have
\begin{align}\label{DW}
\int_S D_\phi(\rho_\phi(s))(X)\grad\Phi_s(x)d\mu_x(s)+\nabla_{\xi} \cX_{\phi,g}=0,
\end{align}
 By \eqref{rho_phi}, \eqref{rho_normalized} and \eqref{E}, the first term in \eqref{DW} takes the form
\begin{align*} 
&\int_S D_\phi(\rho_\phi(s))(X)\grad\Phi_s(x)d\mu_x(s) \\
=&-(dn+d)\int_SX(s)\rho_\phi(s)\grad\Phi_s(x)d\mu_x(s) 
=-\left(\int_S\rho_\phi(s)d\mu_x(s)\right)E_\phi(X).
\end{align*}
A direct computation yields
\begin{align*}
\nabla_\xi \cX_{\phi,g} 
=&\int_S e^{-(dn+d)\phi(s)}\nabla_\xi\left[e^{(dn+d)\Phi_s(x)}\lambda(x,s)\grad\Phi_s(x)\right]ds \\
=&\int_S\rho_\phi(s)e^{-(dn+d)\Phi_s(x)}\nabla_\xi\left[e^{(dn+d)\Phi_s(x)}\lambda(x,s)\grad\Phi_s(x)\right]ds \\
=&\int_S\rho_\phi(s)e^{-(dn+d)\Phi_s(x)}\lambda(x,s)^{-1}\nabla_\xi\left[e^{(dn+d)\Phi_s(x)}\lambda(x,s)\grad\Phi_s(x)\right]d\mu_x(s) \\
=&\int_S\rho_\phi(s)A_{x,s}(\xi)d\mu_x(s) 
=\left(\int_S\rho_\phi(s)d\mu_x(s)\right)A_\phi(\xi),
\end{align*}
where the last two equalities follows from our notations in \eqref{Axs} and \eqref{A}.

Summarizing the above, \eqref{DW} implies that
$$A_\phi(\xi)=E_\phi(X).$$
Therefore,
$$d_\phi P(X)=\xi= A_\phi^{-1}\circ E_\phi(X).$$
By the arbitariness of $X$, we have
\begin{align*}d_\phi P=A_\phi^{-1}\circ E_\phi.\tag*{\qedhere}\end{align*}
\end{proof}

\begin{dfn}\label{RML}
Let $G$ be a Riemannian metric on $P^{-1}(B_{x_0}(\widetilde R_0))$ such that for any $\phi \in  P^{-1}(B_{x_0}(\widetilde R_0))$, the scalar product $G_\phi$ on $T_\phi\cL=\cL$ is defined by
\begin{align}\label{RM_formula}
G_\phi(X,Y)=nd\int_SX(s)Y(s)\overline\rho_\phi(s)d\mu_x(s), \quad\forall X,Y\in\cL,
\end{align}
where $x=P(\phi)$. 
\end{dfn}
\begin{rmk}
There are a number of different inner products mentioned in the latter half of this paper. In order to avoid cumbersome notations we denote by $\langle\cdot,\cdot\rangle$ the inner product of a Euclidean space or a tangent space for a particular Riemannian manifold. The distinction between different settings will be indicated via different notations of Riemannian manifolds/metrics or verbal descriptions. We also denote by $\langle\cdot,\cdot\rangle_{G_\phi}$ the above mentioned scalar product $G_\phi$. The norm induced by $G_\phi$ is denoted by $\|\cdot\|_{G_\phi}$.
\end{rmk}
\begin{lem}\label{RM_verification}
Let $P$ be the map defined in Definition \ref{proj}. Then we have the following.
\begin{enumerate}
\item[(1).] $G$ is a special metric with respect to $\Phi$.
\item[(2).] $P$ is a projection with respect to $G$ and $\Phi$ in the sense of Definition \ref{Projection}.
\end{enumerate}
\end{lem}
\begin{proof}
\begin{enumerate}
\item[(1).] The first two requirements of Definition \ref{fcnal_RM} follow immediately. For any $\phi=\Phi(x)$, we have 
$$\rho_\phi(s)=e^{(dn+d)(\Phi_s(x)-\Phi_s(x))}\equiv 1.$$
Therefore 
$$\cX_{\phi,g}(x)=\int_S\grad\Phi_s(x)d\mu_x(s)=\int_{T^1_xM}vds_x(v)=0,$$
which is equivalent to $\Omega_{\phi,g}(x)=0$ and therefore $P(\phi)=x$. Hence $P\circ\Phi=\id_M$.
Meanwhile, the scalar product $G_\phi$ on $T_\phi\cL=\cL$ is defined by 
$$G_\phi(X,Y)=nd\int_SX(s)Y(s)d\mu_x(s).$$
Direct computations imply that it satisfies the third requirement of Definition \ref{fcnal_RM}.
\item[(2).] The fact that $P$ is $L^2$-smooth in the sense of Definition \ref{L2_smooth} follows from Lemma \ref{well_defined_P} and Lemma \ref{AE_decomp}. Since we proved the first requirement of Definition \ref{Projection} in our previous assertion, it remains to verify the second requirement, that is, for every $x\in M$, $d_{\Phi(x)}P(V)=0$ for every vector $V\in\cL$ orthogonal (with respect to $G$) to $T_x\Phi$. Let $\phi=\Phi(x)$ for simplicity. By Lemma \ref{AE_decomp}, it suffices to show that $E_{\phi}(X)=0$ for any $X\in T_\phi\cL$ perpendicular to $T_x\Phi$ with respect to $G_\phi$.

Let $v\in T_xM$ be an arbitary vector in $T_xM$. We have
\begin{align*}
\langle E_\phi(X),v\rangle 
=&\left\langle (dn+d)\int_SX(s)\grad\Phi_s(x)d\mu_x(s),v\right\rangle \\
=&(dn+d)\int_SX(s)\langle\grad\Phi_s(x),v\rangle d\mu_x(s) \\
=&(dn+d)\int_SX(s)d_x\Phi(v)(s)d\mu_x(s)\\
=&\frac{(dn+d)}{dn}\cdot dn\int_SX(s)d_x\Phi(v)(s)d\mu_x(s)
=\frac{(n+1)}{n}G_\phi(X,d_x\Phi(v))=0.
\end{align*}
By arbitariness of $v\in T_xM$, we have $E_\phi(X)=0$. Hence $d_\phi P(X)=0$.\qedhere
\end{enumerate}
\end{proof}

\section{Approximating $d_\phi P$}\label{s5}
\subsection{The construction of $\widehat{A}_\phi$ when $\KK\neq\OO$}\label{ss51}
We will first recall the model for $M_0=(\KK\mathbf{H}^n,g_0)$ in \cite[Ch II.10]{BH99} when $\KK\neq\OO$. Let $\KK\mathbf{P}^n:=(\KK^{n+1}\setminus\{0\})/\{v\sim v\lambda,\lambda\in\KK\setminus\{0\}\}$ and $[v]\in\KK\mathbf{P}^n$ be the equivalence class of $v$. Denoted by $q$ a quadratic form on $\KK^{n+1}$ defined as
$$q(v,w)=\overline{v}_0w_0-\sum_{j=1}^n\overline{v}_jw_j$$
for any $v=(v_0,...,v_n)$ and $w=(w_0,...,w_n)$ in $\KK^{n+1}$. Define
$$\KK\mathbf{H}^n=\{[v]\in\KK\mathbf{P}^n|v\in\KK^{n+1}\setminus\{0\}, q(v,v)>0\}.$$
For any $u\in\KK^{n+1}$ with $q(u,u)=1$, the tangent space of $\KK\mathbf{H}^n$ at $[u]$ can be identified as $\{v\in\KK^{n+1}|q(u,v)=0\}$. Angles with respect to $g_0$ can be defined by
$$\cos(\angle_{g_0}(v,w))=\frac{\mathrm{Re}(q(v,w))}{\sqrt{q(v,v)q(w,w)}},\quad\forall v,w\in T_{[u]}M_0\setminus\{0\}.$$
For any $[u]\in M_0$ with $q(u,u)=1$ and any linearly independent $v,w\in T_{[u]}M_0$, we have the following classical result on curvatures in $M_0$.
\begin{prop}\cite[10.12 Proposition]{BH99}\label{KHn_info}
\begin{enumerate}
\item[(1).] If $v=w\lambda$ for some $\lambda\in\KK\setminus \RR$, then the sectional curvature $K_{M_0}(v,w)=-4$;
\item[(2).] If $q(v,w)\in\RR$, then the sectional curvature $K_{M_0}(v,w)=-1$.
\end{enumerate}
\end{prop}

Recall that in Section \ref{s4} we wrote the derivative map $d_\phi P$ as a composition of two operators $A_\phi^{-1}$ and $E_\phi$ for any $\phi\in\cB(R)$. A direct application of the above facts gives the following lemma on the operator $A_\phi$ when $g=g_0$.
\begin{lem}\label{hess}
Let $\Phi$ be as in the previous section. If $g=g_0$, then we have the following:
\begin{enumerate}
\item[(1).] We can compute the Hessian of $e^{2\Phi_s(x)}/2$:
$$\hess_{g_0}\frac{1}{2}e^{2\Phi_s(x)}=e^{2\Phi_s(x)}\left[g_0+\sum_{t=0}^{d-1}(d\Phi_s\circ \cJ_{t,g_0})^2\right].$$
The maps $\cJ_{t,g_0}:TM\to TM$ come from the complex structure and quaternionic structure on $M$. In other words, $\cJ_{0,g_0}=\id$, $\cJ_{1,g_0}(v)=v i$ when $\KK=\CC$ or $\HH$, $\cJ_{2,g_0}(v)=vj$ and $\cJ_{3,g_0}(v)=v k$ when $\KK=\HH$. 
\item[(2).] The operator $A_{x,s}$ defined in \eqref{Axs} has the following explicit formula:
\begin{align}\label{Axs_formula}
A_{x,s}(\xi)=\xi+\sum_{t=0}^{d-1}\langle\xi,\cJ_{t,g_0}\grad\Phi_s(x)\rangle \cJ_{t,g_0}\grad\Phi_s(x),\quad\forall \xi\in T_xM.
\end{align}
\end{enumerate}
\end{lem}
\begin{proof}
\begin{enumerate}
\item[(1).] This can be found in \cite[pp. 751]{BCG1}.
\item[(2).] By Lemma \ref{delta_X_density}, we have
$$\lambda(x,s)=e^{-(dn+d-2)\Phi_s(x)},\quad\forall x\in M,s\in S,$$
when $g=g_0$. Therefore, we can further simplify \eqref{Axs} to the following.
\begin{align*}
A_{x,s}(\xi)=&e^{-(dn+d)\Phi_s(x)}\lambda(x,s)^{-1}\nabla_\xi[e^{(dn+d)\Phi_s(x)}\lambda(x,s)\grad\Phi_s(x)] \\
=&e^{-2\Phi_s(x)}\nabla_\xi[e^{2\Phi_s(x)}\grad\Phi_s(x)] 
=e^{-2\Phi_s(x)}\nabla_\xi\grad\left( \frac{1}{2}e^{2\Phi_s(x)}\right).
\end{align*}
Hence for any vector $\eta\in T_xM$, we have
\begin{align*}
\langle A_{x,s}(\xi),\eta\rangle&=\frac{1}{2}\left\langle e^{-2\Phi_s(x)}\nabla_\xi\grad( e^{2\Phi_s(x)}),\eta\right\rangle \\
&=\frac{1}{2}e^{-2\Phi_s(x)}\left\langle \nabla_\xi\grad( e^{2\Phi_s(x)}),\eta\right\rangle \\
&=\frac{1}{2}e^{-2\Phi_s(x)}\hess_{g_0}e^{2\Phi_s(x)}(\xi,\eta) \\
&=\langle\xi,\eta\rangle+\sum_{t=0}^{d-1}\langle\xi,\cJ_{t,g_0}\grad\Phi_s(x)\rangle\langle\eta,\cJ_{t,g_0}\grad\Phi_s(x)\rangle \\
&=\left\langle\xi+\sum_{t=0}^{d-1}\langle\xi,\cJ_{t,g_0}\grad\Phi_s(x)\rangle \cJ_{t,g_0}\grad\Phi_s(x),\eta\right\rangle.
\end{align*}
Therefore,
\begin{align*}A_{x,s}(\xi)=\xi+\sum_{t=0}^{d-1}\langle\xi,\cJ_{t,g_0}\grad\Phi_s(x)\rangle \cJ_{t,g_0}\grad\Phi_s(x),\quad\forall \xi\in T_xM. \tag*{\qedhere}\end{align*}
\end{enumerate}
\end{proof}

The above result relies on the almost complex structure (almost quaternionic structure resp.) on $TM$ when $g=g_0$ and the explicit formulae for density functions $\lambda_g(x,s)$. In order to understand the more general case when $g$ is sufficiently close to $g_0$, similarly we construct an almost complex structure (almost quaternionic structure resp.) which preserves the Riemannian metric $g$. 
\begin{notation}
Let $V,W$ be any finite dimensional real vector spaces. Assume that $v=\{v_1,v_2,...,v_m\}$ and $w=\{w_1,w_2,...,w_l\}$ are bases for $V$ and $W$ respectively. Let $T:V\to W$ be a linear map. We denote by $_w[T]_v$ the matrix of $T$ under bases $v$ and $w$. In other words, if 
$$T(v_i)=\sum_{j=1}^la_{ji}w_j,\quad 1\leq i\leq m,$$
then
$$_w[T]_v=(a_{ij})_{1\leq i\leq m,~	1\leq j\leq l}.$$
\end{notation}
We start with a collection of vector fields $\{\widetilde\xi_{l,t,g_0}|1\leq l\leq n, 0\leq t\leq d-1\}$ such that for any $x\in \KK\mathbf{H}^n$, $\{\widetilde\xi_{l,t,g_0}(x)|1\leq l\leq n, 0\leq l\leq d-1\}$ is an orthonormal basis for $T_x\KK\mathbf{H}^n$ under the standard complex hyperbolic metric $g_0$. In other words, we have the following
$$\langle\widetilde\xi_{l_1,t_1,g_0}(x),\widetilde\xi_{l_2,t_2,g_0}(x)\rangle_{g_0}=\delta_{l_1l_2}\delta_{t_1t_2},$$
where $1\leq l_1,l_2\leq n$ and $0\leq t_1,t_2\leq d-1$.
In addition, we can further assume that these vector fields satisfy the following equations.
$$\cJ_{t,g_0}\widetilde\xi_{l,0,g_0}=\widetilde\xi_{l,t,g_0},\quad 1\leq s\leq n\mathrm{~and~}0\leq t\leq d-1.$$

We can therefore apply a Gram-Schmidt process to obtain a new collection of vector fields $\{\widetilde\xi_{l,t,g}|1\leq l\leq n, 0\leq t\leq d-1\}$ such that for any $x\in M=(\KK\mathbf{H}^n, g)$, $\{\widetilde\xi_{l,t,g}(x)|1\leq l\leq n, 0\leq t\leq d-1\}$ is an orthonormal basis for $T_xM$ under $g$. In other words, we have
\begin{align*}&\left(\widetilde\xi_{1,0,g},...,\widetilde\xi_{1,d-1,g},...,\widetilde\xi_{n,0,g},...,\widetilde\xi_{n,d-1,g}\right)=\left(\widetilde\xi_{1,0,g_0},...,\widetilde\xi_{1,d-1,g_0},...,\widetilde\xi_{n,0,g_0},...,\widetilde\xi_{n,d-1,g_0}\right)\cdot \cN_g,\end{align*}
where $\cN_g$ is a $dn\times dn$ upper triangular matrix recording the Gram-Schmidt process, and 
$$\langle\widetilde\xi_{l_1,t_1,g}(x),\widetilde\xi_{l_2,t_2,g}(x)\rangle=\delta_{l_1l_2}\delta_{t_1t_2},$$
where $1\leq l_1,l_2\leq n$ and $0\leq t_1,t_2\leq d-1$. 

We denote by $\widetilde\Xi_g$ the ordered basis $\{\widetilde\xi_{1,0,g},...,\widetilde\xi_{1,d-1,g},...,\widetilde\xi_{n,0,g},...,\widetilde\xi_{n,d-1,g}\}$ and for simplicity $\widetilde\Xi$ when there is no ambiguity of the metric. We construct an almost complex structure or almost quaternionic structure $\cJ_{t,g}$ on $TM$ such that 
$$_{\widetilde\Xi_g}[\cJ_{t,g}]_{\widetilde\Xi_g}=~_{\widetilde\Xi_{g_0}}[\cJ_{t,g_0}]_{\widetilde\Xi_{g_0}},\quad\forall 0\leq t\leq d-1.$$
It is clear that $\cJ_{t,g}$ preserves $g$. Moreover, we have the following lemma.
\begin{lem}\label{err_ACS}
If $\|g-g_0\|_{C^r}\leq \epsilon\ll 1$, then $\|\cJ_g-\cJ_{g_0}\|_{C^r}\leq \cK_0(n,r)\|g-g_0\|_{C^r}$ for some constant $\cK_0(n,r)$ depending only on $n$ and $r$.
\end{lem}
We define a new operator $\widehat A_{x,s}:T_xM\to T_xM$ as follows
\begin{align}\label{Axs_approx}
\widehat A_{x,s}(\xi)=\xi+\sum_{t=0}^{d-1}\langle\xi,\cJ_{t,g}\grad\Phi_s(x)\rangle \cJ_{t,g}\grad\Phi_s(x),\quad\forall \xi\in T_xM.
\end{align}
One should expect that $\widehat A_{x,s}$ is close to $A_{x,s}$ if $g$ is close to $g_0$ since $\widehat A_{x,s}$ and $A_{x,s}$ on $M$ (when the metric is $g$) are close to their counterparts on $M_0$ (when the metric is $g_0$) and $\widehat A_{x,s}=A_{x,s}$ when $g=g_0$. If $x=P(\phi)$, we have a corresponding $\widehat A_\phi:T_xM\to T_xM$ defined as
\begin{align}\label{A_approx}
\widehat A_\phi=\int_S\overline\rho_\phi(s)\widehat A_{x,s}d\mu_x(s).
\end{align}
$\widehat A_\phi$ is close to $A_\phi$ when $g$ is close to $g_0$. In particular, $\widehat A_\phi=A_\phi$ when $g=g_0$.

\subsection{The construction of $\widehat{A}_\phi$ when $\KK=\OO$ and $n=2$}\label{ss52}
Since $\OO$ is not associative, the models for complex hyperbolic and quaternionic hyperbolic spaces do not work for the Cayley hyperbolic space. In particular, the remark after Corollary \ref{curv_data_OH^2} suggests that we cannot find any fiberwise linear map $\cJ:TM_0\to TM_0$ such that for any $v\in TM_0$, the sectional curvature $K_{M_0}(v,\cJ(v))=-4$. Such $\cJ$ maps exist for complex and quaternionic hyperbolic spaces and was used in the constructions when $\KK\neq\OO$ (see Subsection \ref{ss51}). Hence we need a different way to construct $\widehat{A}_\phi$ in the Cayley hyperbolic setting.

Recall that in Proposition \ref{curv_OH^2} we define the \emph{Cayley line} of a vector $0\neq v=(a,b)\in\OO^2$ as
$$\Cay(v)=\begin{cases}\displaystyle \OO\cdot (1,a^{-1}b),\quad &a\neq 0;\\ \OO\cdot(0,1) ,&a=0.\displaystyle\end{cases}.$$
Let $F_4^{-20}=\cK\cA\cN$ be the Iwasawa decomposition of $F_4^{-20}$. Denoted by $v_{l,t}=(\delta_{1t}e_t,\delta_{2t}e_t)$ with $l=1,2$, $\delta_{lm}$ the Kronecker delta, $0\leq t\leq 7$ and $\{e_t\}_{0\leq t\leq7}$ the standard orthonormal basis for $\OO$. Then we can construct $\cA\cN$-invariant vector fields $\widetilde\xi_{l,t,g_0}$ such that  
$$\widetilde\xi_{l,t,g_0}(x_0)=\Psi(v_{l,t})\quad l=1,2\mathrm{~and~}0\leq t\leq 7,$$
where $\Psi$ is the same as $d\chi$ in Proposition \ref{curv_OH^2}. Define $\Psi_{x,g_0}:\OO^2\to T_xM_0$ such that
$$\Psi_{x,g_0}(v_{l,t})=\widetilde\xi_{l,t,g_0}(x),\quad l=1,2\mathrm{~and~}0\leq t\leq7.$$
Denoted by 
$$\Cay_{g_0}(\xi)=\Psi_{x,g_0}\left(\Cay(\Psi_{x,g_0}^{-1}(\xi))\right),\quad0\neq\xi\in T_xM_0$$
the \emph{Cayley line} containing $\xi$. Then following the fact that the Cayley hyperbolic space is Riemannian symmetric, the same arguments in Corollary \ref{curv_data_OH^2} can be applied to all points in $M_0$. Namely,
\begin{cor}
For any $x\in M_0$ and any non-zero $\xi,\xi'\in T_{x}M_0$.
\begin{enumerate}
\item[(1).] If $\xi,\xi'$ belong to the same Cayley line and $\xi\not\in\RR \xi'$, then the sectional curvature $K_{M_0}(\xi,\xi')=-4$;
\item[(2).] If $\Cay_{g_0}(\xi)\perp\Cay_{g_0}(\xi')$, then the sectional curvature $K_{M_0}(\xi,\xi')=-1$.
\end{enumerate}
\end{cor}
Therefore we have following the Cayley hyperbolic version of Lemma \ref{hess}.
\begin{lem}\label{hess_Cay}
Let $\Phi$ be as in the previous section and $\KK\mathbf{H}^n=\OO\mathbf{H}^2$. If $g=g_0$, then we have the following:
\begin{enumerate}
\item[(1).] We can compute the Hessian of $e^{2\Phi_s(x)}/2$:
$$\hess_{g_0}\frac{1}{2}e^{2\Phi_s(x)}=e^{2\Phi_s(x)}\left[g_0+g_0\left(\pi_{\Cay_{g_0}(\xi)}(\cdot),\pi_{\Cay_{g_0}(\xi)}(\cdot)\right)\right].$$
\item[(2).] The operator $A_{x,s}$ defined in \eqref{Axs} has the following explicit formula:
\begin{align}\label{Axs_formula_Cay}
A_{x,s}(\xi)=\xi+\pi_{\Cay_{g_0}(\grad\Phi_s(x))}(\xi),\quad\forall \xi\in T_xM,
\end{align}
where $\pi_{\Cay_{g_0}(\xi)}$ denotes the orthogonal projection onto $\Cay_{g_0}(\xi)$.
\end{enumerate}
\end{lem}
\begin{proof}
\begin{enumerate}
\item[(1).] See \cite[pp. 36]{CF} and \cite[pp. 47]{Eberlein96}. 
\item[(2).] Same as Lemma \ref{hess}.
\end{enumerate}
\end{proof}
As in the statement of the above lemma, we heavily used the concept of Cayley lines similar to the way we used almost complex and almost quatenionic structure in complex and quaternionic hyperbolic spaces. Therefore, similar to the previous subsection, the construction of the operator $\widehat{A}_\phi$ approximating $A_\phi$ relies on a notion of Cayley lines for $M$. 

We first apply a Gram-Schmidt process to $\{\tilde\xi_{l,t,g_0}\}$ with respect to the perturbed metric $g$ and obtain $\{\tilde\xi_{l,t,g}\}$ as a collection of orthonormal vector fields on $M$ which gives an orthonormal basis at every point in $M$. Define $\Psi_{x,g}:\OO^2\to T_xM$ such that
$$\Psi_{x,g}(v_{l,t})=\tilde\xi_{l,t,g}(x)\quad l=1,2\mathrm{~and~}0\leq t\leq 7.$$
For any $x\in M$ and $\xi\in T_xM$, we define the \emph{Cayley line} containing $\xi$ by
$$\Cay_g(\xi)=\Psi_{x,g}\left(\Cay(\Psi_{x,g}^{-1}(\xi))\right)\quad0\neq\xi\in T_xM.$$
Then we can mimic the case when $g=g_0$ and define a new operator $\widehat{A}_{x,s}:T_xM\to T_xM$ such that
\begin{align}\label{Axs_approx_Cay}
\widehat{A}_{x,s}(\xi)=\xi+\pi_{\Cay_g(\grad\Phi_s(x))}(\xi),\quad\forall \xi\in T_xM,
\end{align}
where $\pi_{\Cay_g(\grad\Phi_s(x))}$ refers to the orthogonal projection onto $\Cay_g(\grad\Phi_s(x))$. For any $\phi\in P^{-1}(\cB(R))$ and $x=P(\phi)$, we have a corresponding $\widehat{A}_\phi:T_xM\to T_xM$ defined as
\begin{align}\label{A_approx_Cay}
\widehat{A}_\phi=\int_S\overline{\rho}_\phi(s)\widehat{A}_{x,s}d\mu_x(s).
\end{align}
In the case when $g=g_0$, $\widehat{A}_\phi=A_\phi$ and they are close when $g$ is close to $g_0$ similar to the complex and quaternionic hyperbolic cases.

\subsection{$E_\phi$ and Jacobian inequalities}

\begin{notation}
For any $\phi\in\cB(R)$ fixed, we let $x=P(\phi)$. For simplicity we write $\xi_{l,t,g}=\widetilde\xi_{l,t,g}(x)$ for all $1\leq l\leq n$ and $0\leq t\leq d-1$. Denoted by $\Xi$ short for $\widetilde\Xi(x)$. Since $\cJ_{t,g}$ preserves $g$, $\cJ_{t,g}\Xi$ is also an (ordered) orthonormal basis of $T_xM$.

We define the following vectors in $T_\phi\cL=\cL$:
 $$X_{l,t,g}(s)=\langle\grad\Phi_s(x),\xi_{l,t,g}\rangle=d\Phi_s(\xi_{l,t,g}),\quad 1\leq l \leq n\mathrm{~and~}0\leq t\leq d-1.$$
They are linearly independent due to $\Xi$ being a basis and the second requirement of $G$ being a special Riemannian metric (see Definition \ref{fcnal_RM} and Definition \ref{RML}). Let $\cV_\phi:=\mathrm{span}\{X_{l,t,g}|1\leq l\leq n,0\leq t\leq d-1\}$. We write 
$$Z_\phi:=\{X_{1,0,g},...,X_{1,d-1,g},...,X_{n,0,g},...,X_{n,d-1,g}\}=d\Phi(\Xi),$$ which is a(n) (ordered) basis for $\cV_\phi$. 
Define $K_{x,g}\subset\SO(T_xM)$ such that
$$K_{x,g}=\{\Psi_{x,g}\circ T\circ \Psi_{x,g}^{-1}| T\in \Spin(9)\subset \SO(\OO^2)=\SO(16)\},$$
where the inclusion of $\Spin(9)\subset\SO(16)$ is given by the third assertion in Corollary \ref{curv_data_OH^2}.
\end{notation}

We are now ready to take a closer look at $\widehat A_\phi$ and $E_\phi$. Write 
$$J_t=~_{\Xi}[\id]_{\cJ_{t,g}\Xi},\quad\forall 1\leq t\leq d-1$$
and 
$$\widehat{Q}_\phi=~_\Xi[n(\widehat{A}_\phi-\id)]_\Xi.$$
Then we have the following

\begin{lem}\label{A_formula}
\begin{enumerate}
\item[(1).] When $\KK\neq\OO$, we have the following formula for $\widehat A_\phi$ defined in \eqref{A_approx}:
\begin{align*}
~_{\Xi} [\widehat{A}_\phi]_{\Xi}=\id+\frac{1}{n}\widehat{Q}_\phi=\id+\frac{1}{dn}\sum_{t=0}^{d-1}J_tQ_{\phi}J_t^{-1},
\end{align*}
where $Q_{\phi}=(Q_{ml,\phi})_{1\leq m,l\leq n}$ is a $dn\times dn$ real symmetric matrix such that 
\begin{align}\label{big_fat_matrix}
Q_{ml,\phi}
=\left(
\begin{array}{ccc}
\langle  X_{m,0,g}, X_{l,0,g}\rangle_{G_\phi}&... &\langle  X_{m,0,g}, X_{l,d-1,g}\rangle_{G_\phi}\\
\vdots&\ddots &\vdots\\
\langle  X_{m,d-1,g}, X_{l,0,g}\rangle_{G_\phi}&... & \langle  X_{m,d-1,g}, X_{l,d-1,g}\rangle_{G_\phi}
\end{array}
\right)
\end{align}
for any $1\leq m,l\leq n$;
\item[(2).] When $\KK=\OO$ and $n=2$, for any $0\neq\xi\in T_xM$, the normalized inner product $\langle\widehat{A}_\phi(\xi),\xi\rangle/\|\xi\|^2$ is constant along $\Cay_g(\xi)$. Equivalently, for any $O\in K_{x,g}$, the matrix $O^T\widehat{Q}_\phi O$ has the form
$$O^T\widehat{Q}_\phi O=\matii{\lambda\id&L}{L^T&\eta\id}\in\mathrm{Mat}_{16\times16}(\RR),$$
where $\lambda,\eta\in\RR_{\geq 0}$ and every block in the above matrix has size $8\times 8$.
\item[(3).] When $\KK=\OO$ and $n=2$, for any $0\neq\xi\in T_xM$, there exist orthogonal linear maps $I_{t,\xi}\in\SO(\Cay_g(\xi))$, $t=0,1,...,7$, such that $I_{0,\xi}=\id$ and
$$\widehat{Q}_{\phi}|_{\Cay_g(\xi)}=\frac{1}{8}\sum_{t=0}^7I_{t,\xi}Q_{\phi}|_{\Cay_g(\xi)}I_{t,\xi}^{-1}.$$
Equivalently, for any $O\in K_{x,g}$, write
$$O^T\widehat{Q}_\phi O=\matii{\widehat Q_{11,O,\phi}&\widehat Q_{12,O,\phi}}{\widehat Q_{21,O,\phi}&\widehat Q_{22,O,\phi}}\mathrm{~~and~~}O^TQ_\phi O=\matii{Q_{11,O,\phi}&Q_{12,O,\phi}}{Q_{21,O,\phi}&Q_{22,O,\phi}},$$
where every block in the above matrix has size $8\times 8$. Then there exist $I_{t,O,\phi}\in\SO(8)$, $0\leq t\leq 7$, such that $I_{0,O,\phi}=\id$ and
$$\widehat{Q}_{11,O,\phi}=\frac{1}{8}\sum_{t=0}^7I_{t,O,\phi}Q_{11,O,\phi}I_{t,O,\phi}^{-1}.$$
\end{enumerate}
\end{lem}
\begin{proof}
\begin{enumerate}
\item[(1).] Since $\Xi$ is an orthonormal basis for $T_xM$, it suffices to calculate elements of the form
$$\langle \widehat{A}_{\phi}(\xi),\eta\rangle, \quad \xi,\eta\in\Xi.$$
Notice that
\begin{align*}
\langle \widehat{A}_\phi(\xi_{l_1,t_1,g}),\xi_{l_2,t_2,g}\rangle
=&\left\langle \int_S\overline\rho_\phi(s)\widehat{A}_{x,s}(\xi_{l_1,t_1,g})d\mu_x(s),\xi_{l_2,t_2,g}\right\rangle\\
=&\int_S\overline\rho_\phi(s)\langle\xi_{l_1,t_1,g} ,\xi_{l_2,t_2,g}\rangle d\mu_x(s) \\
&+\sum_{t=0}^{d-1}\int_S\overline\rho_\phi(s)\langle\xi_{l_1,t_1,g},\cJ_{t,g}\grad\Phi_s(x)\rangle\langle\xi_{l_2,t_2,g},\cJ_{t,g}\grad\Phi_s(x)\rangle d\mu_x(s) \\
=&\int_S\overline\rho_\phi(s)\langle\xi_{l_1,t_1,g} ,\xi_{l_2,t_2,g}\rangle d\mu_x(s) \\
&+\sum_{t=0}^{d-1}\int_S\overline\rho_\phi(s)\langle\cJ_{t,g}\xi_{l_1,t_1,g},\grad\Phi_s(x)\rangle\langle\cJ_{t,g}\xi_{l_2,t_2,g},\grad\Phi_s(x)\rangle d\mu_x(s) \\
=&\delta_{l_1l_2}\delta_{t_1t_2}+\frac{1}{dn}\sum_{t=0}^{d-1}\langle d\Phi(\cJ_{t,g}\xi_{l_1,t_1,g}),d\Phi(\cJ_{t,g}\xi_{l_2,t_2,g})\rangle_{G_\phi}.
\end{align*}
Therefore the first assertion follows.
\item[(2).] Recall the definition of $\widehat{A}_\phi$ in \eqref{Axs_approx_Cay} and \eqref{A_approx_Cay}, it suffices to show that
\begin{align}\label{Q_diag}
\langle\xi,\pi_{\Cay_g(\eta)}(\xi)\rangle/\|\xi\|^2
\end{align} 
is constant along $\Cay_g(\xi)$ for any choice of $\xi,\eta\in T_xM$. By applying $K_{x,g}$ actions we can assume without loss of generality that $\xi=\Psi_{x,g}(1,0)$ and $\eta=\Psi_{x,g}(b\cos\theta, ba\sin\theta)$, where $\theta\in\RR$ and $a,b$ are unit octonions. Then for any unit octonion $c$ and any $\xi_c=\Psi_{x,g}(c,0)$, we have
$$\pi_{\Cay_g(\eta)}(\xi_c)=\Psi_{x,g}(c\cos^2\theta,ca\sin\theta\cos\theta).$$
Hence
\begin{align*}
\langle \xi_c,\pi_{\Cay_{g}(\eta))}(\xi_c)\rangle=\langle (c,0),(c\cos^2\theta,ca\sin\theta\cos\theta)\rangle=\cos^2\theta,
\end{align*}
which proves that \eqref{Q_diag} is constant along $\Cay_g(\xi)$.
\item[(3).] The computations in the first assertion imply that the operator on $T_xM$ defined by
$$\xi\to \int_S\overline\rho_\phi(s)\pi_{\grad\Phi_s(x)}(\xi)d\mu_x(s)$$
has matrix expression exactly equal to $Q_\phi/16$ under the basis $\Xi$, where $\pi_{\grad\Phi_s(x)}$ is the orthogonal projection onto $\RR\grad\Phi_s(x)$. It suffices to show that for any $0\neq\xi,\eta\in T_xM$, there exists some $I_{t,\xi,\phi}\in\SO(\Cay_g(\xi))$ with $0\leq t\leq 7$ such that
\begin{align}\label{local_avg_Cay}
\langle\xi',\pi_{\Cay_g(\eta)}(\xi')\rangle=\sum_{t=0}^7\langle I_{t,\xi,\phi}(\xi'),\pi_{\eta}(I_{t,\xi,\phi}(\xi'))\rangle,\quad\forall \xi'\in\Cay_g(\xi).
\end{align}
By applying $K_{x,g}$ actions we can assume without loss of generality that $\xi=\Psi_{x,g}(1,0)$ and $\eta=\Psi_{x,g}(b\cos\theta,ba\sin\theta)$, where $a,b$ are unit octonions. Let $e_0=1,e_1,...,e_7$ be the standard orthonormal basis for $\OO$. Then 
$$\{\Psi_{x,g}(\overline{e}_tb\cos\theta,(\overline{e}_tb)a\sin\theta):0\leq t\leq 7\}$$ 
is an orthonormal basis for $\Cay_g(\eta)$. Hence
$$\langle\xi',\pi_{\Cay_g(\eta)}(\xi')\rangle=\sum_{t=0}^7\langle\xi',\pi_{\Psi_{x,g}(\overline{e}_tb\cos\theta,(\overline{e}_tb)a\sin\theta)}\xi'\rangle, \forall \xi'\in\Cay_g(\xi).$$
Notice that any $\xi'$ has the form $\Psi_{x,g}(c,0)$ for some unit octonion $c$ and
\begin{align*}
&\langle\xi',\pi_{\Psi_{x,g}(\overline{e}_tb\cos\theta,(\overline{e}_tb)a\sin\theta)}\xi'\rangle\\
=&\langle\Psi_{x,g}(c,0),\Psi_{x,g}(\overline{e}_tb\cos\theta,(\overline{e}_tb)a\sin\theta)\rangle^2 \\
=&\langle c,\overline{e}_tb\cos\theta\rangle^2 \\
=&\langle e_tc,b\cos\theta\rangle^2 
=\langle \Psi_{x,g}(e_tc,0),\eta\rangle^2 
=\langle \Psi_{x,g}(e_t\Psi_{x,g}^{-1}(\xi')),\pi_\eta(\Psi_{x,g}(e_t\Psi_{x,g}^{-1}(\xi')))\rangle,\quad 0\leq t\leq 7.
\end{align*}
Choose $I_{t,\xi}(\cdot)=\Psi_{x,g}(e_t\Psi_{x,g}^{-1}(\cdot))$ and \eqref{local_avg_Cay} follows.\qedhere
\end{enumerate}
\end{proof}
\begin{rmk}
A similar version of the second and the third assertions also hold for complex and quaternionic hyperbolic spaces. In particular, we can choose the orthogonal maps $I_{t,\xi}$ in the third assertion to be restrictions of $\cJ_{t,g}$ onto $\KK$-lines when $\KK\neq \OO$. It is not hard to verify that the first assertion implies the counterparts for the second and third assertion in the complex and quaternionic case. However, due to the remark of Corollary \ref{curv_data_OH^2}, there is no obvious way to find $\cJ_{t,g}$ due to non-associativity of octonionic multiplication, making the Cayley hyperbolic case more complicated.
\end{rmk}

In order to understand $E_\phi:T_\phi\cL\to T_xM$, recall that $\cV_\phi=\mathrm{span}Z_\phi$. We first claim that
$$E_\phi|_{\cV_\phi^{\perp}}=0,$$
where $\cV^{\perp}_\phi$ denotes the orthogonal complement in $T_\phi\cL=\cL$ with respect to $G_\phi$. This is because for any $X\in \cV^\perp_\phi$ and any $\xi_{l,t,g}\in\Xi$,
\begin{align*}
\langle E_\phi(X),\xi_{l,t,g}\rangle=&(dn+d)\int_S\overline\rho_\phi(s)X(s)\langle\grad\Phi_s(x),\xi_{l,t,g}\rangle d\mu_x(s) \\
=&(dn+d)\int_S\overline\rho_\phi(s)X(s)X_{l,t,g}(s) d\mu_x(s) 
=\frac{n+1}{n}\langle X,X_{l,t,g}\rangle_{G_\phi}=0.
\end{align*}
Therefore
$$E_\phi|_{\cV^{\perp}_\phi}=0$$
since $\Xi$ is a basis for $T_xM$.

Now it suffices to understand $E_\phi|_{\cV_\phi}$. Let $W_\phi$ be an ordered orthonormal basis for $\cV_\phi$. We have the following result.
\begin{lem}\label{E_formula}
$$~_{\Xi}[E_\phi|_{\cV_\phi}]_{Z_\phi}=\frac{n+1}{n}Q_{\phi}$$
and
$$~_{\Xi}[E_\phi|_{\cV_\phi}]_{W_\phi}=\frac{n+1}{n}~_{W_\phi} [\id]_{Z_\phi}^T,$$
where $_{W_\phi} [\id]_{Z_\phi}^T$ denotes the transpose of $_{W_\phi} [\id]_{Z_\phi}$.
\end{lem}
\begin{proof}
Applying a change of basis, we have
$$_{\Xi}[E_\phi|_{\cV_\phi}]_{W_\phi}=~_{\Xi}[E_\phi|_{\cV_\phi}]_{ Z_\phi}\cdot~_{Z_\phi}[\id]_{W_\phi}.$$
Since $\Xi$ is orthonormal in $T_xM$, $_{\Xi}[E_\phi|_\cV]_{Z_\phi}$ has its entries in the form of $\langle E_\phi(X),\xi\rangle$, where $X\in Z_\phi$ and $\xi\in\Xi$.
Notice that 
\begin{align*}
\langle E_\phi(X_{l_1,t_1,g}),\xi_{l_2,t_2,g}\rangle=&(dn+d)\int_S\overline\rho_\phi(s)X_{l_1,t_1,g}(s)\langle\grad\Phi_s(x),\xi_{l_2,t_2,g}\rangle d\mu_x(s) \\
=&(dn+d)\int_S\overline\rho_\phi(s)X_{l_1,t_1,g}(s)X_{l_2,t_2,g}(s) d\mu_x(s) 
=\frac{n+1}{n}\langle X_{l_1,t_1,g},X_{l_2,t_2,g}\rangle_{G_\phi}.
\end{align*}
Therefore by \eqref{big_fat_matrix}
$$_{\Xi}[E_\phi|_{\cV_\phi}]_{ Z_\phi}=\frac{n+1}{n}Q_{\phi}.$$
Since $W_\phi$ is orthonormal in $\cV_\phi$, we have
\begin{align*}
Q_{\phi} 
=&\left(
\begin{array}{ccc}
\langle  X_{m,0,g}, X_{l,0,g}\rangle_{G_\phi}&... &\langle  X_{m,0,g}, X_{l,d-1,g}\rangle_{G_\phi}\\
...&... &...\\
\langle  X_{m,d-1,g}, X_{l,0,g}\rangle_{G_\phi}&... & \langle  X_{m,d-1,g}, X_{l,d-1,g}\rangle_{G_\phi}
\end{array}
\right)_{1\leq m,l\leq n}
=~_{W_\phi}[\id]_{Z_\phi}^T\cdot~_{W_\phi}[\id]_{Z_\phi}.
\end{align*}
Therefore
\begin{align*}
_{\Xi}[E_\phi|_{\cV_\phi}]_{W_\phi}=&~_{\Xi}[E_\phi|_{\cV_\phi}]_{Z_\phi}\cdot~_{Z_\phi}[\id]_{W_\phi} \\
=&\frac{n+1}{n}~_{W_\phi}[\id]_{Z_\phi}^T\cdot~_{W_\phi}[\id]_{ Z_\phi}\cdot~_{Z_\phi}[\id]_{W_\phi} 
=\frac{n+1}{n}~_{W_\phi}[\id]_{Z_\phi}^T\cdot~_{W_\phi}[\id]_{W_\phi}
=\frac{n+1}{n}~_{W_\phi}[\id]_{Z_\phi}^T.\qedhere
\end{align*}
\end{proof}
\begin{rmk}
For simplicity we write $U_\phi=~_{W_\phi}[\id]_{ Z_\phi}$ and hence $Q_{\phi}=U_\phi^TU_\phi$.  
\end{rmk}
\begin{notation}
For any $m\in\ZZ_+$ and any matrix $H\in\mathrm{Mat}_{m\times m}(\RR)$, we define its norm by
$$\|H\|=\left(\sup_{v\neq 0,v\in\mathrm{Mat}_{m\times 1}(\RR)=\RR^m}\frac{v^TH^THv}{v^Tv}\right)^{\frac{1}{2}}=\sup_{\|v\|=1}\|Hv\|.$$
where $\|v\|=\sqrt{v^Tv}$ denotes the Euclidean norm of a real vector $v\in\mathrm{Mat}_{m\times 1}(\RR)$. In particular, $H$ is $\|H\|$-Lipschitz, regarded as an endomorphism on a finite dimensional Euclidean space (with a prescibed orthonormal basis).
\end{notation}
Before estimating determinants, we first recall a linear algebra fact from \cite[B.2 Lemme]{BCG1}.

\begin{lem}[Besson-Courtois-Gallot, \cite{BCG1}]\label{det_sum_ineq}
The determinant function is log-concave on positive semi-definite matrices with real entries. As a corollary, let $A_1,A_2,...,A_n$ be real positive definite $m\times m$ matrices such that $\det(A_1)=\det(A_2)=...=\det(A_n)=c>0$. Then
$$\det\left(\frac{1}{n}\sum_{j=1}^nA_j\right)\geq c$$
with equality holds only when $A_1=A_2=...=A_n$.
\end{lem}

\begin{lem}\label{easy_ineq}
\begin{enumerate}
\item[(1).]$$\det\left(\id+\frac{1}{n}\widehat{Q}_\phi\right)\geq\det\left(\id+\frac{1}{n}U_\phi^TU_\phi\right);$$
\item[(2).]The trace of $\widehat{Q}_\phi$ and $U_\phi^TU_\phi$ are both equal to $dn$. Moreover, let $0\leq\eta_1\leq...\leq\eta_{dn}$ be eigenvalues of $\widehat{Q}_\phi$ and $0\leq\lambda_1\leq...\leq\lambda_{dn}$ be eigenvalues of $U_\phi^TU_\phi$, then 
\begin{align}\label{eigenval_ineq}
\lambda_1\leq\eta_1\leq1\leq\eta_{dn}\leq\lambda_{dn}\leq dn.
\end{align}
\item[(3).]$$\det\left(\id+\frac{1}{n}U_\phi^TU_\phi\right)\geq\ \left(\frac{n+1}{n}\right)^{dn}\det(U_\phi^TU_\phi)^\frac{1}{n+1}$$
and $$\det(U_\phi^TU_\phi)\leq \id.$$
Equality in both inequalities holds if and only if $U_\phi^TU_\phi=\id$.
\item[(4).]  $\|U_\phi\|,\|U^T_\phi\|\leq\sqrt{dn}$.
\end{enumerate}
\end{lem}
\begin{proof}
\begin{enumerate}
\item[(1).] When $\KK\neq\OO$, this follows from Lemma \ref{det_sum_ineq} and the first assertion in Lemma \ref{A_formula}. It remains to prove for the Cayley hyperbolic case. Let $\xi$ be a unit eigenvector of $\widehat{A}_\phi$ having the largest eigenvalue denoted by $\eta_{\max}$. Since the second assertion in Lemma \ref{A_formula} implies that $\langle\xi,\widehat{A}_\phi(\xi)\rangle/\|\xi\|^2$ is constant on $\Cay_g(\xi)$, all non-zero vectors in $\Cay_g(\xi)$ are eigenvalues of $\widehat{A}_\phi$ with eigenvalue $\eta_{\max}$. Let $\Cay_g(\xi')$ be the orthogonal complement of $\Cay_g(\xi)$. Since $\widehat{A}_\phi$ is self-adjoint, $\Cay_g(\xi')$ is an invariant subspace of $\widehat{A}_\phi$ and non-zero vectors in this subspace are all eigenvectors of $\widehat{A}_\phi$ for some common eigenvalue $\eta_{\min}$ due to the second assertion in Lemma \ref{A_formula}. Therefore there exists some $O\in K_{x,g}\subset\SO(T_xM)$ such that 
$$O^T\left(\id+\frac{1}{n}\widehat{Q}_\phi\right)O=\matii{\eta_{\max}\id&0}{0&\eta_{\min}\id},$$
where every block in the above matrix has size $8\times 8$. The assertion then follows from Lemma \ref{det_sum_ineq}, the third assertion in Lemma \ref{A_formula} and the fact that for any positive semi-definite matrix $\matii{A&B}{B^T&C}$ with $A,B,C\in\mathrm{Mat}_{m\times m}(\RR)$, we have
$$\det\matii{A&B}{B^T&C}\leq\det(A)\det(C).$$
\item[(2).] Notice that
\begin{align*}
\tr (U_\phi^TU_\phi)
=&\sum_{l=1}^n\sum_{t=0}^{d-1}\langle X_{l,t,g},X_{l,t,g}\rangle_{G_\phi} \\
=&\sum_{l=1}^n\sum_{t=0}^{d-1}dn\int_S\overline\rho_\phi(s)\langle\grad\Phi_s,\xi_{l,t,g}\rangle^2d\mu_x(s) \\
=&dn\int_S\overline\rho_\phi(s)\langle\grad\Phi_s,\grad\Phi_s\rangle^2d\mu_x(s) 
=dn\int_S\overline\rho_\phi(s)d\mu_x(s)=dn.
\end{align*}
Therefore the trace of $\widehat{Q}_\phi$ and $U_\phi^TU_\phi$ are both equal to $dn$, which follows from Lemma \ref{A_formula}. The fact that $\lambda_1\leq\eta_1\leq\eta_{dn}\leq\lambda_{dn}$ follows directly from Lemma \ref{A_formula} and two linear algebra facts listed below:
\begin{enumerate}
\item[(a).] Let $A$ be a positive semi-definite matrix and $O_j$ be orthogonal matrices with $1\leq j\leq m$. Denoted by $k_{\max}$, $k_{\min}$ the largest and the smallest eigenvalues of $A$. Let $l_{\max}$, $l_{\min}$ be the largest and the smallest eigenvalues of $\left(\sum_{j=1}^m O_jAO_j^T\right)/m$. Then $k_{\min}\leq l_{\min}\leq l_{\max}\leq k_{\max}$. (This is used when $\KK=\CC$, $\HH$ or $\OO$.)
\item[(b).] Let $\matii{A&B}{B^T&C}$ be a positive semi-definite matrix with $A,B,C\in \mathrm{Mat}_{m\times m}(R)$. Let $k_{\max}$, $k_{\min}$ be the largest and the smallest eigenvalues of $\matii{A&B}{B^T&C}$ respectively. Denoted by $l_{\max}$, $l_{\min}$ the largest and the smallest eigenvalues of $\matii{A&0}{0&C}$ respectively. Then $k_{\min}\leq l_{\min}\leq l_{\max}\leq k_{\max}$. (This is used when $\KK=\OO$.)
\end{enumerate}
\item[(3).] 
Following the notation in the second assertion, we have
\begin{align}\label{trace}
\sum_{t=1}^{dn}\lambda_t=dn.
\end{align}
Therefore by Lemma \ref{det_sum_ineq}
\begin{align*}
\det\left(\id+\frac{1}{n}U^T_{\phi}U_\phi\right)
=\prod_{l=1}^{dn}\left(1+\frac{1}{n}\lambda_l\right) 
=&\prod_{l=1}^{dn}\frac{n+1}{n}\cdot\frac{n\cdot 1+\lambda_l}{n+1} \\
\geq&\left(\frac{n+1}{n}\right)^{dn}\prod_{l=1}^{dn}(1^{n}\cdot\lambda_l)^{\frac{1}{n+1}} \\
=&\left(\frac{n+1}{n}\right)^{dn}\left(\prod_{l=1}^{dn}\lambda_l\right)^{\frac{1}{n+1}} 
=\left(\frac{n+1}{n}\right)^{dn}\det(U_\phi^TU_\phi)^{\frac{1}{n+1}}.
\end{align*}
The inequality $\det(U_\phi^TU_\phi)\leq \id$ follows directly by applying the arithmetic mean-geometric mean inequality to \eqref{trace}. Equal signs in both inequalities are achieved if and only if all $\lambda_l=1$. In other words, $U_\phi^TU_\phi=\id$.  
\item[(4).] By \eqref{trace}, $\|U_\phi^TU_\phi\|\leq dn$. Hence $\|(U_\phi^TU_\phi)^\frac{1}{2}\|\leq\sqrt{dn}$. Since $U_\phi=K(U_\phi^TU_\phi)^\frac{1}{2}$ for some $K\in \mathrm{O}(dn)$. Therefore, by the fact that any element in $\mathrm{O}(dn)$ has norm $1$, $\|U_\phi\|,\|U_\phi^T\|\leq\sqrt{dn}$. \qedhere
\end{enumerate}
\end{proof}
\begin{cor}\label{AE_det_ineq}
The maps $E_\phi $ and $\widehat{A}_\phi^{-1}\circ E_\phi:T_\phi\cL\to T_xM$ is $(n+1)\sqrt{d/n}$-Lipschitz and the latter one is volume non-increasing. In particular, when $g=g_0$, $P$ is $(n+1)\sqrt{d/n}$-Lipschitz and volume non-increasing and the first two assertion in Proposition \ref{key_prop} holds when $g=g_0$.
\end{cor}
\begin{proof}

Since $\id+\frac{1}{n}\widehat Q_\phi\geq\id$, by Lemma \ref{A_formula} $\widehat{A}_\phi^{-1}$ is $1$-Lipschitz. It follows from Lemma \ref{E_formula} and the second assertion in Lemma \ref{easy_ineq} that $E_\phi$ is $(n+1)\sqrt{d/n}$-Lipschitz. Hence their composition $\widehat{A}_\phi^{-1}\circ E_\phi:T_\phi\cL\to T_xM$ is $(n+1)\sqrt{d/n}$-Lipschitz.

It remains to show that 
$$\Jac_G(\widehat{A}_\phi^{-1}\circ E_\phi)=\left|\det \widehat{A}_\phi^{-1}\right|J_GE_\phi\leq 1.$$
Since $E_\phi|_{\cV^{\perp}_\phi}=0$, we have $E_\phi=E_\phi|_{\cV_\phi}\circ P_{\cV_\phi}$, where $P_{\cV_\phi}:\cL\to\cV_\phi$ is the orthogonal projection onto $\cV_\phi$. Hence 
$$\Jac_G E_\phi=\Jac_{G,\cV_\phi}E_\phi\cdot \Jac_GP_{\cV_\phi}.$$
We claim that $\Jac_GP_{\cV_\phi}=1$. This is because on the one hand, $\Jac_{G,\cV_\phi}P_{\cV_\phi}=\det\id_{\cV_\phi}=1$. On the other hand, for any arbitary $dn$-dimensional subspace $\cW\subset\cL$ with an orthonormal basis $\{\widetilde X_1,\widetilde X_2,...,\widetilde X_{dn}\}$, we have the following
\begin{align}\label{proj_matrix_jacobian_est}
\Jac_{G,\cW}P_{\cV_\phi} =&\Jac_G(P_{\cV_\phi}|_\cW)
\leq\sqrt{\left(\frac{1}{dn}\sum_{i=1}^{dn}\|P_{\cV_\phi}(\widetilde X_i)\|_{G_\phi}^2\right)^{dn}}
\leq\sqrt{\left(\frac{1}{dn}\sum_{i=1}^{dn}\|\widetilde X_i\|_{G_\phi}^2\right)^{dn}} =\sqrt{1^{dn}}=1.
\end{align}
The second inequality sign follows from the fact that $P_{\cV_\phi}$ is $1$-Lipschitz and the first inequality sign follows from the following fact from linear algebra: For any $m\times m$ real matrix $H=(h_1~~h_2~~...~~h_m)$, assuming $k_1,k_2,...,k_{m}\geq 0$ are eigenvalues of $H^TH$, we have
\begin{align*}
|\det H|=&\sqrt{\det H^TH}
=\sqrt{\prod_{i=1}^{m}k_i}
\leq\sqrt{\left(\frac{1}{m}\sum_{i=1}^mk_i\right)^{m}} 
=\sqrt{\left(\frac{1}{m}\sum_{i=1}^{m}h_i^Th_i\right)^{m}}.
\end{align*}
Therefore 
$$\Jac_GE_\phi=\Jac_{G,\cV_\phi}E_\phi=\left|\det \frac{n+1}{n}U^T_\phi\right|=\left(\frac{n+1}{n}\right)^{dn}(\det{U_\phi^TU_\phi)}^{\frac{1}{2}}.$$
On the other hand 
$$|\det\widehat{A}_\phi^{-1}|=\left|\det ~_{\Xi}[\widehat{A}_\phi]_{\Xi}\right|^{-1}=\det\left(\id+\frac{1}{n}\widehat{Q}_\phi\right)^{-1}.$$
Hence by the first assertion in Lemma \ref{easy_ineq},
\begin{align}\label{AE_det_est}
\Jac_G(\widehat{A}_\phi^{-1}\circ E_\phi)=&\frac{\left(\frac{n+1}{n}\right)^{dn}(\det{U_\phi^TU_\phi)}^{\frac{1}{2}}}{\det\left(\id+\frac{1}{n}\widehat{Q}_\phi\right)} \nonumber\\
\leq&\frac{\left(\frac{n+1}{n}\right)^{dn}(\det{U_\phi^TU_\phi)}^{\frac{1}{2}}}{\left(\frac{n+1}{n}\right)^{dn}(\det{U_\phi^TU_\phi)}^{\frac{1}{n+1}}}=(\det{U_\phi^TU_\phi)}^{\frac{1}{2}-\frac{1}{n+1}}\leq 1^{\frac{1}{2}-\frac{1}{n+1}}=1.
\end{align}
In particular, when $g=g_0$, $A_\phi=\widehat{A}_\phi.$ Hence the above result implies that $\Jac_GP\leq 1$. 
\end{proof}

\section{Estimating the correction factor}\label{s6}
In the previous section, we introduced an operator $\widehat{A}_\phi$ for any $\phi\in P^{-1}(B_{x_0}(R_1))$ which has an explicit formula and coincides with $A_\phi$ when $g=g_0$. We will first estimate the difference of these two operators.

\begin{notation}
Denoted by $\epsilon=\epsilon(g,r)$ the $C^r$-norm of $g-g_0$.

Let $x=P(\phi)$. Define 
$$\phi_t(s)=\Phi_s(x)-\frac{1}{dn+d}\ln\left(1-t+te^{(dn+d)(\Phi_s(x)-\phi(s))}\right).$$
Then
$$e^{(dn+d)(\Phi_s(x)-\phi_t(s))}=1-t+t\rho_\phi(s)=1+t(\rho_\phi(s)-1).$$

\end{notation}
\begin{lem}\label{prelim_est}
For any $\phi\in\cB(R)$, we have the following facts
\begin{enumerate}
\item[(1).] $P(\phi_t)=P(\phi)=x$ for any $\phi_t$ in the domain of $P$;
\item[(2).] $A_{\phi_0}=\widehat{A}_{\phi_0}=\frac{n+1}{n}\id$;
\item[(3).] $\left.\frac{d}{dt}\right|_{t=0}\det A_{\phi_t}=\left.\frac{d}{dt}\right|_{t=0}\det \widehat{A}_{\phi_t}=0$.
\end{enumerate}
\end{lem}
\begin{proof}
\begin{enumerate}
\item[(1).] Let $x=P(\phi)$. Notice that 
\begin{align*}
\Omega_{\phi_t,g}(x)=&\int_Se^{(dn+d)(\Phi_s(x)-\phi_t(s))}d \Phi_s(x)d\mu_x(s) \\
=&\int_S(1-t+t\rho_{\phi}(s))d \Phi_s(x)d\mu_x(s)  
=\int_St\rho_{\phi}(s)d \Phi_s(x)d\mu_x(s)=t\Omega_{\phi,g}(x)=0.
\end{align*}
Hence $P(\phi_t)=P(\phi)=x$ by Definition \ref{proj}. In particular,
$$\rho_{\phi_t}(s)=1-t+t\rho_\phi(s)=1+t(\rho_\phi(s)-1).$$
\item[(2).] Since $\phi_0=\Phi_s(x)$, we have $\rho_{\phi_0}\equiv 1$, which implies that $U_{\phi_0}^TU_{\phi_0}=\id$ given the formula \eqref{big_fat_matrix}. Therefore Lemma \ref{A_formula} and Lemma \ref{easy_ineq} imply that $\widehat A_{\phi_0}=\frac{n+1}{n}\id$. Since Lemma \ref{E_formula} suggests that $E_{\phi_0}|_{\cV_{\phi_0}}$ is surjective and $d_x\Phi(\Xi)= Z_{\phi_0}$ by their definition in the previous section (equal in an order preserving way),  it suffices to show that 
$$\id=\frac{n}{n+1}E_{\phi_0}\circ d_x\Phi,$$
given that $\id=d_{\phi_0} P\circ d_x\Phi=A_{\phi_0}^{-1}\circ E_{\phi_0}\circ d_x\Phi$ from Definition \ref{Projection}, Lemma \ref{RM_verification} and Lemma \ref{AE_decomp}.

By Lemma \ref{E_formula} we have 
$$_{\Xi}[E_{\phi_0}]_{Z_{\phi_0}}=\frac{n+1}{n}U_{\phi_0}^TU_{\phi_0}=\frac{n+1}{n}\id.$$ 
Therefore $\frac{n}{n+1}E_{\phi_0}(Z_{\phi_0})=\Xi$ (equal in an order-preserving way) and 
$$\frac{n}{n+1}E_{\phi_0}\circ d_x\Phi(\Xi)=\frac{n}{n+1}E_{\phi_0}(Z_{\phi_0})=\Xi$$
(equal in an order preserving way), which implies that 
$$\id=\frac{n}{n+1}E_{\phi_0}\circ d_x\Phi.$$
\item[(3).] Let $\psi$ be an arbitary element in $\cB(R)$. We assume that $\widehat{Q}_\psi$ has eigenvalues $0\leq\eta_1\leq\eta_2\leq...\leq\eta_{dn}$. Lemma \ref{easy_ineq} shows that 
$$\sum_{s=1}^{dn}\eta_s=dn.$$
Hence 
\begin{align*}
\det\widehat{A}_\psi =&\det \left(\id+\frac{1}{n}\widehat{Q}_\psi\right)=\prod_{t=1}^{dn}\left(1+\frac{1}{n}\eta_t\right)
\leq\left[\frac{1}{dn}\sum_{t=1}^{dn}\left(1+\frac{1}{n}\eta_t\right)\right]^{dn}=\left(\frac{n+1}{n}\right)^{dn}.
\end{align*}
Lemma \ref{easy_ineq} implies that the equality holds when $U_\psi^T U_\psi=\id$. In particular 
$$\det \widehat{A}_{\phi_0}=\left(\frac{n+1}{n}\right)^{dn}.$$
Therefore 
$$\left.\frac{d}{dt}\right|_{t=0}\det \widehat{A}_{\phi_t}=0.$$
Notice that $d_{\phi_t}P=A_{\phi_t}^{-1}\circ E_{\phi_t}$. We have
$$\det A_{\phi_t}=\frac{\Jac_GE_{\phi_t}}{\Jac_GP(\phi_t)}.$$
Lemma \ref{E_formula} and the third assertion in Lemma \ref{easy_ineq} imply that $\Jac_GE_{\phi_t}$ attains maximum at $t=0$. On the other hand, for any $\xi\in T_x M$, 
\begin{align*}
\left\langle\left.\frac{d}{dt}\right|_{t=0}\phi_t(s),\langle\grad\Phi_s(x),\xi\rangle\right\rangle_{G_{\Phi(x)}}=&\frac{-dn}{dn+d}\int_S(\rho_\phi(s)-1)\langle\grad\Phi_s(x),\xi\rangle d\mu_x(s) \\
=&\frac{-dn}{dn+d}\left[\Omega_{\phi,g}(x)(\xi)-\Omega_{\Phi_s(x),g}(x)(\xi)\right] \\
=&\frac{-dn}{dn+d}\left[\Omega_{\phi,g}(P(\phi))(\xi)-\Omega_{\Phi_s(x),g}(P(\Phi_s(x)))(\xi)\right]=0,
\end{align*}
where the last three equal signs follow from Definition \ref{proj} and the first assertion in Lemma \ref{well_defined_P}. Therefore $\left.\frac{d}{dt}\right|_{t=0}\phi_t(s)\in T_{\Phi_s(x)}\cL$ is perpendicular to $T_x\Phi=\Phi_*(T_xM)$. By Proposition \ref{crit_jacobian_cond} and Lemma \ref{RM_verification}, $t=0$ is a critical point for $\Jac_GP(\phi_t)$. Hence
\begin{align*}\left.\frac{d}{dt}\right|_{t=0}\det A_{\phi_t}=0.\tag*{\qedhere}\end{align*}
\end{enumerate}
\end{proof}

\begin{prop}\label{A_err}
There exist positive constants $r$ and $C_0=C_0(n,R)$ such that for any $g$ sufficiently close to $g_0$ (without loss of generality, assume $\epsilon<1$) and every $\phi\in\cB(R)$ one has
$$\|A_\phi-\widehat{A}_\phi \| \leq C_0\epsilon\|\phi-\Phi(P(\phi))\|_{L^2(S)}$$
and
$$|\det A_\phi-\det\widehat{A}_\phi | \leq C_0\epsilon\|\phi-\Phi(P(\phi))\|_{L^2(S)}^2.$$
\end{prop}
\begin{proof}
Let $$b(t)=\int_S\rho_{\phi_t}(s)d\mu_x(s).$$
Notice that
\begin{align}\label{A_err_formula}
A_{\phi_t}-\widehat{A}_{\phi_t}
=&\frac{1}{b(t)}\int_S\rho_{\phi_t}(s)(A_{x,s}-\widehat{A}_{x,s} )d\mu_x(s) \nonumber\\
=&\frac{1}{b(t)}\int_S(A_{x,s}-\widehat{A}_{x,s} )d\mu_x(s)+\frac{t}{b(t)}\int_S(\rho_\phi(s)-1)(A_{x,s}-\widehat{A}_{x,s} )d\mu_x(s) \nonumber\\
=&\frac{t}{b(t)}\int_S(\rho_\phi(s)-1)(A_{x,s}-\widehat{A}_{x,s} )d\mu_x(s).
\end{align}
The last equality follows from the first two assertions in Lemma \ref{prelim_est} and the fact that
\begin{align*}
\frac{1}{b(t)}\int_S(A_{x,s}-\widehat{A}_{x,s} )d\mu_x(s) 
=&\frac{1}{b(t)}\int_S\overline\rho_{\phi_0}(A_{x,s}-\widehat{A}_{x,s} )d\mu_x(s) 
=\frac{1}{b(t)}(A_{\phi_0}-\widehat{A}_{\phi_0})=0.
\end{align*}
Since $\phi\in\cB(R)$, Lemma \ref{well_defined_P} implies  $\Phi(x)=\Phi(P(\phi))\in\cB(R_1)$. Hence
\begin{align}\label{rho_phi_est} 
c_0^{-1}\leq \rho_\phi(s)=e^{(dn+d)(\Phi_s(x)-\phi(s))}\leq c_0
\end{align}
and
$$\|\rho_\phi-1\|_{L^2(S)}=\left\|e^{(dn+d)(\Phi_s(x)-\phi(s))}-1\right\|_{L^2(S)}\leq c_0\|\Phi_s(x)-\phi(s)\|_{L^2(S)},$$
where 
$$c_0=c_0(n,R)=e^{(dn+d)(R+R_1)}(dn+d).$$
Since $x\in B_{x_0}(R_1)$, $A_{x,s}$ and $\widehat{A}_{x,s}$ depend smoothly on $g$ and they coincide when $g=g_0$, we have
$$\|A_{x,s}-\widehat{A}_{x,s}\|_{L^2(S)}\leq c_1\epsilon$$
for some $c_1=c_1(n,R)$. Therefore by the Cauchy-Schwarz inequality, \eqref{A_err_formula} implies
\begin{align*}
\|A_{\phi_t}-\widehat{A}_{\phi_t}\|=&\frac{t}{b(t)}\left\|\int_S(\rho_\phi(s)-1)(A_{x,s}-\widehat{A}_{x,s} )d\mu_x(s)\right\| \\
\leq&\frac{t}{b(t)}\|\rho_\phi-1\|_{L^2(S,d\mu_x)}\cdot\|A_{x,s}-\widehat{A}_{x,s}\|_{L^2(S,d\mu_x)}.
\end{align*}
Choose $R_0=1+R_1$ in \eqref{density_est} and Lemma \ref{well_defined_P} implies that
\begin{align}\label{density_abs_est}
\frac{1}{2}e^{-(dn+d-2) R_1}\leq|\lambda(x,s)|\leq 2e^{(dn+d-2) R_1}
\end{align}
for any $\phi\in\cB(R)$ and $x=P(\phi)\in B_{x_0}(R_1)$.
Hence by \eqref{rho_phi_est} and \eqref{density_abs_est} we have
\begin{align*}
\frac{t}{b(t)}\|\rho_\phi-1\|_{L^2(S,d\mu_x)}\|A_{x,s}-\widehat{A}_{x,s}\|_{L^2(S,d\mu_x)}
\leq~&2e^{(dn+d-2) R_1}\frac{t}{b(t)}\|(\rho_\phi-1)\|_{L^2(S)}\|A_{x,s}-\widehat{A}_{x,s}\|_{L^2(S)} \\
\leq~&2te^{(dn+d-2) R_1}c_0\|(\rho_\phi-1)\|_{L^2(S)}\|A_{x,s}-\widehat{A}_{x,s}\|_{L^2(S)}  \\
\leq~&2te^{(dn+d-2) R_1}c_0^2c_1\epsilon\|\Phi_s(x)-\phi(s)\|_{L^2(S)}.
\end{align*}
Summarizing up, we have 
$$\|A_{\phi_t}-\widehat{A}_{\phi_t}\|\leq 2te^{(dn+d-2) R_1}c_0^2c_1\epsilon\|\Phi_s(x)-\phi(s)\|_{L^2(S)}.$$
In particular, when $t=1$, we have
\begin{align}\label{A_norm_err}
\|A_{\phi}-\widehat{A}_{\phi}\|\leq 2e^{(dn+d-2) R_1}c_0^2c_1\epsilon\|\Phi_s(x)-\phi(s)\|_{L^2(S)}
\end{align}
In order to estimate the difference in determinants, we first introduce a linear algebra lemma.

\begin{lem}\label{matrix_lem}
Let $H_1,H_2$ be $n\times n$ trace-free matrices such that $\|H_1\|,\|H_2\|\leq R$. Then we have
$$\left|\det(\id+H_1+H_2)-\det(\id+H_1)\right|\leq \cK(n,R)\|H_2\|(\|H_1\|+\|H_2\|),$$
for some $\cK(n,R)>0$.
\end{lem}
\begin{proof}[Proof of Lemma \ref{matrix_lem}]
Let $H_1=(a_{ij})_{1\leq i,j\leq n}$ and $H_2=(b_{ij})_{1\leq i,j\leq n}$. Then we have 
$$\left|a_{ij}\right|\leq\|H_1\|\leq R ~~\mathrm{and}~~\left|b_{ij}\right|\leq\|H_2\|\leq R,\quad 1\leq i,j\leq n.$$
Direct computation shows that $\left|\det(\id+H_1+H_2)-\det(\id+H_1)\right|$ can be written as the sum of at most $3^n$ monomials in $a_{ij}$ and $b_{kl}$ with coefficients $\pm 1$. These monomials have degree at least 2 due to trace-free assumptions and always have positive degrees with respect to some $b_{kl}$. More precisely, 
\begin{align*}
&\left|\det(\id+H_1+H_2)-\det(\id+H_1)\right| \\
=&\left|\sum_{\sigma\in S_{n}}(-1)^{\mathrm{sgn}(\sigma)}\left[\prod_{i=1}^{n}\left(\delta_{i\sigma(i)}+a_{i\sigma(i)}+b_{i\sigma(i)}\right)-\prod_{i=1}^{n}\left(\delta_{i\sigma(i)}+a_{i\sigma(i)}\right)\right]\right| \\
=&\left|\sum_{\sigma\in S_{n}, \mathrm{Fix}(\sigma)^\mathrm{c}\subset\Omega\subset S_{n},|\Omega|\geq2 \mathrm{~if~} \sigma=\id}(-1)^{\mathrm{sgn}(\sigma)}\left(\prod_{i\in\Omega}\left(a_{i\sigma(i)}+b_{i\sigma(i)}\right)-\prod_{i\in\Omega}a_{i\sigma(i)}\right)\right| \\
\leq&\sum_{\sigma\in S_{n}, \mathrm{Fix}(\sigma)^\mathrm{c}\subset\Omega\subset S_{n},|\Omega|\geq2 \mathrm{~if~} \sigma=\id}\left|\prod_{i\in\Omega}\left(a_{i\sigma(i)}+b_{i\sigma(i)}\right)-\prod_{i\in\Omega}a_{i\sigma(i)}\right| \\
\leq&\sum_{\sigma\in S_{n}, \mathrm{Fix}(\sigma)^\mathrm{c}\subset\Omega\subset S_{n},|\Omega|\geq2 \mathrm{~if~} \sigma=\id}(2^{|\Omega|}-1)\|H_2\|(\|H_1\|+\|H_2\|)^{|\Omega|-1} \\
\leq&\sum_{\sigma\in S_{n}, \mathrm{Fix}(\sigma)^\mathrm{c}\subset\Omega\subset S_{n}}(2^{n}-1)(2R+1)^{n-2}\|H_2\|(\|H_1\|+\|H_2\|) \\
\leq&|S_{n}|2^{n}(2^{n}-1)(2R+1)^{n-2}\|H_2\|(\|H_1\|+\|H_2\|),
\end{align*}
where $S_{n}$ is the group of symmetry over $\{1,2,3,...,n\}$. Choose $\cK(n,R)=|S_{n}|2^{n}(2^{n}-1)(2R+1)^{n-2}$ and the lemma follows.
\end{proof}

Back to the proof of Proposition \ref{A_err}, notice that Lemma \ref{easy_ineq} implies that 
$$\tr \left(\widehat{A}_\psi-\frac{n+1}{n}\id\right)=0,$$
for any $\psi\in\cB(R)$.

Lemma \ref{prelim_est} and the above fact suggest that
\begin{align*}
0=&\tr \left.\frac{d}{dt}\right|_{t=0}\left(A_{\phi_t}-\frac{n+1}{n}\id\right)  \\
=&\tr \left.\frac{d}{dt}\right|_{t=0}\left(\widehat{A}_{\phi_t}-\frac{n+1}{n}\id\right) +\tr \frac{1}{b(0)}\int_S(\rho_\phi(s)-1)(A_{x,s}-\widehat{A}_{x,s})d\mu_x(s) \\
=&\tr \frac{1}{b(0)}\int_S(\rho_\phi(s)-1)(A_{x,s}-\widehat{A}_{x,s})d\mu_x(s).
\end{align*}
Let
$$H_1=\widehat{A}_\phi-\frac{n+1}{n}\id$$
and
$$H_2=A_\phi-\widehat{A}_\phi=A_\phi-\widehat{A}_\phi-\frac{1}{b(1)}(A_{\phi_0}-\widehat{A}_{\phi_0})=\frac{1}{b(1)}\int_S(\rho_\phi(s)-1)(A_{x,s}-\widehat{A}_{x,s})d\mu_x(s).$$
In order to estimate $\|H_1\|$, we first recall its matrix under the orthonormal basis $\Xi\subset T_xM$.
$$_{\Xi}[H_1]_{\Xi}=\frac{1}{n}\left(\widehat{Q}_\phi-\id\right).$$
The second assertion of Lemma \ref{easy_ineq} implies that 
$$\|H_1\|\leq\left\|\frac{1}{n}(U_\phi^TU_\phi-\id)\right\|$$
Hence it suffices to consider components in $\frac{1}{dn}(U_\phi^TU_\phi-\id)$. Notice that
\begin{align*}
\frac{1}{dn}\left(\langle X_{l_1,t_1,g}, X_{l_2,t_2,g}\rangle_{G_\phi}-\delta_{l_1l_2}\delta_{t_1t_2}\right)
=&~\int_S(\overline\rho_\phi(s)-1) X_{l_1,t_1,g}(s) X_{l_2,t_2,g}(s)d\mu_x(s) \\
=&~\frac{1}{\int_S\rho_\phi(s)d\mu_x(s)}\int_S(\rho_\phi(s)-1) X_{l_1,t_1,g}(s) X_{l_2,t_2,g}(s)d\mu_x(s)  \\
&-\frac{\int_S(\rho_\phi(s)-1) d\mu_x(s)}{\int_S\rho_\phi(s)d\mu_x(s)}\int_S X_{l_1,t_1,g}(s) X_{l_2,t_2,g}(s)d\mu_x(s). 
\end{align*}
Therefore \eqref{rho_phi_est} and \eqref{density_abs_est} imply that
\begin{align*}
&~\left|\frac{1}{dn}\left(\langle X_{l_1,t_1,g}, X_{l_2,t_2,g}\rangle_{G_\phi}-\delta_{l_1l_2}\delta_{t_1t_2}\right)\right| \\
\leq&~\left|\frac{1}{\int_S\rho_\phi(s)d\mu_x(s)}\int_S(\rho_\phi(s)-1) X_{l_1,t_1,g}(s) X_{l_2,t_2,g}(s)d\mu_x(s)\right|  \\
&+\left|\frac{\int_S(\rho_\phi(s)-1) d\mu_x(s)}{\int_S\rho_\phi(s)d\mu_x(s)}\int_S X_{l_1,t_1,g}(s) X_{l_2,t_2,g}(s)d\mu_x(s)\right| \\
\leq&~c_0\left|\int_S(\rho_\phi(s)-1) X_{l_1,t_1,g}(s) X_{l_2,t_2,g}(s)d\mu_x(s)\right|  \\
&+c_0\left|\left(\int_S(\rho_\phi(s)-1) d\mu_x(s)\right)\int_S X_{l_1,t_1,g}(s) X_{l_2,t_2,g}(s)d\mu_x(s)\right| \\
\leq&~2c_0\|(\rho_\phi-1)\|_{L^1(S,d\mu_x)} 
\leq~4c_0e^{(dn+d-2) R_1}\|(\rho_\phi-1)\|_{L^1(S)} \leq~4c_0e^{(dn+d-2) R_1}\|(\rho_\phi-1)\|_{L^2(S)}
\end{align*}
for any $1\leq l_1,l_2\leq n$ and $0\leq t_1,t_2\leq d-1$. Hence by \eqref{rho_phi_est} we have
\begin{align*}
\|H_1\|\leq&~\left\|\frac{1}{n}(U_\phi^TU_\phi-\id)\right\| \\
\leq&~ dn\sup_{1\leq l_1,l_2\leq n, 0\leq t_1,t_2\leq d-1}\left|\frac{1}{n}\left(\langle X_{l_1,t_1,g}, X_{l_2,t_2,g}\rangle_{G_\phi}-\delta_{l_1l_2}\delta_{t_1t_2}\right)\right| \\
\leq&~4d^2nc_0e^{(dn+d-2) R_1}\|(\rho_\phi-1)\|_{L^2(S)} 
\leq~4d^2nc_0^2e^{(dn+d-2) R_1}\|\Phi_s(x)-\phi(s)\|_{L^2(S)}
\end{align*}
Recall that \eqref{A_norm_err} implies 
$$\|H_2\|\leq 2e^{(dn+d-2) R_1}c_0^2c_1\epsilon\|\Phi_s(x)-\phi(s)\|_{L^2(S)}.$$
Choose $R_3=R_3(n,R)=(R+R_1)R_2$ and
$$R_2=R_2(n,R)=\left(2e^{(dn+d-2) R_1}c_0^2c_1+4d^2nc_0^2e^{(dn+d-2) R_1}\right).$$
Hence 
$\|H_1\|+\|H_2\|\leq \min\{R_3,R_2\|\Phi_s(x)-\phi(s)\|_{L^2(S)}\}$ 
by the assumption of $\phi\in\cB(R)$, $\epsilon<1$ and Lemma \ref{well_defined_P}. It follows from Lemma \ref{matrix_lem} that
$$\|\det A_\phi-\det\widehat{A}_\phi\|\leq \cK_1(n,R_3)\|H_2\|(\|H_1\|+\|H_2\|).$$
Hence
\begin{align*}
&~\frac{\|\det A_\phi-\det\widehat{A}_\phi\|}{\epsilon\|\Phi_s(x)-\phi(s)\|^2_{L^2(S)}}\leq~ \cK_1(n,R_3)(R_2+1)^2.
\end{align*}
Based on the above inequality and \eqref{A_norm_err}, we can choose $C_0(n,R)$ satisfying
$$C_0(n,R)>\max\{\cK_1(n,R_3)(R_2+1)^2,2e^{(dn+d-2) R_1}c_0^2c_1\}.$$
Then we have 
$$\|A_\phi-\widehat{A}_\phi \| \leq C_0\epsilon\|\phi-\Phi(P(\phi))\|_{L^2(S)}$$
and
\begin{align*}
|\det A_\phi-\det\widehat{A}_\phi | \leq C_0\epsilon\|\phi-\Phi(P(\phi))\|_{L^2(S)}^2. \tag*{\qedhere}
\end{align*}
\end{proof}
\begin{cor}\label{A_invert}
Let $\phi\in\cB(R)$. If $g$ is $C^r$ sufficiently close to $g_0$, then $A_\phi$ is invertible. Moreover, we have 
$$\left\|A_\phi^{-1}-\widehat{A}_\phi^{-1} \right\| \leq \widetilde C_0\epsilon\|\phi-\Phi(P(\phi))\|_{L^2(S)}$$
and
$$\left|\det A^{-1}_\phi-\det\widehat{A}^{-1}_\phi \right| \leq \widetilde C_0\epsilon\|\phi-\Phi(P(\phi))\|_{L^2(S)}^2,$$
for some constant $\widetilde C_0=\widetilde C_0(n,R)>0$.
\end{cor}
\begin{proof}
 When $\phi\in\cB(R)$, Lemma \ref{well_defined_P} implies $\Phi(P(\phi))\in\cB(R_1)$. Since Lemma \ref{A_formula} implies that $ \id\leq \widehat{A}_\phi\leq (d+1)\id$, we have 
$$1\leq \det \widehat{A}_\phi\leq (d+1)^{dn}.$$
Therefore, when $\epsilon(g,r)\leq \min\{1/[2C_0(R+R_1)],1/[2C_0(R+R_1)^2]\}$, by the norm estimate and the determinant estimate in Proposition \ref{A_err} we have
$$\frac{1}{2}\id\leq A_\phi\leq \frac{2d+3}{2}\id$$
and
$$\frac{1}{2}\leq\det A_\phi\det\widehat{A}_\phi\leq (d+1)^{dn}\left[(d+1)^{dn}+\frac{1}{2}\right],$$
which imply that
\begin{align*}
\left\|A_\phi^{-1}-\widehat{A}_\phi^{-1} \right\| \leq&\left\|A_\phi^{-1}\right\| \left\|A_\phi-\widehat{A}_\phi \right\| \left\|\widehat{A}_\phi^{-1} \right\|\leq 2C_0\epsilon\|\phi-\Phi(P(\phi))\|_{L^2(S)}
\end{align*}
and that
$$\left|\det A^{-1}_\phi-\det\widehat{A}^{-1}_\phi \right| \leq 2C_0\epsilon\|\phi-\Phi(P(\phi))\|_{L^2(S)}^2.$$
The corollary follows from choosing $\widetilde C_0=2C_0$.
\end{proof}

We need one further estimate on the Jacobian of $\widehat{A}_\phi^{-1}\circ E_\phi$.
\begin{prop}\label{AE_err}
For any $\phi\in\cB(R)$ (assuming $x=P(\phi)$), any $dn$-dimensional subspace $\cW\subset \cL$ and any $X\in \cW$ such that $\langle X,X\rangle_{G_\phi}=1$, we have
$$\det(\widehat{A}_\phi^{-1})\Jac_{G,\cW}E_\phi\leq 1-C_1\|X-d\Phi\circ \widehat{A}_\phi^{-1}\circ E_\phi(X)\|^2_{L^2(S)}$$ 
for some $C_1=C_1(n,R)$.
\end{prop}
\begin{proof}
First, we proof the case when $\cW=\cV_\phi=\mathrm{span}Z_\phi$, where $\cV_\phi,Z_\phi$ are defined in the previous chapter. 

\textbf{Special case: $\cW=\cV_\phi$. }

Since the left hand side of the inequality can be written in terms of eigenvalues of $U_\phi^TU_\phi$, our main idea is to estimate $\|X-d\Phi\circ \widehat{A}_\phi^{-1}\circ E_\phi(X)\|^2_{L^2(S)}$ by eigenvalues of $U_\phi^TU_\phi$.

Recall that in Lemma \ref{A_formula} and Lemma \ref{E_formula} we have 
$$_{\Xi}[\widehat{A}_\phi]_{\Xi}=\id+\frac{1}{n}\widehat{Q}_\phi$$
and
$$_{\Xi}[E_\phi|_{\cV_\phi}]_{W_\phi}=\frac{n+1}{n} U_\phi^T,$$
we can conclude that
\begin{align}\label{AE_formula_1}
_{\Xi}[\widehat{A}_\phi^{-1}\circ E_\phi|_{\cV_\phi}]_{W_\phi}=\frac{n+1}{n}\left(\id+\frac{1}{n}\widehat{Q}_\phi\right)^{-1}U_\phi^T.
\end{align}
Let $\phi_0:=\Phi(P(\phi))$. We define a map $i_\phi:T_{\phi_0}\cL\to T_\phi\cL$ such that, if we identify both tangent spaces as $\cL$, $i_\phi(Y)=Y$ for all $Y\in \cL=T_{\phi_0}\cL$. In particular, $i_\phi(Z_{\phi_0})=Z_\phi$. Notice that $d\Phi(\Xi)=Z_{\phi_0}$, \eqref{AE_formula_1} implies that
$$_{ Z_\phi}[i_\phi\circ d\Phi\circ \widehat{A}_\phi^{-1}\circ E_\phi|_{\cV_\phi}]_{W_\phi}=\frac{n+1}{n}\left(\id+\frac{1}{n}\widehat{Q}_\phi\right)^{-1}U_\phi^T.$$
Hence 
\begin{align*}
_{W_\phi}[i_\phi\circ d\Phi\circ \widehat{A}_\phi^{-1}\circ E_\phi|_{\cV_\phi}]_{W_\phi}
=&\frac{n+1}{n}~_{W_\phi}[\id]_{ Z_\phi}\left(\id+\frac{1}{n}\widehat{Q}_\phi\right)^{-1}U_\phi^T 
=\frac{n+1}{n}U_\phi\left(\id+\frac{1}{n}\widehat{Q}_\phi\right)^{-1}U_\phi^T .
\end{align*}
Therefore 
$$_{W_\phi}[\id|_{\cV_\phi}-i_\phi\circ d\Phi\circ \widehat{A}_\phi^{-1}\circ E_\phi|_{\cV_\phi}]_{W_\phi}=\id-\frac{n+1}{n}U_\phi\left(\id+\frac{1}{n}\widehat{Q}_\phi\right)^{-1}U_\phi^T .$$
For simplicity, let 
$$\cG(U_\phi)=\id-\frac{n+1}{n}\left(\id+\frac{1}{n}\widehat{Q}_\phi\right)^{-1}U_\phi^TU_\phi\in\mathrm{Mat}_{dn\times dn}(\RR)$$
and hence 
\begin{align}\label{AE_formula_2}
_{W_\phi}[\id-i_\phi\circ d\Phi\circ \widehat{A}_\phi^{-1}\circ E_\phi|_{\cV_\phi}]_{W_\phi}=U_\phi\cG(U_\phi)U_\phi^{-1}.
\end{align}
Assume that $\widehat Q_\phi$ has eigenvalues $0\leq\eta_1\leq \eta_2\leq...\leq\eta_{dn}$ and $U_\phi^TU_\phi$ has eigenvalues $0\leq\lambda_1\leq\lambda_2\leq...\leq\lambda_{dn}$. Lemma \ref{easy_ineq} shows that
$$\sum_{t=1}^{dn}\lambda_t=\sum_{t=1}^{dn}\eta_t=dn$$
and
\begin{align*}
\lambda_1\leq \eta_1\leq 1\leq \eta_{dn}\leq\lambda_{dn}\leq dn.
\end{align*}
Therefore
\begin{align}\label{eigenval_est}
\|\cG(U_\phi)\|=&\left\|\id-\frac{n+1}{n}\left(\id+\frac{1}{n}\widehat{Q}_\phi\right)^{-1}U_\phi^TU_\phi\right\| \nonumber \\
\leq&\left\|\id-\frac{n+1}{n}\left(\id+\frac{1}{n}\widehat{Q}_\phi\right)^{-1}\right\| +\left\|\frac{n+1}{n}\left(\id+\frac{1}{n}\widehat{Q}_\phi\right)^{-1}(U_\phi^TU_\phi-\id)\right\| \nonumber \\
\leq& \left\|\id-\frac{n+1}{n}\left(\id+\frac{1}{n}\widehat{Q}_\phi\right)^{-1}\right\|+\frac{n+1}{n}\|(U_\phi^TU_\phi-\id)\| \nonumber \\
\leq&\frac{n+1}{n}\left(\frac{1}{1+\frac{1}{n}\eta_1}-\frac{1}{1+\frac{1}{n}\eta_{dn}}+\lambda_{dn}-\lambda_1\right) \nonumber\\
\leq&\frac{n+1}{n}\left(\frac{1}{1+\frac{1}{n}\lambda_1}-\frac{1}{1+\frac{1}{n}\lambda_{dn}}+\lambda_{dn}-\lambda_1\right) 
\leq\left(\frac{n+1}{n}\right)^2(\lambda_{dn}-\lambda_{1}),
\end{align}
where the third inequality follows from the fact that if a symmetric matrix has non-positive and non-negative eigenvalues, then its norm is controlled by the difference between its largest and smallest eigenvalues. In order to estimate $\left\|U_\phi\cG(U_\phi)U_\phi^{-1}\right\|$, we need the following result.
\begin{lem}\label{lower_bd_eigenval}
For any $\phi\in\cB(R)$ (assuming $x=P(\phi)$), there exists a positive constant $C_2=C_2(n,R)>0$ such that 
$$\lambda_1\geq C_2,$$
where $0\leq\lambda_1\leq\lambda_2\leq...\leq\lambda_{dn}$ are eigenvalues of $U_\phi^TU_\phi$.
\end{lem}
\begin{proof}[Proof of Lemma \ref{lower_bd_eigenval}.]
Let $v\in \RR^{dn}$ be an arbitary unit vector in an Euclidean space. Define
$$c_2(n)=\frac{\mathrm{Area}(\{w\in S^{dn-1}:\langle w,v\rangle_{\RR^{dn}}\geq 1/2\},ds)}{\mathrm{Area}(S^{dn-1},ds)}>0,$$
where $\langle\cdot,\cdot\rangle_{\RR^{dn}}$ denotes the inner product in the standard Euclidean space. $c_2$ is independent of the choice of $v$ due to spherical symmetry. For any unit vector $v=(a_{1,0},...,a_{1,d-1},...,a_{n,0},...,a_{n,d-1})^T\in\mathrm{Mat}_{dn\times 1}(\RR)=\RR^{dn}$, we construct a vector (see Subsection \ref{ss51} and \ref{ss52} for definitions of $\xi_{l,t,g}$ and $X_{l,t,g}$.)
$$\xi_v=\sum_{l=1}^{n}\sum_{t=0}^{d-1}a_{l,t}\xi_{l,t,g}\in T_xM.$$
Then \eqref{big_fat_matrix} implies that 
\begin{align}\label{lower_bd_technique}
v^TU_\phi^TU_\phi v=&dn\int_{S}\overline\rho_\phi(s)\left(\sum_{l=1}^{n}\sum_{t=0}^{d-1}a_{l,t}X_{l,t,g}(s)\right)^2d\mu_x(s) \nonumber\\
=&dn\int_{S}\overline\rho_\phi(s)\langle\grad\Phi_s(x),\xi_v\rangle^2d\mu_x(s) \nonumber\\
=&dn\int_{T^1_xM}\overline\rho_\phi(\alpha_g^{-1}(s))\langle \xi,\xi_v\rangle^2ds_x(\xi) \nonumber\\
\geq&dn\int_{\{w\in T^1_xM:\langle \xi,\xi_v\rangle \geq 1/2\}}\overline\rho_\phi(\alpha_g^{-1}(s))\langle \xi,\xi_v\rangle^2ds_x(\xi) \nonumber\\
\geq&\frac{dn}{4}\int_{\{w\in T^1_xM:\langle w,v\rangle \geq 1/2\}}\overline\rho_\phi(\alpha_g^{-1}(s))ds_x(\xi). 
\end{align}
Since \eqref{rho_phi_est} implies that
\begin{align}\label{rho_phi_bar_est}
c_0^{-2}\leq\overline\rho_\phi\leq c_0^2.
\end{align}
Therefore we have
$$\frac{dn}{4}\int_{\{\xi\in T^1_xM:\langle \xi,\xi_v\rangle \geq 1/2\}}\overline\rho_\phi(\alpha^{-1}(s))ds_x(\xi)\geq \frac{dn}{4}c_2(n)c_0^{-2}.$$
As a consequence of \eqref{lower_bd_technique}, we have
$$v^TU_\phi^TU_\phi v\geq \frac{dn}{4}c_2(n)c_0^{-2},$$
which implies that
$$\|U_\phi^TU_\phi\|\geq\frac{dn}{4}c_2(n)c_0^{-2}.$$
In particular, let $C_2=C_2(n,R)=\frac{dn}{4}c_2(n)c_0^{-2}$, we have $\lambda_1\geq C_2$.
\end{proof}
Going back to the proof of Proposition \ref{AE_err}, Lemma \ref{lower_bd_eigenval}, \eqref{eigenval_ineq} and \eqref{eigenval_est} imply that
\begin{align}
\left\|~_{W_\phi}[\id-i_\phi\circ d\Phi\circ \widehat{A}_\phi^{-1}\circ E_\phi|_{\cV_\phi}]_{W_\phi}\right\|=\left\|U_\phi\cG(U_\phi)U_\phi^{-1}\right\| 
\leq&\left\|U_\phi\right\|\left\|\cG(U_\phi)\right\|\|U_\phi^{-1}\| \nonumber \\
\leq &\lambda_{dn}^{1/2}\cdot\left(\frac{n+1}{n}\right)^2(\lambda_{dn}-\lambda_{1})\cdot\lambda_{1}^{-1/2} \nonumber \\
\leq &\sqrt{dn}\cdot\left(\frac{n+1}{n}\right)^2C_2^{-1/2}(\lambda_{dn}-\lambda_{1}), \nonumber
\end{align}
which implies that
$$\|X-d\Phi\circ \widehat{A}_\phi^{-1}\circ E_\phi(X)\|_{G_\phi}\leq\sqrt{dn}\cdot\left(\frac{n+1}{n}\right)^2C_2^{-1/2}(\lambda_{dn}-\lambda_{1}).$$
Notice that $\|\cdot\|_{G_\phi}=\|\cdot\|_{L^2(S,dn\overline\rho_\phi d\mu_x)}=\|\cdot\|_{L^2(S,dn\overline\rho_\phi \lambda(x,s)ds)}$, it follows from \eqref{density_abs_est} and \eqref{rho_phi_bar_est} that 
\begin{align}\label{norm_ineq}
\sqrt{\frac{dn}{2}c_0^{-2}e^{-(dn+d-2)R_1}}\|\cdot\|_{L^2(S)}\leq\|\cdot\|_{G_\phi}\leq \sqrt{2dnc_0^2e^{(dn+d-2)R_1}}\|\cdot\|_{L^2(S)}.
\end{align}
Write $c_3=c_3(n,R)=\sqrt{2(dn)^{-1}c_0^2e^{(dn+d-2)R_1}}\cdot\sqrt{dn}\cdot\left(\frac{n+1}{n}\right)^2C_2^{-1/2}$. Therefore
\begin{align}\label{AE_est}
\|X-d\Phi\circ \widehat{A}_\phi^{-1}\circ E_\phi(X)\|_{L^2(S)}\leq c_3(\lambda_{dn}-\lambda_{1}).
\end{align}
Following \eqref{AE_det_est} and Lemma \ref{lower_bd_eigenval}, we have
\begin{align*}
\det(\widehat{A}_\phi^{-1})\Jac_{G,\cV}E_\phi\leq&(\det{U^T_\phi U_\phi})^{\frac{1}{2}-\frac{1}{n+1}} \\
=&\left(\lambda_1\lambda_{dn}\prod_{t=2}^{dn-1}\lambda_t\right)^{\frac{1}{2}-\frac{1}{n+1}} \\
=&\left[\left(\frac{\lambda_1+\lambda_{dn}}{2}\right)^2\prod_{t=2}^{dn-1}\lambda_t-\frac{1}{4}(\lambda_1-\lambda_{dn})^2\prod_{t=2}^{dn-1}\lambda_t\right]^{\frac{1}{2}-\frac{1}{n+1}} \\
\leq&\left[\left(\frac{\sum_{t=1}^{dn}\lambda_t}{dn}\right)^{dn}-\frac{1}{4}(\lambda_1-\lambda_{dn})^2\lambda_1^{dn-2}\right]^{\frac{1}{2}-\frac{1}{n+1}} \\
\leq&\left[1-\frac{1}{4}C_2^{dn-2}(\lambda_1-\lambda_{dn})^2\right]^{\frac{1}{2}-\frac{1}{n+1}}
\leq1-\left(\frac{1}{2}-\frac{1}{n+1}\right)\frac{1}{4}C_2^{dn-2}(\lambda_1-\lambda_{dn})^2
\end{align*}
Choose 
$$C_1'=C_1'(n,R)=\left(\frac{1}{2}-\frac{1}{n+1}\right)\frac{1}{4}C_2^{dn-2}c_3^{-2}.$$
Then it follows from \eqref{AE_est} that
\begin{align}\label{W=V}
\det(\widehat{A}_\phi^{-1})\Jac_{G,\cV}E_\phi
\leq&1-\left(\frac{1}{2}-\frac{1}{n+1}\right)\frac{1}{4}C_2^{dn-2}(\lambda_1-\lambda_{dn})^2 \nonumber\\
\leq&1-C_1'c_3^2(\lambda_1-\lambda_{dn})^2 
\leq1-C_1'\|X-d\Phi\circ \widehat{A}_\phi^{-1}\circ E_\phi(X)\|^2_{L^2(S)}.
\end{align}

Now we return to the general case when $\cW$ is an arbitary $dn$-dimensional subspace of $\cL$. For any $X\in\cW$, we write $X=X^\parallel+X^\perp$, where $X^\perp\in\cV_\phi^\perp$ and $X^\parallel\in\cV_\phi$. We first make the following claim.
\begin{lem}\label{proj_det}
Let $P_{\cV_\phi}:T_\phi\cL=\cL\to\cV_\phi$ be the orthogonal projection onto $\cV_\phi$. Then there exists some constant $C_3=C_3(n)>0$ such that
$$\Jac_{G,\cW}P_{\cV_\phi}=\Jac_G\left(\left.P_{\cV_\phi}\right|_\cW\right)\leq 1-C_3\|X^\perp\|_{G_\phi}^2.$$
\end{lem}
\begin{proof}[Proof of Lemma \ref{proj_det}]
We follow the same idea as in Corollary \ref{AE_det_ineq} and \eqref{proj_matrix_jacobian_est}. Choose an othonormal basis $\{\widetilde X_1=X,\widetilde X_2,...,\widetilde X_{dn}\}$ in $\cW$. Then \eqref{proj_matrix_jacobian_est} implies that
\begin{align*}
\Jac_{G,\cW}P_{\cV_\phi} =&\Jac_G\left(P_{\cV_\phi}|_\cW\right)\\
\leq&\sqrt{\left(\frac{1}{dn}\sum_{t=1}^{dn}\|P_{\cV_\phi}(\widetilde X_t)\|_{G_\phi}^2\right)^{dn}} 
\leq\sqrt{\left(1-\frac{\|X^\perp\|_{G_\phi}^2}{dn}\right)^{dn}} 
\leq1-\frac{1}{2}\left(1-\frac{1}{dn}\right)^{\frac{dn}{2}-1}\|X^\perp\|_{G_\phi}^2.
\end{align*}
The last inequality follows from the Mean Value Theorem and the fact that $\|X^\perp\|_{G_\phi}\leq 1$. Choose $C_3(n)=\frac{1}{2}\left(1-\frac{1}{dn}\right)^{\frac{dn}{2}-1}$ and the lemma follows.
\end{proof}
Going back to the proof of Proposition \ref{AE_err}, \eqref{norm_ineq} and \eqref{W=V} imply that
\begin{align*}
\det(\widehat{A}_\phi^{-1})\Jac_{G,\cW}E_\phi 
=&\det(\widehat{A}_\phi^{-1})\Jac_{G}\left(E_\phi|_{\cV_\phi}\circ P_{\cV_\phi}\right) \\
=&\det(\widehat{A}_\phi^{-1})\Jac_{G,\cV_\phi}E_\phi \cdot \Jac_G\left(\left.P_{\cV_\phi}\right|_\cW\right)  \\
\leq&\det(\widehat{A}_\phi^{-1})\Jac_{G,\cV_\phi}E_\phi\left(1-C_3\|X^\perp\|_{G_\phi}^2\right) \\
\leq&\det(\widehat{A}_\phi^{-1})\Jac_{G,\cV_\phi}E_\phi\left(1-\frac{dn}{2}c_0^{-2}e^{-(dn+d-2)R_1}C_3\|X^\perp\|_{L^2(S)}^2\right) \\
\leq&\left(1-\frac{C_1'}{\left\|X^\parallel\right\|^2_{G_\phi}}\left\|X^\parallel-d\Phi\circ \widehat{A}_\phi^{-1}\circ E_\phi(X^\parallel)\right\|^2_{L^2(S)}\right)\left(1-c_4\|X^\perp\|_{L^2(S)}^2\right) \\
\leq&\left(1-C_1'\left\|X^\parallel-d\Phi\circ \widehat{A}_\phi^{-1}\circ E_\phi(X^\parallel)\right\|^2_{L^2(S)}\right)\left(1-c_4\|X^\perp\|_{L^2(S)}^2\right) \\
=&\left(1-C_1'\left\|X^\parallel-d\Phi\circ \widehat{A}_\phi^{-1}\circ E_\phi(X)\right\|^2_{L^2(S)}\right)\left(1-c_4\|X^\perp\|_{L^2(S)}^2\right) \\
\leq&1-\frac{1}{2}\left(C_1'\left\|X^\parallel-d\Phi\circ \widehat{A}_\phi^{-1}\circ E_\phi(X)\right\|^2_{L^2(S)}+c_4\|X^\perp\|_{L^2(S)}^2\right),
\end{align*}
where $c_4=c_4(n,R)=\frac{dn}{2}c_0^{-2}e^{-(dn+d-2)R_1}C_3$ and 
the last inequality follows from the fact that $(1-2a)(1-2b)\leq1-a-b$ provided $0\leq 1-2a,1-2b\leq 1$. Take 
$$C_1=C_1(n,R)=\frac{1}{2}\min\left\{\frac{1}{2}C_1',\frac{1}{2}c_4\right\},$$
we can summarize that
\begin{align*}
\det(\widehat{A}_\phi^{-1})\Jac_{G,\cW}E_\phi
\leq&1-\frac{1}{2}\left(C_1'\left\|X^\parallel-d\Phi\circ \widehat{A}_\phi^{-1}\circ E_\phi(X)\right\|^2_{L^2(S)}+c_4\|X^\perp\|_{L^2(S)}^2\right) \\
\leq&1-2C_1\left(\left\|X^\parallel-d\Phi\circ \widehat{A}_\phi^{-1}\circ E_\phi(X)\right\|^2_{L^2(S)}+\|X^\perp\|_{L^2(S)}^2\right) \\
\leq&1-C_1\left\|X^\perp+X^\parallel-d\Phi\circ \widehat{A}_\phi^{-1}\circ E_\phi(X)\right\|_{L^2(S)}^2 \\
=&1-C_1\left\|X-d\Phi\circ \widehat{A}_\phi^{-1}\circ E_\phi(X)\right\|_{L^2(S)}^2. \tag*{\qedhere}
\end{align*}
\end{proof}
\section{A compression trick}\label{s7}
This subsection is a review of \cite[Section 7, A compression trick]{Burago2} in the more general case of $\KK\mathbf{H}^n$. Recall in Proposition \ref{key_prop}, we need a ''projection'' map $P_\sigma$. In the previous two sections, we constructed a projection map $P$, but whether it is area non-increasing is not clear. In this subsection, we will give our construction of $P_\sigma$ as a small perturbation of $P$. Notice that the second assertion in Definition \ref{Projection} suggests that $P$ is a locally ``orthogonal'' projection in the sense that $d_\phi P$ vanish on the orthogonal complement of $\mathrm{Im}(d_{P(\phi)}\Phi)$. Therefore $P_\sigma$ constructed as a small perturbation of $P$ can be viewed as a ''almost locally orthogonal'' projection.

\begin{notation}
We define a ``height'' map $h(\phi):\cB(R)\to\RR_+$ as follows.
$$h(\phi)=\|\phi-\Phi(P(\phi))\|_{L^2(S)}.$$
Similar to \cite{Burago2}, we construct a map $F_c:\cB(R)\to M\times\RR_+$ by 
$$F_c(\phi)=(P(\phi),ch(\phi)).$$
Since $h$ is smooth on $\cB(R)\setminus\Phi(M)$, so is $F_c$.
\end{notation}
\begin{lem}\label{preparation}
There exist positive constants $C_5=C_5(n,R),C_6=C_6(n,R)>0$ such that for any $0\leq c\leq C_5$ and any $\phi\in\cB(R)\setminus\Phi(M)$
$$\Jac_GF_c\leq 1+C_6\epsilon h^2(\phi)$$
provided $\epsilon\ll 1$.
\end{lem}
\begin{proof}
Let $\cW\subset T_\phi\cL=\cL$ be any $dn$-dimensional subspace. Denoted by $\{\widetilde X_1,\widetilde X_2,...,\widetilde X_{dn}\}$ an orthonormal basis in $\cW$ and $\{\omega_1,\omega_2,...,\omega_{dn}\}$ an orthonormal basis in $T_xM$, where $x=P(\phi)$. Without loss of generality, we can assume that 
\begin{enumerate}
\item[(1).] $d_\phi h(\widetilde X_l)=0$ for any $l\geq2$;
\item[(2).] $d_\phi P|_\cW$ has upper triangular matrix $\left(a_{sl}\right)_{1\leq s,l\leq dn}$ with non-negative diagonal under the above choices of bases.
\end{enumerate}
Then
\begin{align}\label{Fc_jacobian}
\Jac_GF_c|_\cW=\sqrt{a_{11}^2+t^2}\prod_{l=2}^{dn}a_{ll},
\end{align}
where 
$$t=d_\phi(ch(\widetilde X_1))=\frac{c}{2h(\phi)}d_\phi (h^2)(\widetilde X_1).$$
Therefore 
\begin{align}\label{t_bd}
|t|\leq c\|\widetilde X_1-d\Phi\circ d_\phi P(\widetilde X_1)\|_{L^2(S)}.
\end{align}
The first assertion in Corollary \ref{A_invert} suggests that 
$$\|A_\phi^{-1}-\widehat{A}_\phi^{-1} \| \leq \widetilde C_0\epsilon\|\phi-\Phi(P(\phi))\|_{L^2(S)}.$$
Therefore $\|A_\phi^{-1}\|\leq 2$ provided $\epsilon\leq [\widetilde C_0(R+R_1)]^{-1}$ (which is actually weaker than $\epsilon\leq\min\{[2C_0(R+R_1)]^{-1},[2C_0(R+R_1)^2]^{-1}\}$ mentioned in the proof of Corollary \ref{A_invert}). Conbined with Corollary \ref{AE_det_ineq}, we have $P$ is $2(dn+d)/\sqrt{dn}$-Lipschitz. In particular, 
\begin{align}\label{a_est}
0\leq a_{ii}\leq\frac{2(dn+d)}{\sqrt{dn}}.
\end{align} 
By \eqref{density_est} and \eqref{norm_ineq} we have
\begin{align}\label{t_bd_key}
&\|\widetilde X_1-d\Phi\circ d_\phi P(\widetilde X_1)\|_{L^2(S)} \nonumber \\
=&\|\widetilde X_1-a_{11}d\Phi(\omega_1)\|_{L^2(S)} \nonumber \\
\leq&\|\widetilde X_1\|_{L^2(S)}+a_{11}\|d\Phi(\omega_1)\|_{L^2(S)} \nonumber \\
\leq&\|\widetilde X_1\|_{L^2(S)}+\frac{2(dn+d)}{\sqrt{dn}}\|d\Phi(\omega_1)\|_{L^2(S)} \nonumber \\
\leq&\sqrt{2(dn)^{-1}c_0^{2}e^{(dn+d-2)R_1}}\|\widetilde X_1\|_{G_\phi}+\frac{2(dn+d)}{\sqrt{dn}}\sqrt{2e^{(dn+d-2) R_1}}\|d\Phi(\omega_1)\|_{L^2(S,d\mu_x)} \nonumber \\
=&\sqrt{2(dn)^{-1}c_0^{2}e^{(dn+d-2)R_1}}\|\widetilde X_1\|_{G_\phi}+\frac{2(dn+d)}{dn}\sqrt{2e^{(dn+d-2) R_1}}\|d\Phi(\omega_1)\|_{G_{\Phi(x)}} \nonumber \\
=&\sqrt{2(dn)^{-1}c_0^{2}e^{(dn+d-2)R_1}}\|\widetilde X_1\|_{G_\phi}+\frac{2(n+1)}{n}\sqrt{2e^{(dn+d-2) R_1}}\|\omega_1\| \nonumber \\
=&\sqrt{2(dn)^{-1}c_0^{2}e^{(dn+d-2)R_1}}+\frac{2(n+1)}{n}\sqrt{2e^{(dn+d-2) R_1}}=:c_5(n,R)=c_5. 
\end{align}
Without loss of generality, assume that 
$$C_5=C_5(n,R)\leq \max\left\{c_5^{-1}\left[2(dn+d)/\sqrt{dn}\right]^{-dn},1\right\}.$$
It follows from \eqref{t_bd} and the assumption of $0\leq c\leq C_5$ that 
$$|t|\leq\left[2(dn+d)/\sqrt{dn}\right]^{-dn}.$$ 
We consider two different cases.

\textbf{Case 1.} Assume that $a_{11}<\left[2(dn+d)/\sqrt{dn}\right]^{-dn}$. Then
 $$\sqrt{a_{11}^2+t^2}\leq\sqrt{2}\left[\frac{2(dn+d)}{\sqrt{dn}}\right]^{-dn}.$$ 
It follows from \eqref{Fc_jacobian} and \eqref{a_est} that
\begin{align*}
\Jac_GF_c|_\cW=&\sqrt{a_{11}^2+t^2}\prod_{l=2}^{dn}a_{ll}\\
\leq&\sqrt{2}\left[\frac{2(dn+d)}{\sqrt{dn}}\right]^{-dn}\prod_{l=2}^{dn}a_{ll}
\leq\sqrt{2}\left[\frac{2(dn+d)}{\sqrt{dn}}\right]^{-dn}\left[\frac{2(dn+d)}{\sqrt{dn}}\right]^{dn-1} 
=\frac{\sqrt{2dn}}{2(dn+d)}<1.
\end{align*}

\textbf{Case 2.} Assume that $a_{11}\geq\left[2(dn+d)/\sqrt{dn}\right]^{-dn}$. Then
$$\sqrt{a_{11}^2+t^2}\leq a_{11}+\frac{t^2}{2a_{11}}$$
and 
\begin{align}\label{Fc_jacobian_est_2}
\Jac_GF_c|_\cW=\sqrt{a_{11}^2+t^2}\prod_{l=2}^{dn}a_{ll}
\leq \left(a_{11}+\frac{t^2}{2a_{11}}\right)\prod_{l=2}^{dn}a_{ll}
=&\Jac_{G,\cW}P+\frac{t^2}{2a_{11}}\prod_{l=2}^{dn}a_{ll}  \nonumber\\
\leq&\Jac_{G,\cW}P+\frac{t^2}{2\left[\frac{2(dn+d)}{\sqrt{dn}}\right]^{-dn}}\left[\frac{2(dn+d)}{\sqrt{dn}}\right]^{dn-1}\nonumber \\
=&\Jac_{G,\cW}P+\frac{1}{2}\left[\frac{2(dn+d)}{\sqrt{dn}}\right]^{2dn-1} t^2.
\end{align}
Lemma \ref{easy_ineq}, Corollary \ref{A_invert} and Proposition \ref{AE_err} imply that
\begin{align*}
\Jac_{G,\cW}P=\det(A_\phi^{-1})\Jac_{G,\cW}E_\phi 
\leq&\left(\det(\widehat{A}_\phi^{-1})+\widetilde C_0\epsilon\|\phi-\Phi(P(\phi))\|_{L^2(S)}^2\right)\Jac_{G,\cW}E_\phi  \\
=&\det(\widehat{A}_\phi^{-1})\Jac_{G,\cW}E_\phi+\widetilde C_0\epsilon h^2(\phi)\Jac_{G,\cW}E_\phi \\
\leq&\det(\widehat{A}_\phi^{-1})\Jac_{G,\cW}E_\phi+\widetilde C_0 \epsilon h^2(\phi)\Jac_{G}E_\phi \\
=&\det(\widehat{A}_\phi^{-1})\Jac_{G,\cW}E_\phi+\widetilde C_0 \epsilon h^2(\phi)\left|\det\left(\frac{n+1}{n}U_\phi^T\right)\right| \\
=&\det(\widehat{A}_\phi^{-1})\Jac_{G,\cW}E_\phi+\widetilde C_0 \epsilon h^2(\phi)\det\left[\frac{n+1}{n}\left(U_\phi^TU_\phi\right)^{\frac{1}{2}}\right]\\
\leq&\det(\widehat{A}_\phi^{-1})\Jac_{G,\cW}E_\phi+\left(\frac{n+1}{n}\right)^{dn}\widetilde C_0 \epsilon h^2(\phi) \\
\leq& 1-C_1\|\widetilde X_1-d\Phi\circ \widehat{A}_\phi^{-1}\circ E_\phi(\widetilde X_1)\|^2_{L^2(S)}+\left(\frac{n+1}{n}\right)^{dn}\widetilde C_0\epsilon h^2(\phi). 
\end{align*}
It follows from \eqref{Fc_jacobian_est_2} that
\begin{align}\label{Fc_jacobian_est_3}
\Jac_GF_c|_\cW\leq& 1-C_1\|\widetilde X_1-d\Phi\circ \widehat{A}_\phi^{-1}\circ E_\phi(\widetilde X_1)\|^2_{L^2(S)}+\left(\frac{n+1}{n}\right)^{dn}\widetilde C_0\epsilon h^2(\phi)+\frac{t^2}{2}\left[\frac{2(dn+d)}{\sqrt{dn}}\right]^{2dn-1} . 
\end{align}
Corollary \ref{AE_det_ineq}, Corollary \ref{A_invert}, \eqref{density_abs_est} and \eqref{t_bd} imply that
\begin{align*}
|t|\leq & c\|\widetilde X_1-d\Phi\circ d_\phi P(\widetilde X_1)\|_{L^2(S)} \\
\leq& c\left\|\widetilde X_1-d\Phi\circ \widehat{A}_\phi^{-1}\circ E_\phi(\widetilde X_1)\right\|_{L^2(S)} 
+c\left\|d\Phi\circ \left(A_\phi^{-1}-\widehat{A}_\phi^{-1}\right)\circ E_\phi(\widetilde X_1)\right\|_{L^2(S)} \\
\leq& c\left\|\widetilde X_1-d\Phi\circ \widehat{A}_\phi^{-1}\circ E_\phi(\widetilde X_1)\right\|_{L^2(S)} 
+c\sqrt{2e^{(dn+d-2) R_1}}\left\|d\Phi\circ \left(A_\phi^{-1}-\widehat{A}_\phi^{-1}\right)\circ E_\phi(\widetilde X_1)\right\|_{L^2(S,d\mu_x)} \\
=& c\left\|\widetilde X_1-d\Phi\circ \widehat{A}_\phi^{-1}\circ E_\phi(\widetilde X_1)\right\|_{L^2(S)}  +c\sqrt{\frac{2}{dn}e^{(dn+d-2) R_1}}\left\|d\Phi\circ \left(A_\phi^{-1}-\widehat{A}_\phi^{-1}\right)\circ E_\phi(\widetilde X_1)\right\|_{G_{\Phi(x)}} \\
\leq& c\left\|\widetilde X_1-d\Phi\circ \widehat{A}_\phi^{-1}\circ E_\phi(\widetilde X_1)\right\|_{L^2(S)} +c\sqrt{\frac{2}{dn}e^{(dn+d-2) R_1}}\left\|d\Phi\right\|\left\| \left(A_\phi^{-1}-\widehat{A}_\phi^{-1}\right)\right\|\left\| E_\phi\right\|\left\|\widetilde X_1\right\|_{G_{\phi}} \\
\leq& c\left(\left\|\widetilde X_1-d\Phi\circ \widehat{A}_\phi^{-1}\circ E_\phi(\widetilde X_1)\right\|_{L^2(S)}+\frac{(n+1)}{n}\sqrt{2e^{(dn+d-2) R_1}}\widetilde C_0\epsilon \|\phi-\Phi(P(\phi))\|_{L^2(S)}\right)\\
=&c\left(\left\|\widetilde X_1-d\Phi\circ \widehat{A}_\phi^{-1}\circ E_\phi(\widetilde X_1)\right\|_{L^2(S)}+\frac{(n+1)}{n}\sqrt{2e^{(dn+d-2) R_1}}\widetilde C_0\epsilon h(\phi)\right).
\end{align*}
Since $C_5\leq1$ by definition, therefore $c<1$ and 
\begin{align}\label{t_upper_bd}
t^2\leq 2c^2\left\|\widetilde X_1-d\Phi\circ \widehat{A}_\phi^{-1}\circ E_\phi(\widetilde X_1)\right\|^2_{L^2(S)}+c_6\epsilon h^2(\phi),
\end{align}
where 
$$c_6=c_6(n,R)=2\left(\frac{(n+1)}{n}\sqrt{2e^{(dn+d-2)R_1}}\widetilde C_0\right)^2,$$
and $\epsilon<1$ by our assumption.
Let 
$$C_6=C_6(n,R)=\frac{1}{2}\left[\frac{2(dn+d)}{\sqrt{dn}}\right]^{2dn-1}c_6+\left(\frac{n+1}{n}\right)^{dn}\widetilde C_0.$$
\eqref{Fc_jacobian_est_3} and \eqref{t_upper_bd} can be summarized as
\begin{align*}
\Jac_GF_c|_\cW\leq& 1-\left\{C_1-\left[\frac{2(dn+d)}{\sqrt{dn}}\right]^{2dn-1}c^2\right\}\|\widetilde X_1-d\Phi\circ \widehat{A}_\phi^{-1}\circ E_\phi(\widetilde X_1)\|^2_{L^2(S)}+C_6\epsilon h^2(\phi).
\end{align*}
Choose $C_5$ in Lemma \ref{preparation} such that $$C_5=\min\left\{C_1^{\frac{1}{2}}\left[\frac{2(dn+d)}{\sqrt{dn}}\right]^{\frac{1-2dn}{2}},c_5\left[\frac{2(dn+d)}{\sqrt{dn}}\right]^{-dn},1\right\}.$$
Then
$$\Jac_GF_c|_\cW\leq 1+C_6\epsilon h^2(\phi)$$
and the lemma follows.
\end{proof}
We follow the idea in \cite{Burago2} to construct a ''homothety'' $\cA_t:M\to M$ by
$$\cA_t(p)=\exp_{x_0}(t\exp_{x_0}^{-1}(p)),\quad \forall p\in M,$$
and a map $\cQ_{\sigma}:M\times\RR_+\to M$ by
$$\cQ_\sigma(p,h)=\cA_{1+\sigma h^2}(p),\quad \forall p\in M.$$
\begin{lem}\label{lem_Rauch}
If $d(p,x_0)\leq 4\sigma^{-1/2}$, assuming $g$ sufficiently close to $g_0$ such that $(M,g)$ has strictly negative sectional curvature, then the $dn$-dimensional Jacobian of $\cQ_\sigma$ at $(p,h)$ is no greater than $(1+\sigma h^2)^{-1}$.
\end{lem}
\begin{proof}
The proof can be found in \cite[Lemma 7.2]{Burago2}.
\end{proof}
Now we define $P_\sigma(\phi)=\cQ_\sigma(P(\phi),\sigma h(\phi))=\cQ_\sigma(F_\sigma(\phi))$ on $\cB(R)$. By definition $P_\sigma$ is smooth.
\begin{prop}\label{final_prop}
For every $R>0$, there exists a $\sigma>0$, $c>0$ and $\epsilon>0$ such that the $dn$-dimensional Jacobian $J(\phi):=\Jac_GP_\sigma(\phi)$  with respect to $G$ at any point $\phi\in\cB(R)$ satisfies
$$J(\phi)\leq 1-c h^2(\phi)$$
provided $\epsilon(g,r)\leq \epsilon$.
\end{prop}
\begin{proof}
Choose $\sigma$ such that $\sigma<C_5(n,R)$ in Lemma \ref{preparation} and $(4\sigma)^{-1/2}>R_1(n,R)$ in Lemma \ref{well_defined_P}.

For any $\phi\in\Phi(M)$, $h(\phi)=0$, Hence we have $\Jac ~d_\phi P_\sigma=\Jac~ d_\phi P=1$, following the equality conditions in Lemma \ref{easy_ineq} and Corollary \ref{AE_det_ineq}. 

For any $\phi\in\cB(R)\setminus\Phi(M)$, Lemma \ref{preparation} and Lemma \ref{lem_Rauch} implies
\begin{align*}
J(\phi)\leq \Jac\cQ_\sigma\cdot \Jac_GF_\sigma\leq\frac{1+C_6\epsilon(g,r) h^2(\phi)}{1+\sigma^3h^2(\phi)}.
\end{align*}
Choose $\sigma=\sqrt[3]{3C_6\epsilon}$ and $\epsilon$ such that 
\begin{enumerate}
\item[(1).]$\epsilon$ is smaller than all upper bounds for $\epsilon(g,r)$ mentioned prior to this proposition;
\item[(2).]$3C_6\epsilon h^2(\phi)\leq 3C_6\epsilon(R+R_1)^2\leq 1$;
\item[(3).]$\sigma<C_5$;
\item[(4).]$(4\sigma)^{-1/2}\geq R_1$.
\end{enumerate}
Let $c=C_6(n,R)\epsilon=\sigma^3/3$, Hence 
\begin{align*}J(\phi)\leq\frac{1+C_6\epsilon(g,r)h^2(\phi)}{1+\sigma^3h^2(\phi)}\leq\frac{1+\frac{\sigma^3}{3}h^2(\phi)}{1+\sigma^3h^2(\phi)}\leq1-ch^2(\phi).\tag*{\qedhere}\end{align*}
\end{proof}
Now we are in position to complete the proof of Proposition \ref{key_prop}.
\begin{proof}[Proof of Proposition \ref{key_prop}]
Let $\epsilon$ and $\sigma$ be as in Proposition \ref{final_prop}. Assume that $\epsilon(g,r)<\epsilon$. Consider the map $P_\sigma:\cB(R)\to M$ constructed above. Proposition \ref{vol_ineq} and the inequality in Proposition \ref{final_prop} suggests that $P_\sigma\circ f$ does not increase volume for any Riemannian manifold $N$ and any $1$-Lipschitz map $f:N\to\cB(R)$. In the case of equality, $P_\sigma$ has Jacobian equal to $1$ for almost all points in $f(N)$. Therefore by continuity and the above proposition, $h(\phi)=0$ for all $\phi\in f(N)$, hence $f(N)\subset \Phi(M)$. Therefore the map $P_\sigma$ possesses all properties claimed in Proposition \ref{key_prop}. 
\end{proof}
\begin{rmk}
In the proof of Proposition \ref{final_prop}, we mentioned three additional conditions on $g$. Here is a summary of all conditions posed on $g$ before Proposition \ref{final_prop}.
\begin{enumerate}
\item[(1).] $d_g(x,x_0)\leq R/4$ for any $x\in B_{x_0}(R/5)$, introduced at the begining of Section \ref{s2};
\item[(2).] $g$ has negative sectional curvature, introduced in the proof of Corollary \ref{main_thm};
\item[(3).] \eqref{density_est} with $R_0>R_1$, introduced in the remark of Lemma \ref{delta_X_density} and \eqref{density_abs_est};
\item[(4).] \eqref{cond_on_g_2} with $\widetilde R_0>R_1$, introduced in Lemma \ref{well_defined_P};
\item[(5).] $\mathrm{diam}_g(B_{x_0}(2R))<5R$, introduced in Lemma \ref{well_defined_P};
\item[(6).] $\epsilon=\epsilon(g,r)$ should be small enough such that we can apply the Gram-Schmidt process mentioned in Section \ref{s5};
\item[(7).] $\epsilon=\epsilon(g,r)$ should be small enough such that $\|\widehat A_{x,s}- A_{x,s}\|_{L^2(S)}\leq c_1\epsilon$ for some $c_1=c_1(n,R)$ and any $x\in B_{x_0}(R_1)$, introduced in the proof of Proposition \ref{A_err};
\item[(8).] $\epsilon=\epsilon(g,r)=\|g-g_0\|_{C^r}<\min\{[2C_0(R+R_1)]^{-1},[2C_0(R+R_1)^2]^{-1},1\}$, introduced in Corollary \ref{A_invert} and Lemma \ref{preparation}.
\end{enumerate}
Therefore Theorem \ref{main_thm_alt} and Corollary \ref{main_thm} hold for any $g$ with $\epsilon$ small enough such that $g$ satisfies all 11 conditions. In particular, there exists some $g\neq g_0$ such that $(D,g)$ is a strict minimal filling and boundary rigid.
\end{rmk}

\appendix

\section{The Cayley hyperbolic space}\label{appB}

The set of octonions $\OO$ is an $8$-dimensional non-associative, non-commutative division algebra over $\RR$. Let $\langle\cdot,\cdot\rangle$ be the Euclidean inner product on $\OO$ and $|\cdot|$ the induced norm. Then we have the following properties. (See \cite[1. Composition algebra]{Springer00})
\begin{enumerate}
\item[(1).] $|ab|=|a||b|$ for any $a,b\in\OO$;
\item[(2).] $\langle ab,ac\rangle=\langle ba,ca\rangle=|a|^2\langle b,c\rangle$ for any $a,b,c\in \OO$;
\item[(3).] $\langle ac,bd\rangle+\langle ad,bc\rangle=2\langle a,b\rangle \langle c,d\rangle$ for any $a,b,c,d\in\OO$;
\item[(4).] $\overline{a}=2\langle a,1\rangle-a$ for any $a\in \OO$;
\item[(5).] $2\langle a,b\rangle=2\langle \overline{a},\overline{b}\rangle=\overline{a}b+\overline{b}a=a\overline{b}+b\overline{a}$ for any $a,b\in\OO$;
\item[(6).] $(ba)\overline{a}=\overline{a}(ab)=|a|^2b$ for any $a,b\in \OO$;
\item[(7).] $a(\overline{b}c)+b(\overline{a}c)=(c\overline{a})b+(c\overline{b})a=\langle a,b\rangle c$ for any $a,b,c\in\OO$;
\item[(8).] (Moufang Identities)
\begin{enumerate}
\item[(i).] $(ab)(ca)=a((bc)a)$ for any $a,b,c\in\OO$;
\item[(ii).] $a(b(ac))=(a(ba))c$ for any $a,b,c\in\OO$;
\item[(iii).] $b(a(ca))=((ba)c)a$ for any $a,b,c\in\OO$;
\end{enumerate}
\item[(9).] Multiplications involving only two octonions are associative.
\end{enumerate}
Let 
$$I_{1,2}=\matiii{1&0&0}{0&-1&0}{0&0&-1}$$
and 
$$\FFJ(1,2,\OO)=\left\{\FFX\in\mathrm{Mat}_{3\times 3}(\OO):I_{1,2}\FFX^*I_{1,2}=\FFX\right\}.$$
Any element $\FFX$ can be written in the following form
$$\FFX(\theta,a)=\matiii{\theta_1&a_3&\overline{a}_2}{-\overline{a}_3&-\theta_2&-a_1}{-a_2& -\overline{a}_1&-\theta_3},\quad \theta_j\in\RR, a_j\in\OO, j=1,2,3,$$
where $\theta=(\theta_1,\theta_2,\theta_3)$ and $x=(x_1,x_2,x_3)$.
\begin{dfn}[``Matrix model'']\label{matrix_model_OH^2}
We define the Cayley hyperbolic space $\OO\mathbf{H}^2$ as
$$\OO\mathbf{H}^2=\left\{\FFX\in\FFJ(1,2,\OO):\FFX^2=\FFX,\tr(\FFX)=1,\FFX_{11}>0 \right\}.$$
\end{dfn}
\begin{prop}\label{model_equiv_OH^2}
For any trace $1$ idempotent $\FFX\in\FFJ(1,2,\OO)$ with $\FFX_{11}\neq 0$, there exists a unique vector $(\theta,a,b)\in\RR_+\times\OO^2$ such that 
$$\FFX=\mathrm{sgn}(\FFX_{11})I_{1,2}(\theta,b,c)^*(\theta,b,c),$$
where $\mathrm{sgn}(t)=t/|t|$ when $t\neq 0$. The set
$$\FFJ_{1,0}:=\{\FFX\in\FFJ(1,2,\OO):\FFX^2=\FFX,\tr(\FFX)=1,\FFX_{11}=0\}$$
is isomorphic to $\OO\mathbf{P}^1\isom S^8$.
\end{prop}
\begin{proof}
Write
$$\FFX=\matiii{\theta_1&a_3&\overline{a}_2}{-\overline{a}_3&-\theta_2&-a_1}{-a_2& -\overline{a}_1&-\theta_3}.$$
It suffices to show that 
\begin{align}\label{matrix_vector_equiv}
\matiii{\theta_1&a_3&\overline{a}_2}{\overline{a}_3&\theta_2&a_1}{a_2&\overline{a}_1&\theta_3}=\mathrm{sgn}(\theta_1)(\theta,b,c)^*(\theta,b,c)
\end{align}
for some $(\theta,b,c)\in\RR\times\OO^2$.
Notice that 
$$\FFX^2=\matiii{\theta_1^2-|a_3|^2-|a_2|^2&\theta_1a_3-\theta_2a_3-\overline{a_1a_2} &\theta_1\overline{a}_2-a_3a_1-\theta_3\overline{a}_2}{-\theta_1\overline{a}_3+\theta_2\overline{a}_3+a_1a_2&\theta_2^2-|a_3|^2+|a_1|^2& -\overline{a_2a_3}+\theta_2a_1+\theta_3a_1}{-\theta_1a_2+\overline{a_3a_1}+\theta_3a_2 & -a_2a_3+\theta_2\overline{a}_1+\theta_3\overline{a}_1&\theta_3^2-|a_2|^2+|a_1|^2 }.$$
The condition that $\FFX^2=\FFX$ implies that 
$$a_ja_{j+1}=m_{j+2}\overline{a}_{j+2},\quad m_j\in\RR,j\in\ZZ\mathrm{~mod~}3.$$

\textbf{Case 1:} If $a_3=a_2=0$, then $\theta_1^2=\theta_1$ implies that $\theta_1=1$ or $0$. When $\theta_1=1$, $\tr\FFX=\theta_1^2+\theta_2^2+\theta_3^2+2|a_1|^2=1$ implies $\theta_2=\theta_3=a_1=0$. Hence $\FFX=I_{1,2}(1,0,0)^*(1,0,0)$. When $\theta_1=0$, we have a natural diffeomorphism
$$\FFJ_{1,0}\to\{\FFX\in\mathrm{Mat}_{2\times 2}(\OO):\FFX^2=\FFX, \tr \FFX=1, \FFX^*=\FFX\}$$
by forgetting the first row and the first column, where the latter one is $\OO\mathbf{P}^1\isom S^8$ by \cite[3. Octonionic projective geometry]{Baez}.

\textbf{Case 2:} If $a_3=0$ and $a_2\neq 0$ (the case when $a_2=0$ and $a_3\neq 0$ is similar), then $\FFX^2=\FFX$ implies the following.
\begin{enumerate}
\item[(1).] $\theta_1^2-|a_2|^2=\theta_1\neq 0$;
\item[(2).] $-\theta_1+\theta_3=-1$;
\item[(3).] $\theta_2^2+|a_1|^2=-\theta_2$;
\item[(4).] $\theta_3^2-|a_2|^2+|a_1|^2=-\theta_3$.
\end{enumerate}
Notice that $1=\theta_1-\theta_3$, $\tr\FFX=1$ implies that $\theta_2=a_1=0$ and that $|a_2|^2=\theta_1\theta_3$. Hence we can choose $\theta=\sqrt{|\theta_1|}$, $b=0$ and $c=\mathrm{sgn}(\theta_1)\overline{a}_2/\theta$ and \eqref{matrix_vector_equiv} holds.

\textbf{Case 3:} If $a_2,a_3$ both not equal to $0$, then 
$$|a_j|^2=m_{j+1}m_{j+2},\quad j=1,2,3\mathrm{~mod~} 3.$$
Therefore $\FFX^2=\FFX$ and $\tr\FFX=1$ imply the following.
\begin{enumerate}
\item[(1).] $\theta_1-\theta_2-m_3=1$;
\item[(2).] $\theta_1-\theta_3-m_2=1$;
\item[(3).] $m_1-\theta_2-\theta_3=1$;
\item[(4).] $\theta_1-\theta_2-\theta_3=1$.
\end{enumerate}
Hence $m_j=\theta_j$, where $j=1,2,3$. Choose $\theta=\sqrt{|\theta_1|}$, $b=\mathrm{sgn}(\theta_1)a_3/\theta$ and $c=\mathrm{sgn}(\theta_1)\overline{a}_2/\theta$ and \eqref{matrix_vector_equiv} holds. Uniqueness of the vector is trivial.
\end{proof}
\begin{rmk}
It is easy to see that the condition ``$\tr \FFX=1$'' and the condition ``all $2\times 2$ subdeterminants of $\FFX$ vanish'' introduced in \cite[\S19. Spaces of $R$-rank-1, pp. 137]{Mostow73} are equivalent in this setting.
\end{rmk}
Therefore we have the following alternative definition for the Cayley hyperbolic space (also see \cite[pp. 87]{Parker}).
\begin{dfn}[``Vector model'']\label{vector_model_OH^2}
The Cayley hyperbolic space can be alternatively defined as $$\OO\mathbf{H}^2=\{(\theta,a,b)\in\RR_{+}\times\OO^2: \theta^2-|a|^2-|b|^2=1\}.$$
\end{dfn}
\begin{rmk}
It follows from simple computations that for any $\FFX\in\OO\mathbf{H}^2$, $\FFX_{11}\geq 1$. 

The descriptions here coincide with \cite[\S19. Spaces of $\RR$-rank 1]{Mostow73} in the following way. For any $v=(\theta,b,c)\in\RR_{\geq 0}\times\OO^2$, define $\FFX_v=I_{1,2}v^*v\in\FFJ(1,2,\OO)$. Let $\FFJ_0=\{\FFX\in\FFJ(1,2,\OO):\FFX^2=0,\FFX\neq 0\}$ denotes the collection of all \emph{nilpotents}. It follows from a similar argument as in the above proof that all elements in $\FFJ_0$ can be written as $\pm\FFX_v$ for some $v\in\RR_{\geq 0}\times\OO^2$ satisfying $vI_{1,2}v^*=0$. Define
$$\FFJ_1=\{\FFX\in\FFJ(1,2,\OO):\FFX^2=\FFX,\tr\FFX=1\}.$$
Then for any $\FFX\in\FFJ_1$, the above proposition implies that either $\FFX=\FFX_v$ for some $v\in\RR_{\geq 0}\times\OO^2$ satisfying $vI_{1,2}v^*=1$ (when $\FFX_{11}\geq 1$) or $\FFX=-\FFX_w$ for some $w\in\RR_{\geq 0}\times\OO^2$ satisfying $wI_{1,2}w^*=-1$ (when $\FFX_{11}\leq 0$). (Given the remark of Proposition \ref{model_equiv_OH^2}, one can check that the set $\FFJ_0\union\FFJ_1$ is naturally identified with the Cayley projective plane $\OO\mathbf{P}^2$ described in \cite[\S19. Spaces of $\RR$-rank 1, pp. 137]{Mostow73}.)

Following the definition in \cite{Mostow73}, a point $\FFX\in\FFJ_1$ is called an \emph{inner point} if for any $\FFY\in\FFJ_0\cup\FFJ_1$ such that $\tr\FFX\circ\FFY= 0$, there exists some $\FFX_0\in\FFJ_0$ satisfying $\tr\FFX_0\circ\FFY=0$, where $\FFX\circ\FFY$ is the Jordan multiplication defined as $(\FFX\FFY+\FFY\FFX)/2$. Otherwise $\FFX$ is an \emph{outer point}. In \cite[\S19. Spaces of $\RR$-rank 1, pp. 138]{Mostow73}, the Cayley hyperbolic space is defined to be the collection of all inner points in $\FFJ_1$. Hence, the equivalence of our model for the Cayley hyperbolic space and the model in \cite[\S19. Spaces of $\RR$-rank 1]{Mostow73} boil down to the following proposition.

\begin{prop}
The ``matrix model'' for $\OO\mathbf{H}^2$ defined in Definition \ref{matrix_model_OH^2} is the collection of all inner points in $\FFJ_1$.
\end{prop}
\begin{proof}
For any $0\neq v=(1,a,b),w=(1,c,d)\in\RR_{\geq 0}\times\OO^2$, we have
\begin{align*}
\tr\FFX_v\circ\FFX_w=1-2\langle a,c\rangle-2\langle b,d\rangle+|ac|^2+|bd|^2+2\langle \overline{a}b,\overline{c}d\rangle.
\end{align*}
When $a\neq 0$, the first Moufang identity implies that
\begin{align*}
\langle \overline{a}b,\overline{c}d\rangle=\langle b,a(\overline{c}d)\rangle=\frac{1}{|a|^2}\langle ba,(a(\overline{c}d))a\rangle=\frac{1}{|a|^2}\langle ba,(a\overline{c})(da)\rangle=\frac{1}{|a|^2}\langle c\overline{a},(da)\overline{(ba)}\rangle.
\end{align*}
Therefore
\begin{align}\label{inner_prod_sq}\tr\FFX_v\circ\FFX_w=
\begin{cases}
\displaystyle \left|1-c\overline{a}-\frac{(da)\overline{(ba)}}{|a|^2}\right|^2,~~& a\neq 0,\\
\displaystyle |1-d\overline{b}|^2, & a=0.
\end{cases}
\end{align}

\textbf{Case 1:} When $a=b=0$, the quantity in \eqref{inner_prod_sq} never equal to 0. The only $w\in\RR_{\geq 0}\times\OO^2$ such that $\tr\FFX_v\circ\FFX_w=0$ is in the form of $w=(0,c,d)$ with $c,d\in\OO$. Let $u=(1,d/|(c,d)|,-c/|(c,d)|)$. One can verify that $\FFX_u\in\FFJ_0$ and $\tr\FFX_u\circ\FFX_w=0$. Therefore $\FFX_v$ is an inner point. Denote by $x_0$ this very special inner point, i.e.,
 $$x_0=\matiii{1&0&0}{0&0&0}{0&0&0}.$$
Simple computation shows that for any $u\in\RR_{\geq 0}\times\OO^2$, $\FFX_u\circ x_0=0$ implies that $u\in\{0\} \times\OO^2$. Let $[\FFX_u]$ be the unique element in $\RR\FFX_u\ints \FFJ_1$ for those $\FFX_u\not\in\FFJ_0$.  Therefore $[\FFX_u]$ is an outer point for any $u\in\{0\}\times\OO^2\setminus \{0\}$.

\textbf{Case 2:} When $(a,b)\neq0$, without loss of generality we can assume that $a\neq 0$. If $|a|^2+|b|^2<1$, $\tr\FFX_v\circ\FFX_w=0$ implies that $|c|^2+|d|^2\geq 1$, following \eqref{inner_prod_sq}. If $c=0$ (or similarly $d=0$), it is easy to find a $u=(1,\lambda,\xi d)$ for some real numbers $\lambda,\xi$ such that $\FFX_u\in\FFJ_0$ and $\tr \FFX_u\circ\FFX_w=0$. If $c,d\neq 0$, let $\lambda,\xi\in\RR$ and $u=(1,\lambda c,\xi d)$. Then the set of equations with respect to $\lambda$ and $\xi$ given by
\begin{align}\label{quad_eq_set}
\begin{cases}
\displaystyle  \FFX_u\in\FFJ_0;\\
\displaystyle  \tr\FFX_u\circ\FFX_w=0
\end{cases}
\Longleftrightarrow
\begin{cases}
\displaystyle  1=\lambda^2|c|^2+\xi^2|d|^2;\\
\displaystyle  1=\lambda|c|^2+\xi|d|^2
\end{cases}
\end{align}
has real solutions
\begin{align*}
\begin{cases}
\displaystyle  \lambda=\frac{1\pm\sqrt{\frac{|d|^2}{|c|^2}(|c|^2+|d|^2-1)}}{|c|^2+|d|^2};\\
\displaystyle  \xi=\frac{1\mp\sqrt{\frac{|c|^2}{|d|^2}(|c|^2+|d|^2-1)}}{|c|^2+|d|^2}
\end{cases}
\end{align*}
if and only if $|c|^2+|d|^2\geq 1$. Therefore we can conclude that $[\FFX_v]$ is an inner point when $vI_{1,2}v^*>0$. 

On the other hand, we define $\widehat{v}=(1,a/(|a|^2+|b|^2),b/(|a|^2+|b|^2))$ for any $v=(1,a,b)\in\RR_{\geq 0}\times\OO^2$ such that $vI_{1,2}v^*<0$, . One can verify that $\tr\FFX_v\circ\FFX_{\widehat{v}}=0$ and that $\widehat{v}I_{1,2}\widehat{v}^*>0$. Therefore $\tr\FFX\circ\FFX_{\widehat{w}}\neq 0$ for any $\FFX\in\FFJ_0$ by \eqref{quad_eq_set}, which implies that $[\FFX_v]$ is an outer point when $vI_{1,2}v^*<0$. 

Hence we can conclude from the above discussions that $\FFX\in\FFJ_1$ is an inner point if $\FFX_{11}\geq 1$ and is an outer point if $\FFX_{11}\leq 0$, which proves that our definition of $\OO\mathbf{H}^2$ coincide with the definition in \cite[\S19. Spaces of $\RR$-rank 1]{Mostow73}.
\end{proof}
We refer to \cite[\S19. Spaces of $\RR$-rank 1]{Mostow73} and \cite{Baez} for an overview of related subjects. A very detailed and general theory can be found in \cite{Springer1}, \cite{Springer00} by T. A. Springer and F. D. Veldkemp for further reference.
\end{rmk}

Denote by $g_0$ the symmetric metric on $\OO\mathbf{H}^2$ and $M=(\OO\mathbf{H}^2,g_0)$ such that the distance function on $\OO\mathbf{H}^2$ is given by
$$\cosh(2d(\FFX,\FFY))=2\tr(\FFX\circ\FFY)-1,\quad\forall \FFX,\FFY\in\OO\mathbf{H}^2$$
as in \cite{Mostow73}. We identify $T_{x_0}M_0$ with $\OO^2$ such that each unit vector $v=(a,b)\in\OO^2$ corresponds to the initial vector of $\gamma_v(t):=(\cosh(t),a\sinh(t),b\sinh(t))$. 
We will compute the Riemannian curvature data at $x_0$ by understanding the \textit{geodesic hinge} $\angle\FFX x_0\FFY$ for any $\FFX,\FFY\in M\setminus \{x_0\}$, where $\angle\FFX x_0\FFY$ consists of two geodesic segments $\overline{\FFX x_0}$, $\overline{\FFY x_0}$ and the angle $\measuredangle\FFX x_0\FFY$. A \textit{comparison hinge} of $\angle\FFX x_0\FFY$ in some space form $M'$ is a geodesic hinge $\angle\FFX 'x_0'\FFY'$ in $M'$ with the same angle such that the lengths of geodesic segments $\overline{\FFX x_0}$, $\overline{\FFY x_0}$ and $\overline{\FFX' x_0'}$, $\overline{\FFY' x_0'}$ are equal respectively.

\begin{prop}
The following hold for the Cayley hyperbolic space.
\begin{enumerate}\label{curv_OH^2}
\item[(1).] For any unit vector $v=(a,b)\in\OO^2$, $\gamma_v(t)$ gives a unit speed geodesic starting at $x_0$ with initial vector $v$; 
\item[(2).] The Riemannian metric at $x_0$ is given by the Euclidean inner product on $\OO^2$. Hence the map $\chi:\OO^2\to M$ such that
$$\chi(a,b)=\left(\cosh(|(a,b)|),\frac{a}{|(a,b)|}\sinh(|(a,b)|),\frac{b}{|(a,b)|}\sinh(|(a,b)|)\right)$$
gives the geodesic normal coordinates centered at $x_0$ (hence $d\chi:\OO^2\to T_{x_0}M_0$ is the isometric correspondence from $\OO^2$ with Euclidean inner product to $T_{x_0}M_0$ mentioned above)
\item[(3).] For any $v=(a,b)\in\OO^2$, denote by 
\begin{align*}
\Cay(v)=\begin{cases}
\displaystyle \OO\cdot(1,a^{-1}b),\quad& a\neq0; \\
\displaystyle \OO\cdot(0,1),&a=0
\end{cases}
\end{align*}
the \emph{Cayley line} containing $v$. Then for any non-parallel pair of non-zero vectors $(a,b),(c,d)$ contained in the same Cayley line and any pair of points $\FFX\in\gamma_{(a,b)}(\RR_+)$ and $\FFY\in\gamma_{(c,d)}(\RR_+)$, the comparison hinge $\angle \FFX' x_0'\FFY'$ of $\angle \FFX x_0\FFY$ in a space form of constant sectional curvature $-4$ satisfies $d(\FFX',\FFY')=d(\FFX,\FFY)$;
\item[(4).] For any $a,b\in\OO\setminus\{0\}$, any pair of points $\FFX\in\gamma_{(a,0)}(\RR_+)$ and $\FFY\in\gamma_{(0,b)}(\RR_+)$, the comparison hinge $\angle \FFX' x_0'\FFY'$ of $\angle \FFX x_0\FFY$ in a space form of constant sectional curvature $-1$ satisfies $d(\FFX',\FFY')=d(\FFX,\FFY)$.
\end{enumerate}
\end{prop}
\begin{proof}
\begin{enumerate}
\item[(1).] This follows easily by direct computations.
\item[(2).] Let $(a,b),(c,d)$ be unit vectors in $\OO^2$. Hence the inner product of these two vectors is given by
\begin{align*}
-\left.\frac{d}{dt}\right|_{t=0}d(\gamma_{(a,b)}(t),\gamma_{(c,d)}(1))=&-\frac{\left.\frac{d}{dt}\right|_{t=0}\left(2\tr\gamma_{(a,b)}(t)\circ\gamma_{(c,d)}(1)-1\right)}{2\sinh(2)} \\
=&\frac{2\sinh(1)\cosh(1)(\langle a,c\rangle+\langle b,d\rangle)}{2\sinh(2)} 
=\langle a,c\rangle+\langle b,d\rangle.
\end{align*}
\item[(3).] Let $a,b,c$ be unit octonions and $\theta\in[0,\pi/2)$. Let $v=(\cos\theta,a\sin\theta)$. For any $t_1,t_2>0$, we have
\begin{align*}
&\cosh(2d(\gamma_{bv}(t_1),\gamma_{cv}(t_2)))\\
=&2\tr\gamma_{bv}(t_1)\circ\gamma_{cv}(t_2)-1 \\
=&2\cosh^2(t_1)\cosh^2(t_2)-4\cosh(t_1)\sinh(t_1)\cosh(t_2)\sinh(t_2)\langle bv,cv\rangle -1\\
&+2\sinh^2(t_1)\sinh^2(t_2)(\cos^4\theta+\sin^4\theta)+4\sinh^2(t_1)\sinh^2(t_2)\cos^2\theta\sin^2\theta \\
=&2\cosh^2(t_1)\cosh^2(t_2)+2\sinh^2(t_1)\sinh^2(t_2)-\sinh(2t_1)\sinh(2t_1)\langle bv,cv\rangle -1 \\
=&\cosh(2t_1)\cosh(2t_2)-\sinh(2t_1)\sinh(2t_1)\langle bv,cv\rangle.
\end{align*}
Notice that $t_1=d(\gamma_{bv}(t_1),E_1)$ and $t_2=d(\gamma_{cv}(t_2),E_1)$, the above equation coincides with the law of cosine in a space form with constant sectional curvature $-4$.
\item[(4).] Let $a,b$ be unit octonions and $t_1,t_2>0$. Write $v=(a,0)$ and $w=(0,b)$. Then
\begin{align*}
\cosh(2d(\gamma_{v}(t_1),\gamma_{w}(t_2)))=2\cosh^2(t_1)\cosh^2(t_2)-1,
\end{align*}
which implies that 
$$\cosh(d(\gamma_{v}(t_1),\gamma_{w}(t_2)))=\cosh(t_1)\cosh(t_2).$$
The above equation coincides with the law of cosine in a space form with constant sectional curvature $-1$.\qedhere
\end{enumerate}
\end{proof}
A direct corollary of the above proposition is the following.
\begin{cor}\label{curv_data_OH^2}
For any non-zero $v,w\in \OO^2$ the following hold.
\begin{enumerate}
\item[(1).] If $v,w$ belong to the same Cayley line and $v\not\in\RR w$, then the sectional curvature of the $2$-dimensional plane spanned by $d\chi(v),d\chi(w)$ is $-4$;
\item[(2).] If $\Cay(v)\perp\Cay(w)$, then the sectional curvature of the $2$-dimensional plane spanned by $d\chi(v),d\chi(w)$ is $-1$.
\item[(3).] Recall that from classical results that $\OO\mathbf{H}^2=F_4^{-20}/\Spin(9)$ with $\Spin(9)$ the stabilizer of $x_0$. Since $F_4^{-20}$ acts by isometries on $M=(\OO\mathbf{H}^2,g_0)$, identifying $T_{x_0}M_0$ with $\OO^2$, one can view $\Spin(9)$ as a subgroup of $\SO(\OO^2)\isom\SO(16)$. In particular, $\Spin(9)$ maps Cayley lines to Cayley lines and acts transitively on the set of all Cayley lines.  
\end{enumerate}
\end{cor}
\begin{rmk}
Let $p$ be a fixed point in $\CC\mathbf{H}^n$. Recall that for any unit vector $v\in T^1_p\CC\mathbf{H}^n$ with respect to the symmetric metric, there exists a linear map $J\in\End(T_p\CC\mathbf{H}^n)$ such that the sectional curvature between $v$ and $J(v)$ is $-4$. The same argument holds true in quaternionic hyperbolic spaces but \textbf{NOT} in the Cayley hyperbolic space. This is due to the non-associativity of octonionic multiplication. In fact, if such a map $J$ exists, without loss of generality we can assume that $p=x_0$ and identify the tangent space at $x_0$ with $\OO^2$ via $d\chi$ as in Proposition \ref{curv_OH^2}. If $J(1,0)=(a,0)$ for some unit octonion $a\in\OO\setminus \RR$, then for any unit octonion $b$ and any $\theta\in[0,\pi/2)$, $J(\cos\theta,b\sin\theta)=a(\cos\theta,b\sin\theta)$. In particular, $J(0,1)=(0,a)$. Similarly we have $J(b\sin\theta,\cos\theta)=a(b\sin\theta,\cos\theta)$ for any unit octonion $b$ and any $\theta\in[0,\pi/2)$, which implies that $J(v)=av$ for any $v\in\OO^2$. Therefore, for any octonions $b\neq 0,c$, $J(b,bc)=(ab,a(bc))\in\Cay(b,bc)=\Cay(1,c)$, which implies that $(ab)c=a(bc)$. Notice that if $(ab)c=a(bc)$ for any $b,c\in\OO$, then $a$ must be real. This contradicts the assumption that $a\in\OO\setminus\RR$.
\end{rmk}

\Addresses
\end{document}